\documentclass[review,onefignum,onetabnum]{article}

\usepackage[utf8]{inputenc} 
\usepackage[T1]{fontenc}
\usepackage{geometry}
\usepackage{amsmath,amssymb,amstext}
\usepackage{amsthm}
\usepackage{algorithm}
\usepackage{color}
\usepackage{multirow}
\usepackage{booktabs}
\usepackage{adjustbox}
\usepackage{diagbox}
\usepackage{subcaption}
\usepackage{comment}
\usepackage{hyperref}
\usepackage{tikz}
\usetikzlibrary{spy}

\newtheorem{theorem}{Theorem}[section]
\newtheorem{assumptions}{Assumption}

\newtheorem{lemma}[theorem]{Lemma}

\theoremstyle{definition}
\newtheorem{definition}{Definition}[section]

\theoremstyle{remark}
\newtheorem*{remark}{Remark}
\newtheorem{prop}[theorem]{Proposition}
\newtheorem{cor}[theorem]{Corollary}

\newcommand{\R}{\mathbb{R}}
\newcommand{\bR}{\bar{\mathbb{R}}}
\newcommand{\N}{\mathbb{N}}
\newcommand{\F}{\mathcal{F}}
\newcommand{\ik}{i_k}
\newcommand{\xk}{x_k}
\newcommand{\xkk}{x_{k+1}}
\newcommand{\xkm}{x_{k-1}}
\newcommand{\yk}{\tilde{y}_k}
\newcommand{\hyk}{\hat{y}_k}
\newcommand{\dk}{d_k}
\newcommand{\uk}{u_k}
\newcommand{\sk}{s_k}
\newcommand{\phik}{\phi_{\ik}^{\xk}}

\newcommand{\ak}{\alpha_{k}}
\newcommand{\Dk}{D_{k}}
\newcommand{\bk}{\beta_{k}}
\newcommand{\Uk}{U_{\ik}}
\newcommand{\mk}{m_k}
\newcommand{\xo}{x_0}
\newcommand{\prox}{\mathrm{prox}}
\newcommand{\ff}{f}
\newcommand{\fg}{\phi}
\DeclareMathOperator{\denoiser}{\mathrm{D}}

\usepackage{tikz}
\colorlet{BLUE}{blue}
\title{Block-coordinate Plug-And-Play Methods with Armijo-like line-search for Image Restoration}

\author{F. Porta\thanks{Dipartimento di Scienze Fisiche, Informatiche e Matematiche, Universit\`{a} di Modena e Reggio Emilia, Via Campi 213/b, 41125 Modena, Italy } \and S. Rebegoldi\footnotemark[1]
\and A. Sebastiani\thanks{MaLGa Center, Dipartimento di Informatica, Bioingegneria, Robotica e Ingegneria dei Sistemi (DIBRIS), Universit\`{a} degli Studi di Genova, Via Dodecaneso 35, 16146 Genova, Italy }
}

\date{}

\begin{document}

\maketitle
\begin{abstract}
   \textcolor{black}{In this paper, we develop a class of block-coordinate forward-backward methods for solving non-convex and non-separable composite optimization problems, and specialize them to imaging inverse problems within the Plug-and-Play (PnP) framework based on Gradient Step (GS) denoisers. The block-coordinate strategy is designed to reduce the high memory consumption associated with the computation of GS denoisers, whose implementation typically requires storing large computational graphs.} The proposed methods allow for the joint use of inertial acceleration, variable metric strategies, inexact proximal computations, and adaptive steplength selection via an appropriate line-search procedure. Under mild assumptions on the objective function, we establish a sublinear convergence rate and the stationarity of the limit points. Moreover, convergence of the entire sequence of the iterates is guaranteed under a Kurdyka-Łojasiewicz assumption. Numerical experiments on ill-posed imaging problems, including deblurring and super-resolution, demonstrate that the  proposed PnP approach achieves state-of-the-art reconstruction quality while substantially reducing GPU memory requirements, making it particularly suitable for large-scale and resource-constrained imaging applications.
\end{abstract}

\section{Introduction}\label{sec:intro}
The main goal of imaging inverse problems is to recover an unknown image $x\in\R^n$ from measurements $b\in\R^m$, which are related through a linear acquisition model $b=Ax+\eta$, where $A \in \R^{m \times n}$ is the forward operator, mapping the image space to the measurement space, and $\eta \in \R^m$ denotes additive measurement noise. Since such inverse problems are typically ill-posed, direct inversion does not provide meaningful solutions.

Variational approaches to image restoration \cite{Bertero-etal-2008,Vogel2002} propose to recover an approximation of the unknown image as the solution of an optimization problem of the form
\begin{equation}\label{eq:problem}
    \underset{x\in\R^n}{\operatorname{argmin}} \ F(x)\equiv \fg(x) + \ff(x),
\end{equation}
where $\fg:\R^n\rightarrow \R$ is a data fidelity term, typically related to the noise distribution and measuring the likelihood of the observed data given the unknown image, and $\ff:\R^n\rightarrow \R$ is a regularization term, enforcing stability and a priori information on the solution. The explicit choice of the second term usually depends on prior knowledge of the image; however, arbitrary assumptions may introduce bias or artifacts, such as the staircasing effect associated with the Total Variation regularization.

Plug-and-Play (PnP) approaches \cite{Venkatakrishnan-etal-2013} represent a powerful and innovative paradigm for solving inverse problems in computational imaging by replacing the usual explicit regularization step with an implicit step, implemented by means of a denoising algorithm. This makes the PnP approach very modular and flexible, since any off-the-shelf denoiser can be plugged into the framework. Instead of designing a hand-crafted regularization term, cutting-edge denoisers can be embedded into the algorithm, often outperforming traditional regularization techniques. Indeed, in the last decade, PnP methods have achieved state-of-the-art results in several image restoration tasks \cite{Ahmad-etal-2020,Buzzard-etal-2018,Meinhardt-etal-2017,Yuan-etal-2020,Zhang-et-al-2021}.
\textcolor{black}{From a theoretical viewpoint, their convergence analysis has often relied on assumptions such as strong convexity of the data-fidelity term \cite{Ryu-etal-2019} or specific properties of the denoiser, including near-non-expansiveness or symmetric Jacobians \cite{Pesquet-et-al-2021,Reehorst-etal-2018,Romano-etal-2017,Ryu-etal-2019,Xu-etal-2020}. While these assumptions may be restrictive, the corresponding methods often remain very competitive in practice. At the same time, the analysis of PnP methods has been extended to different noise models and algorithmic frameworks, including approaches based on Bregman geometries \cite{Al-Shabili-etal-2022,Benfenati2025,hurault2023convergent}.} 

\textcolor{black}{
Motivated by the need for PnP schemes with theoretical guarantees under milder or more natural assumptions,
this work focuses specifically on the Gradient Step PnP approach introduced in \cite{hurault2022gradient}, where PnP schemes can be recast as optimization algorithms for well-defined objective functionals.
Unlike classical black-box PnP approaches, in which an off-the-shelf denoiser is embedded into an iterative scheme without necessarily admitting an explicit variational interpretation, this framework defines the denoising operator as the gradient step of a potential function parametrized by a convolutional neural network, thus allowing for a clear connection between PnP methods and variational regularization. }
More precisely, given a noise level $\sigma\in\R^+$ and a smooth function $g_{\sigma}:\R^n\rightarrow\R$, the Gradient Step (GS) denoiser is defined as 
    \begin{equation}\label{eq:D_sigma}
        \denoiser_{\sigma}(x)=x-\nabla g_{\sigma}(x).
    \end{equation}
In practice, the authors in \cite{hurault2022gradient} propose to choose 
\begin{equation}\label{eq:g_sigma}
    g_{\sigma}(x)=\frac{1}{2}\|x-N_{\sigma}(x)\|^2,
\end{equation}
 where $N_{\sigma}:\R^n\rightarrow\R^n$ is any differentiable neural network architecture that has proven to be efficient for image
denoising. In this framework, given a positive regularization parameter $\textcolor{black}{\varrho}$, the  term $f$ in \eqref{eq:problem} takes the form $\textcolor{black}{\varrho} g_\sigma$ and, by following \cite{hurault2022gradient}, the resulting optimization problem is solved by means of a PnP method whose general iteration can be seen as a forward-backward step on \eqref{eq:problem}, i.e.,
\begin{equation}\label{eq:PnP_original}
    x_{k+1} = \prox_{\alpha \phi}\left(x_k- \alpha \nabla f(x_k)\right) = \prox_{\alpha \phi}\left(x_k- \alpha \textcolor{black}{\varrho}\nabla g_{\sigma}(x_k)\right),
\end{equation}
where $\alpha\in\mathbb{R}^+$ is a suitable steplength and $\prox_{\alpha\phi}$ denotes the proximal operator of $\alpha\phi(\cdot)$, namely
\begin{equation}\label{eq:prox_no_D}
\prox_{\alpha\phi}(x) = \underset{y\in\mathbb{R}^n}{\operatorname{argmin}} \ \frac{1}{2}\|y-x\|^2 + \alpha\phi(y), \quad \forall x\in\mathbb{R}^n.
\end{equation}
Remarkably, by leveraging convergence results valid for the standard forward-backward gradient 
method in the non-convex setting, in \cite{hurault2022gradient} the authors deduce that the PnP scheme \eqref{eq:PnP_original} converges to a stationary point of $\phi+\textcolor{black}{\varrho} g_\sigma$, even when the data-fidelity term is not strongly convex and without assuming the non-expansiveness of the denoiser. The main limitation 
of the PnP forward-backward approach in \eqref{eq:PnP_original} is the evaluation of $\nabla g_\sigma$. Indeed, in view of \eqref{eq:D_sigma} and the definition of $g_\sigma$ in \eqref{eq:g_sigma}, the gradient of $g_\sigma$ can be written explicitly as
\begin{equation}\label{eq:denoising_step}
    \nabla g_\sigma(x) = x - \denoiser_{\sigma}(x) =x - N_{\sigma}(x) - J_{N_{\sigma}}(x)^T(x-N_{\sigma}(x)),
\end{equation}
where $J_{N_{\sigma}}(x)$ denotes the Jacobian of $N_{\sigma}$ evaluated at $x$, which can be computed via an automatic differentiation pipeline. As a consequence, the forward step in \eqref{eq:PnP_original} requires the computation of the Jacobian $J_{N_\sigma}(x)$ using a backpropagation routine that 
stores the entire computational graph associated to the neural network. Unfortunately, the memory required for this operation scales with the image size, which significantly limits the applicability of GS denoisers and the corresponding PnP algorithm \eqref{eq:PnP_original} on real applications in resource-constrained settings.

\textit{Contributions.} \textcolor{black}{The contribution of this work is twofold. First, we propose and analyze a block-coordinate forward-backward framework for solving composite optimization problems of the form \eqref{eq:problem}, allowing for possibly non-convex objective functions and non-separable terms. Second, we leverage this framework to define a novel class of block-coordinate PnP methods based on a GS denoiser of the form \eqref{eq:D_sigma}--\eqref{eq:g_sigma}. In this setting, the block-coordinate strategy is designed to efficiently reduce the memory consumption associated with the computation of $\nabla g_\sigma$, which represents the main computational bottleneck of PnP methods based on GS denoisers. In particular, image patches are used as blocks of coordinates, and the properties of convolutional neural networks are exploited to obtain an efficient block-wise computation of $\nabla g_\sigma$.}

\textcolor{black}{More precisely, the general framework is developed under the assumptions that $\fg:\R^n \to \R$ is continuously differentiable with a locally Lipschitz continuous gradient and is convex along the coordinates, while $\ff:\R^n \to \R$ has a Lipschitz continuous gradient. No separability assumption is imposed on either term. The minimization problem associated with the PnP paradigm based on GS denoisers fits into this framework by taking $\phi$ as a possibly non-separable data-fidelity term and by choosing $f=\textcolor{black}{\varrho} g_\sigma$, with $\textcolor{black}{\varrho}\in\R^+$ and $g_\sigma$ defined in \eqref{eq:g_sigma}. Indeed, classical smooth data-fidelity terms satisfy the assumptions required on $\phi$, while $g_\sigma$ is typically non-separable but has a Lipschitz continuous gradient whenever the neural network $N_\sigma$ in \eqref{eq:g_sigma} is given by a composition of continuously differentiable functions with bounded and Lipschitz continuous derivatives. Such regularity requirement on $N_\sigma$ is compatible with the differentiable convolutional architecture used in the GS denoiser framework of \cite{hurault2022gradient}, and is not meant to cover arbitrary black-box PnP denoisers.
}
\textcolor{black}{Finally, we emphasize that the restriction to GS denoisers plays a key role in the proposed PnP analysis: it provides an explicit variational structure, so that the corresponding scheme can be studied as an optimization algorithm for the objective function $\phi+\varrho g_\sigma$, while still allowing for competitive practical performance with respect to PnP approaches based on black-box denoisers, as shown in \cite{hurault2022gradient}.}

The proposed block-coordinate forward-backward framework can be considered as a generalization of the Proximal Heavy-ball Inexact Line-search Algorithm (PHILA) devised in \cite{Bonettini-Prato-Rebegoldi-2024} to the alternating block-coordinate setting. The PHILA method applied to problem \eqref{eq:problem} relies on the following iteration
\begin{equation}\label{eq:PHILA}
    \begin{aligned}
        \tilde{y}_k &\approx \prox_{\alpha_k \phi}^{D_k}\left(x_k -\alpha_k D_k^{-1}\nabla f(x_k)+\beta_k(x_k-x_{k-1})\right)\\
        x_{k+1} &=x_k+\lambda_k(\tilde{y}_k-x_k),
    \end{aligned}
\end{equation}
where $\alpha_k$ and $\beta_k$ are positive parameters, $D_k$ is a symmetric and positive definite matrix, $\operatorname{prox}_{\ak \phi}^{\Dk}(\cdot)$ is the proximal operator of $\ak \phi$ with respect to the norm induced by $\Dk$, and $\lambda_k\in(0,1]$ is selected by an Armijo-like line-search procedure, to guarantee a sufficient decrease of a proper merit function. The 
inexact computation of the proximal operator related to the data fidelity term $\phi$ is also allowed. Extending PHILA to a block-coordinate framework requires a careful design of the inertial terms across blocks, together with the definition of a suitable merit function decreasing descent 
throughout the iterations. Under the previously stated assumptions on the functions $\phi$ and $f$, we prove a sublinear convergence rate and the stationarity of the limit points. Furthermore, if appropriate surrogate functions of $F$ satisfy the Kurdyka-\L{}ojasiewicz property on their domains, convergence of the entire sequence generated by the proposed framework and related convergence rates on the function values can be established. Numerical experiments on ill-posed imaging problems, such as deblurring and super-resolution, show that the block-coordinate PHILA approach achieves state-of-the-art image restoration performance while significantly reducing GPU memory usage. This property makes the proposed framework particularly suitable for resource-constrained environments and large-scale imaging problems.

The reader may wonder why a forward-backward approach is adopted to solve \eqref{eq:problem}, rather than a full gradient scheme, even though both terms $\phi$ and $f$ are assumed to be differentiable. This choice can be justified by several considerations.
\textit{(i)} As shown in \cite{Combettes-etal-2019}, activating smooth functions through their gradients and nonsmooth ones through their proximity operators is not necessarily the most efficient strategy from a numerical standpoint, since proximal steps may positively affect the asymptotic performance of algorithms, particularly in image restoration problems.
\textit{(ii)} The Lipschitz constant of $\nabla \phi + \nabla f$ may be significantly larger than that of $\nabla f$ \cite[Section~9.3]{Beck-2017}, which leads to smaller feasible steplengths  for the full gradient method, due to the inverse dependence on the Lipschitz constant of $\nabla \phi + \nabla f$.
\textit{(iii)} Finally, the approach proposed in \cite{hurault2022gradient}, which inspired this work and relies on the forward-backward iteration \eqref{eq:PnP_original}, has demonstrated strong empirical performance across various imaging tasks. Our numerical experiments further confirm that a forward-backward strategy seems preferable to a full gradient approach when solving imaging problems with PnP schemes. \textcolor{black}{Although our numerical assessment focuses on imaging inverse problems, the above motivations may also be relevant in other contexts. For instance, forward-backward algorithms have been used as efficient alternatives to purely gradient/subgradient-based methods for minimizing regularized objective functions in machine learning \cite{Duchi-etal-2009} and deep learning \cite{Yang-etal-2022,Yang-etal-2022b}.}

\textcolor{black}{Our proposed framework can be applied to problem \eqref{eq:problem} even when none of the terms $f$ and $\phi$ are convex along the coordinates. For instance, this may arise in the Gradient Step PnP approach \cite{hurault2022gradient} when noise models other than the Gaussian one are considered, see e.g. Cauchy noise \cite{Parenti-etal-2026,Sciacchitano-et-al-2015}. In such cases, only a full gradient approach is admissible; still, the method can be successfully applied, as our numerical experiments show.}

\textit{Related works.} We remark that several block-coordinate forward-backward methods for possibly non-convex optimization have already been proposed in the literature, often incorporating either inertial acceleration or variable metric strategies \cite{Bonettini-Prato-Rebegoldi-2018,Chouzenoux-Pesquet-Repetti-2016,Frankel-Garrigos-Peypouquet-2015,Gur-Sabach-shtern-2023,Ochs-2019,Pock-Sabach-2018,Rebegoldi-2024,Xu-Yin-2013,Xu-Yin-2017}. However, these approaches typically do not combine both mechanisms and are tailored to problems of the form \eqref{eq:problem} with separable (possibly non-smooth) regularization 
terms, which limits their applicability to settings involving differentiable but non-separable objectives, such as PnP frameworks. While a few methods address non-separable composite problems  \cite{Aberdam-etal-2021,chorobura2023random,Grishchenko-etal-2021,Hanzely-etal-2018,Latafat-etal-2022,Necoara-2025}, they generally rely on restrictive steplength rules, lack line-search strategies, do not allow inexact proximal computations, and do not exploit inertial or variable metric acceleration. In contrast, the approach proposed in this paper is specifically designed for non-convex and non-separable problems, offering provable convergence guarantees together with adaptive step-size selection and the joint use of inertial and variable metric strategies.

To the best of our knowledge, only a very limited number of works on block-coordinate PnP algorithms exist in the literature \cite{Gan-etal-2023,Huang-etal-2025,Sun-etal-2020}. \textcolor{black}{The approaches in \cite{Gan-etal-2023,Huang-etal-2025}} belong to the forward-backward class but differ from our proposal in that they consider separable regularization terms, possibly based on distinct learned denoisers acting on different variables. In those works, the notion of block of coordinates is different, as the blocks correspond to distinct variables, e.g., \textcolor{black}{the kernel and the image in blind deconvolution \cite{Gan-etal-2023}} or \textcolor{black}{the dictionary and the sparse coefficients in dictionary learning \cite{Huang-etal-2025}}. In contrast, \textcolor{black}{and similarly to our work}, the PnP setting considered \textcolor{black}{in \cite{Sun-etal-2020}} 
involves blocks corresponding to different patches of the same image; \textcolor{black}{however, it additionally requires the block-nonexpansiveness of the denoiser, a property that can be ensured heuristically, e.g. via spectral normalization. }\textcolor{black}{We mention that a similar image patch-splitting strategy is also adopted in \cite{Zhang-et-al-2021} only for implementation purposes, to apply CNN denoisers to large images in a memory-efficient way, without formulating the corresponding PnP problem or algorithm in a block-coordinate form.} Finally, none of the methods in \cite{Gan-etal-2023,Huang-etal-2025,Sun-etal-2020,Zhang-et-al-2021} employs line-search strategies for steplength selection; instead, the steplength must be bounded above by a constant related to the Lipschitz constant of the gradient of the fidelity term, which may not be known in practice.

\textit{Contents.} The paper is organized as follows. Section~\ref{sec:preliminaries} introduces the notations and recalls basic definitions and results on proximal operators. The proposed algorithm is presented in Section~\ref{sec:method}, together with its application to the PnP framework and an in-depth discussion of related works. The theoretical convergence analysis is provided in Section~\ref{subsec:convergence_analysis}. Numerical experiments are reported in Section~\ref{sec:numerical_experiments}. Finally, Section~\ref{sec:conclusions} is devoted to concluding remarks and future perspectives.

\section{Preliminaries}\label{sec:preliminaries}
The following notations will be used throughout the paper. The symbol 
$\|x\|=\sqrt{\langle x,x\rangle}$ stands for the standard Euclidean norm on $\R^n$, where $\langle\cdot,\cdot\rangle$ is the standard Euclidean inner product. The set of real matrices of order $n$ is denoted by $\R^{n\times n}$, and the identity matrix of order $n$ is indicated by $I_n$. The symbol $\mathcal{S}(n)$ stands for the set of all symmetric matrices of order $n$, whereas $\mathcal{S}_{++}(n)$ denotes the set of real, symmetric and positive definite matrices of order $n$. Given $D\in\mathcal{S}_{++}(n)$, the norm induced by $D$ is defined as $\|x\|_D = \sqrt{\langle x, Dx\rangle}$. Given $D_1,D_2\in\mathcal{S}(n)$, we say that $D_1\preceq D_2$ if and only if $\langle D_1x,x\rangle \leq \langle D_2x,x\rangle$ for all $x\in\R^n$.

The following definition is required to efficiently manage blocks of coordinates in the analysis of the proposed framework.

\begin{definition}\label{def:U}
    Let $U\in\R^{n\times n}$ be a column permutation of the identity matrix and let $U=[U_1,\ldots,U_N]$ be a decomposition of $U$ in $N$ submatrices, where each $U_i\in\R^{n\times n_i}$ is defined with $n_i=|I_i|$, being $I_i\subset\{1,\ldots,n\}$ the set of active indexes, and $\sum_{i=1}^N n_i=n$. 
\end{definition}
\begin{remark}
    Any $x\in\R^n$ can be written as $x=\sum_{i=1}^N U_i x^{(i)}$, where $x^{(i)}=U_i^Tx\in\R^{n_i}$ is the $i-$th block of coordinates of $x$, $i=1,\ldots,N$.
\end{remark}

The concept of proximal operator of a function with respect to the norm induced by a real, symmetric, positive definite matrix, is crucial in the proposed framework. Below we report its definition.

\begin{definition}\label{def:proximal_operator}\cite[Section 2.3]{Frankel-Garrigos-Peypouquet-2015}
Let $\phi:\R^n\rightarrow \R \cup \{\infty\}$ be a proper, lower semicontinuous function, $\alpha>0$ and $D\in\mathcal{S}_{++}(n)$. The proximal operator of $\alpha\phi$ with respect to the norm induced by $D$ is the map $\prox_{\alpha \phi}^D:\R^n\rightrightarrows \R^n$ defined by
\begin{equation}\label{eq:proximal_operator}
    \prox_{\alpha\phi}^D(x)=\underset{y\in\R^n}{\operatorname{argmin} } \ \frac{1}{2}\|y-x\|_D^2+\alpha\phi(y), \quad \forall \ x\in\R^n.
\end{equation}    
\end{definition}

\begin{remark}
    If $\phi$ is convex in \eqref{eq:proximal_operator}, then $\prox_{\alpha \phi}^D$ is single-valued, i.e., there exists a unique solution to problem \eqref{eq:proximal_operator}; this follows from the fact that the objective function in \eqref{eq:proximal_operator} is strongly convex whenever $\phi$ is convex. If $D = I_n$, then $\prox_{\alpha \phi}^D = \prox_{\alpha \phi}$, being $\prox_{\alpha \phi}$ defined as in \eqref{eq:prox_no_D}.
\end{remark}

As explained in Section \ref{sec:intro}, the proposed framework is based on the alternating application of the inertial proximal-gradient step \eqref{eq:PHILA} on each block of coordinates. Such a step may be computed inexactly, for instance when the proximal operator of the function $\phi$ is not known in closed form. For that reason, we require a well-defined inexactness criterion for the computation of the proximal-gradient point that ensures the convergence of the scheme, while being implementable in a number of practical scenarios. The following definition of proximal inexactness is taken from \cite{Bonettini-Prato-Rebegoldi-2024}.

\begin{definition} \cite[Section 3]{Bonettini-Prato-Rebegoldi-2024}
Let $f:\R^n\rightarrow \R$ be continuously differentiable and $\phi:\R^n\rightarrow \R \cup \{\infty\}$ proper, lower semicontinuous and convex. Choose $\alpha>0$, $\beta>0$, $D\in\mathcal{S}_{++}(n)$ and $(x,w)\in\R^n\times \R^n$. Define $h(\cdot;x,w):\R^n\rightarrow \R\cup\{\infty\}$ as the strongly convex function given by
\begin{equation}\label{eq:general_metric_function}
    h(y;x,w) = \langle \nabla f(x)-\frac{\beta}{\alpha}D(x-w),y-x\rangle+\frac{1}{2\alpha}\|y-x\|_D^2+\phi(y)-\phi(x), \quad \forall \ y \in\R^n,
\end{equation}
and the inertial proximal--gradient point as
\begin{equation}\label{eq:general_exact_prox}
\hat{y}=\prox^D_{\alpha \phi}(x-\alpha D^{-1} \nabla f(x)+\beta(x-w))=\underset{y\in\R^n}{\operatorname{argmin}} \ h(y;x,w).     
\end{equation}
Given $\tau\geq 0$, we call $\tau-$approximation of $\hat{y}$ any point $\tilde{y}\in\R^n$ satisfying the following condition
\begin{equation}\label{eq:general_inexact_prox}
    h(\tilde{y};x,w
    )\leq \left(\frac{2}{2+\tau}\right)h(\hat{y};x,w
    ).
\end{equation}
Such a point is denoted with $\tilde{y}\approx_{\tau} \hat{y}$.
\end{definition}

\begin{remark}
Function $h(\cdot;x,w)$ in \eqref{eq:general_metric_function} is defined so that $h(\hat{y};x,w)\leq h(x;x,w)=0$. From this remark, it also follows that the inexactness condition \eqref{eq:general_inexact_prox} is well-defined and $h(\tilde{y};x,w)\leq 0$. If $\tau=0$, then $\tilde{y}=\hat{y}$, i.e., the $\tau-$approximation coincides with the unique proximal--gradient point; otherwise, if $\tau>0$, we have an approximation $\tilde{y}$ that gets closer to $\hat{y}$ as $\tau$ tends to zero (coarser as $\tau$ tends to infinity, respectively).
\end{remark}

The following result estimates the mutual distance between the points $x,\hat{y},\tilde{y}$ in terms of the function value $-h(\tilde{y};x,w)$. It is taken from \cite{Bonettini-Prato-Rebegoldi-2024} and follows by slightly modifying the arguments in \cite[Lemma 8]{Bonettini-Ochs-Prato-Rebegoldi-2023}, whose proof was given under the assumption $D=I_n$.

\begin{lemma}\label{lem:general_inequalities} \cite[Lemma 9]{Bonettini-Prato-Rebegoldi-2024}
Let $f:\R^n\rightarrow \R$ be continuously differentiable and $\phi:\R^n\rightarrow \R \cup \{\infty\}$ proper, lower semicontinuous and convex. Choose $0<\alpha\leq \alpha_{\max}$, $\beta>0$, $\mu>0$, $D\in\mathcal{S}_{++}(n)$ with $\frac{1}{\mu}I_n\preceq D \preceq \mu I_n$ and $(x,w)\in\R^n\times \R^n$. Let $h(\cdot;x,w)$ be the function in \eqref{eq:general_metric_function} and $\hat{y}, \tilde{y}$ the points defined as in \eqref{eq:general_exact_prox}-\eqref{eq:general_inexact_prox}. Then, the following inequalities hold:
\begin{align}
    \|\hat{y}-x\|^2&\leq 2\alpha_{\max}\mu\left(1+\frac{\tau}{2}\right)(-h(\tilde{y};x,w))\label{eq:ine_gene_1}\\
    \|\tilde{y}-\hat{y}\|^2&\leq \alpha_{\max} \mu\tau(-h(\tilde{y};x,w))\label{eq:ine_gene_2}\\
    \|\tilde{y}-x\|^2&\leq 2\alpha_{\max}\mu\left(\sqrt{1+\frac{\tau}{2}}+\sqrt{\frac{\tau}{2}}\right)^2(-h(\tilde{y};x,w)).\label{eq:ine_gene_3}
\end{align}
\end{lemma}

We now recall the definition of the Kurdyka--\L{}ojasiewicz (KL) inequality \cite{Bolte-etal-2014}. Such a definition is usually given for proper, lower semicontinuous functions; however, here we consider the definition only for differentiable functions, as it is our case of interest.

\begin{definition}\label{def:KL} \cite[Definition 3]{Bolte-etal-2014}
    Let $F:\R^n\rightarrow \R$ be a continuously differentiable function. The function $F$ satisfies the Kurdyka-{\L}ojasiewicz (KL) inequality at the point $x^*\in\R^n$ if there exist $\nu>0$, a neighborhood $U$ of $x^*$, and a continuous concave function $\xi:[0,\nu)\rightarrow [0,+\infty) $ such that $\xi(0) = 0$, $\xi\in C^1((0,\nu))$, $\xi'(s)>0$ for all $s\in(0,\nu)$, and the following inequality holds
    \begin{equation}\label{eq:KL}
        \xi'(F(x)-F(x^*))\|\nabla F(x)\|\geq 1,
    \end{equation}
    for all $x\in U \cap \{y\in\R^n: \ F(x^*)< F(y) < F(x^*)+\nu\}$. If $F$ satisfies the KL inequality for all $x^*\in\R^n$, then $F$ is called a KL function.
\end{definition}
The KL inequality \eqref{eq:KL} holds for several objective functions appearing in classical signal and image processing models, see e.g. \cite{Bolte-etal-2014,Xu-Yin-2013}. Notably, such an inequality holds also for most finite dimensional deep learning models, including those involving functions of the form \eqref{eq:g_sigma} \cite{Castera-et-al-2021}. 

\section{A block-coordinate inexact line-search framework for non-separable optimization}\label{sec:method}
In this section, we introduce Block-PHILA - Block Proximal Heavy-ball Inexact Line-search Algorithm - the novel block-coordinate framework that lies at the core of our proposed PnP approach. First, we describe Block-PHILA (Section \ref{subsec:framework}); then we compare it with other block-coordinate approaches already existing in the literature (Section \ref{subsec:related_works}); finally, we detail {its application to the PnP framework (Section \ref{sec:Plug-And-Play})}.

\subsection{The proposed framework}\label{subsec:framework}
Block-PHILA is a block-coordinate framework suitable for problems of the form \eqref{eq:problem} under the 
blanket assumptions stated in Assumption \ref{ass1}. 
\begin{assumptions}\label{ass1}
The functions $f$ and $\phi$ appearing in \eqref{eq:problem} satisfy the following conditions.
    \begin{itemize}
        \item[(i)] The function $f:\R^n\rightarrow \R$ is continuously differentiable.
        \item[(ii)] Given any matrix $U=[U_1,\ldots,U_N]$ defined as in Definition \ref{def:U}, the function $\phi:\R^n\rightarrow \R\cup \{\infty\}$ is convex along coordinates, i.e., the function 
        \begin{equation}\label{eq:phix}
            \phi_i^x:\R^{n_i}\rightarrow \R\cup\{\infty\}, \quad \phi_i^x(z) = \phi(x+U_i(z-U_i^Tx)), \quad \forall \ z\in\R^{n_i}, 
        \end{equation}
        is convex for all $x\in\R^n$ and $i=1,\ldots,N$.
        \item[(iii)] The function $F=f+\phi$ is bounded from below, i.e., there exists $F_{low}\in\R$ such that $F\geq F_{low}$.
    \end{itemize}
\end{assumptions}

\textcolor{black}{
\begin{remark}\label{remark_assumption1}
Optimization problems of the form \eqref{eq:problem} satisfying Assumption \ref{ass1} are frequently encountered in imaging inverse problems, usually expressed as the minimization of a function given by
\begin{equation}\label{eq:minDR}
F(x)=\mathcal{D}(x) + \varrho \mathcal{R}(x), 
\end{equation}
where $\mathcal{D}:\R^n\rightarrow \R$ and $\mathcal{R}:\R^n\rightarrow \R$ are continuously differentiable and bounded from below functions, the former representing a data fidelity term and the latter being a regularizer, and $\varrho\in\R_+$ is the regularization parameter.\\
We note that properties (i) and (ii) naturally allow for the inclusion of optimization models where the data fidelity $\mathcal{D}$ is both continuously differentiable and convex along the coordinates on $\R^n$. For instance, the regularized function $F$ in \eqref{eq:minDR} with a least squares data fidelity of the form
$$
\mathcal{D}(x)=\frac{1}{2}\|Ax-b\|^2,
$$
where $A\in\R^{m\times n}$ and $b\in\R^m$, satisfies Assumption \ref{ass1} for any matrix $U=[U_1,\ldots,U_N]$ by setting either $f(x)=\mathcal{R}(x)$ and $\phi(x)=\frac{1}{2}\|Ax-b\|^2$, or alternatively $f(x)=F(x)$ and $\phi(x)\equiv 0$. Least squares data fidelity functions frequently appear in image restoration, in the presence of either Gaussian noise \cite{Chambolle-Pock-2016,hurault2022gradient} or Poisson noise, in particular when approximating the generalized Kullback-Leibler divergence with its second-order Taylor expansion \cite{Burger_2014,Rebegoldi-Calatroni-2022,Ramani-et-al-2012,Sawatzky2011}.\\
In spite of property (ii), our framework can even include models where the data fidelity is continuously differentiable but nonconvex along the coordinates. In this case, for any fixed matrix $U=[U_1,\ldots,U_N]$, we can circumvent the absence of convexity by moving both the data fidelity term and the regularizer inside the term $f$, i.e., by setting $f(x)=F(x)$ and $\phi(x)\equiv 0$. As an example, we consider the regularized function $F$ in \eqref{eq:minDR} with data fidelity given by
\begin{equation}\label{eq:cauchy_fidelity} 
\mathcal{D}(x) = \sum_{i=1}^n\log(\gamma^2+((Ax)_i-b_i)^2),
\end{equation}
where $A\in\R^{n\times n}$, $b\in\R^n$ and $\gamma>0$. Such a data fidelity is nonconvex and typically appears in image deblurring in the presence of Cauchy noise \cite{Parenti-etal-2026,Sciacchitano-et-al-2015}.\\
Finally, we note that similar non-separable optimization models can be found in several other applications, such as matrix factorization, wireless resource allocation or the training of machine learning models, see e.g. \cite{chorobura2023random} and references therein.
\end{remark}
}

We start by summarizing the steps of the proposed framework. At each iteration $k\in\N$, Block-PHILA cyclically selects the index $i_k\in\{1,\ldots,N\}$, two parameters $\ak>0$, $\bk>0$, and a symmetric positive definite matrix $D_k\in\R^{n_{i_k}\times n_{i_k}}$. Starting from the current iterate $\xk$, the algorithm performs a gradient step on $f$ restricted to the $\ik-$th block of coordinates, possibly equipped with an inertial term combining the two previous updates on that block. Then, a proximal step is performed on the function $\phi$ restricted to the subspace generated by the columns of $\Uk$. Note that $\alpha_k$ and $D_k$ define the variable metric of the proximal--gradient operator, whereas $\bk$ scales the inertial term. Furthermore, such a block proximal--gradient point can be approximately computed according to a well-posed inexactness condition, which is implementable in some cases of interest. Once the proximal--gradient point is computed, the new iterate $x_{k+1}\in\R^n$ is detected by means of a line-search procedure that enforces the sufficient decrease of a suitable merit function along the search direction.

We now report Block-PHILA in full detail in Algorithm \ref{alg:block-VMILA}.

\begin{algorithm}\caption{Block Proximal Heavy-ball Inexact Line-search Algorithm (Block-PHILA)}
Choose $\xo\in\R^n$, $x_{-1} = \ldots = x_{-N}=\xo$, $0<\alpha_{\min}\leq \alpha_{\max}$, $\beta_{\max}>0$, $\mu>0$, $\tau\geq 0$, $\gamma>0$, $\delta,\textcolor{black}{\vartheta} \in(0,1)$.\\
{\sc For} $k=0,1,\ldots$
\begin{enumerate}
    \item[] {\hspace{-5mm}\sc Step 1.} Set $\ik=\mod(k,N)+1$.
    \item[] {\hspace{-5mm}\sc Step 2.} Choose $\ak\in [\alpha_{\min},\alpha_{\max}]$, $\beta_k\in [0,\beta_{\max}]$, $\Dk\in \mathcal{S}_{++}(n_{i_k})$ such that $\frac{1}{\mu}I_{n_{i_k}}\preceq \Dk\preceq \mu I_{n_{i_k}}$. 
    \item[] {\hspace{-5mm}\sc Step 3.} Let $\phik(z) = \phi(\xk + \Uk(z-\Uk^T\xk))$ and compute 
    \begin{equation}\label{eq:yk}
         \yk \approx_{\tau} \prox_{\ak \phik}^{\Dk}(\Uk^T(\xk+\beta_k(\xk-x_{k-N}))-\ak \Dk^{-1}\Uk^T\nabla \ff(\xk)).
     \end{equation}
     \item[] {\hspace{-5mm}\sc Step 4.} Set $\dk = \yk-\Uk^T\xk $ and compute
     \begin{align}\label{eq:Deltak}
     h_k(\yk) &= \langle \Uk^T\nabla \ff(\xk)-\frac{\beta_k}{\alpha_k}D_k\Uk^T(x_k-x_{k-N})
     ,\dk\rangle \nonumber\\
     &+\frac{1}{2\ak}\|\dk\|^2_{\Dk}+\phik(\yk)-\phik(\Uk^T\xk).
     \end{align}
    \item[] {\hspace{-5mm}\sc Step 5.} 
    Compute the smallest non-negative integer $\mk$ such that
    \begin{equation}\label{eq:armijo}
        F(\xk + \delta^{\mk}\Uk \dk)+\frac{\gamma}{2}\|\delta^{m_k}d_k\|^2 \leq F(\xk)+\frac{\gamma}{2}\|U_{i_k}^T(x_k-x_{k-N})\|^2+\textcolor{black}{\vartheta}\delta^{\mk}h_k(\yk),
    \end{equation}
    and set $\lambda_k = \delta^{m_k}$.
    \item[] {\hspace{-5mm}\sc Step 6.} Compute
    \begin{equation}\label{eq:update}
        \xkk = \xk + \begin{cases}
            \Uk \dk, \quad &\text{if }F(\xk+ \Uk\dk)+\frac{\gamma}{2}\|d_k\|^2 < F(\xk+\lambda_k\Uk \dk)+\frac{\gamma}{2}\lambda_k^2\|d_k\|^2\\
            \lambda_k\Uk \dk, \quad &\text{otherwise}.
        \end{cases}
    \end{equation}
\end{enumerate}
\label{alg:block-VMILA}
\end{algorithm}

Let us describe Block-PHILA step-by-step. Prior to the beginning of the iterative procedure, the framework requires the user to select the following parameters:
\begin{itemize}
    \item  $\xo\in\R^n$ and $x_{-1}=\ldots=x_{-N}=\xo$ as the $N+1$ initial guesses;
    \item $0<\alpha_{\min}\leq \alpha_{\max}$ as the lower and upper bounds, respectively, for the steplengths $\{\alpha_k\}_{k\in\N}$;
    \item $\beta_{\max}>0$ defining the interval $[0,\beta_{\max}]$ constraining the inertial parameters $\{\beta_k\}_{k\in\N}$;
    \item $\mu>0$ defining the space to which the scaling matrices $\{D_k\}_{k\in\N}$ belong to, i.e., the space of all symmetric, positive definite matrices whose eigenvalues are in the interval $[\frac{1}{\mu},\mu]$;
    \item $\tau\geq 0$ controlling the level of accuracy in the computation of the inexact proximal point;
    \item $\gamma>0$, $\delta\in(0,1)$, $\textcolor{black}{\vartheta}\in(0,1)$ as the line-search parameters.
\end{itemize}

At {\sc Steps 1-2}, we cyclically select the block index $i_k=\mod(k,N)+1$, and compute the steplength $\alpha_k\in [\alpha_{\min},\alpha_{\max}]$, the inertial parameter $\beta_k\in[0,\beta_{\max}]$ and the scaling matrix $D_k\in\mathcal{S}_{++}(n_{i_k})$ with $\frac{1}{\mu}I_{n_{i_k}}\preceq \Dk\preceq \mu I_{n_{i_k}}$. These three parameters can be computed according to any chosen updating rule and are intended to make the $k-$th block step faster than standard implementations employing either fixed parameters or local estimates of the Lipschitz constant. \textcolor{black}{Note that, in practice, $D_k$ should be chosen as a diagonal matrix, in order to keep the computational cost per iteration low.}

{\sc Step 3} aims at computing an inexact proximal--gradient step with respect to the $i_k-$th block of coordinates $x^{(i_k)} =
U_{i_k}^Tx$, while fixing the other blocks as the ones of the current iterate $\xk$. In detail, we denote with $\phik:\R^{n_{i_k}}\rightarrow \R\cup\{\infty\}$ the restriction of the function $\phi$ to the subspace generated by the columns of $\Uk$, i.e.,
\begin{equation}\label{eq:phik}
    \phik(y) = \phi(\xk+ \Uk(y-\Uk^T\xk)), \quad \forall \ y\in\R^{n_{i_k}}. 
\end{equation}
We consider the block proximal--gradient point $\hat{y}_k\in\R^{n_{i_k}}$ with parameters $\alpha_k,\beta_k,D_k$ defined as
\begin{equation}\label{eq:hyk}
    \hyk = \prox_{\ak \phik}^{\Dk}(\Uk^T(\xk+\beta_k(\xk-x_{k-N}))-\ak \Dk^{-1}\Uk^T\nabla \ff(\xk)).
\end{equation}
Note that the matrix $\Dk^{-1}$ has the role of scaling the term $-\Uk^T\nabla f(\xk)$, namely the $\ik-$th block of coordinates of the negative gradient at $\xk$, whereas $\beta_k$ scales the inertial term $\Uk^T(\xk-x_{k-N})$. By Definition \ref{def:proximal_operator} of proximal operator, the point $\hat{y}_k$ can be characterized as follows
\begin{align}
\hyk &= \underset{y\in\R^{n_{i_k}}}{\operatorname{argmin}} \ \frac{1}{2}\|y- \Uk^T(\xk+\beta_k(\xk-x_{k-N}))+\ak \Dk^{-1}\Uk^T\nabla \ff(\xk)\|_{\Dk}^2 + \ak \phik (y)\nonumber\\
&= \underset{y\in\R^{n_{i_k}}}{\operatorname{argmin}} \ \langle  \Uk^T\nabla f(\xk) - \frac{\bk}{\ak}D_k\Uk^T(\xk-x_{k-N}),y-\Uk^Tx_k\rangle+\frac{1}{2\ak}\|y-\Uk^T\xk\|^2_{\Dk}+\phik(y)\label{eq:almost_hki}\\
&=\underset{y\in\R^{n_{i_k}}}{\operatorname{argmin}} \ h_k(y),\nonumber
\end{align}
where $h_k: \R^{n_{i_k}}\rightarrow \mathbb{R}\cup\{\infty\}$ is the function defined as
\begin{align}\label{eq:hki}
   h_k(y)&=\langle \Uk^T\nabla f(\xk)-\frac{\bk}{\ak}\Dk\Uk^T(\xk-x_{k-N})
     ,y-U_{i_k}^Tx_k\rangle\nonumber\\
     &+\frac{1}{2\ak}\|y-U_{i_k}^Tx_k\|^2_{\Dk}+\phik(y)-\phik(\Uk^T\xk), \quad \forall \ y\in\R^{n_{i_k}}.
\end{align}
Note that $h_k$ is defined by shifting the function appearing in \eqref{eq:almost_hki} by the constant $-\phik(U_{i_k}^T\xk)$, so that $h_k(\Uk^T \xk) = 0$ and $h_k(\hyk)\leq h_k(\Uk^T \xk)=0$. Then, given $\tau>0$, Block-PHILA computes an inexact inertial proximal--gradient point $\tilde{y}_k\in\R^{n_i}$ complying with the following inexactness condition
\begin{equation}\label{eq:tilde_yki}
h_k(\tilde{y}_{k})\leq \left(\frac{2}{2+\tau}\right)h_k(\hyk).    
\end{equation}
Note that the search of a point $\tilde{y}_k$ complying with \eqref{eq:tilde_yki} is feasible since $h_k(\hyk)\leq 0$. If $\tau=0$, then $\tilde{y}_k=\hyk$, meaning that Block-PHILA computes the proximal--gradient point in an exact manner; viceversa, if $\tau>0$, then Block-PHILA computes an inexact point which gets closer and closer to $\yk$ as $\tau$ tends to zero.

In {\sc Step 4}, Block-PHILA defines the search direction $\dk = \yk - \Uk^T \xk$ and keeps track of the value $h_k(\yk)$. Then, {\sc Step 5} performs a line-search along the direction $\Uk \dk$ in order to satisfy an appropriate Armijo-like condition. More precisely, Block-PHILA computes the parameter $\lambda_k = \delta^{\mk}\in (0,1]$, where $\mk$ is the smallest non-negative integer such that
\begin{equation*}
F(\xk + \delta^{\mk}\Uk \dk)+\frac{\gamma}{2}\|\delta^{m_k}d_k\|^2 \leq F(\xk)+\frac{\gamma}{2}\|U_{i_k}^T(x_k-x_{k-N})\|^2+\textcolor{black}{\vartheta} \delta^{\mk}h_k(\yk).
\end{equation*}
In general, the above condition does not imply the decrease of the objective function $F$ from one iteration to the other; nonetheless, since $h_k(\yk)\leq 0$, the above condition forces the quantity $F(\xk + \lambda_k\Uk \dk)+\frac{\gamma}{2}\lambda_k^2\| d_k\|^2$ to be smaller than $F(\xk)+\frac{\gamma}{2}\|U_{i_k}^T(x_k-x_{k-N})\|^2$. We will see in Section \ref{subsec:convergence_analysis} (precisely in Lemma \ref{lem:descent}) that {\sc Step 5} is actually equivalent to imposing the sufficient decrease of a suitable merit function $\Psi$ along the sequence $\{(\xk,\xkm,\ldots,x_{k-N})\}_{k\in\N}$ of $N+1$ successive iterates.

Finally, in {\sc Step 6}, the new iterate $\xkk$ is set as either the convex combination $\xk+\lambda_k \Uk \dk$ or the point $\xk+\Uk \dk$, depending on whether the quantity $F(\xk + \lambda_k\Uk \dk)+\frac{\gamma}{2}\lambda_k^2\| d_k\|^2$ appearing in the Armijo-like condition is smaller than the quantity $F(\xk + \Uk \dk)+\frac{\gamma}{2}\| d_k\|^2$ computed for the first trial point of the backtracking procedure. \textcolor{black}{We note that {\sc Step 6} is purely technical and required in order to ensure the convergence of the algorithm. In particular, this final step is needed in order to ensure that one of the hypotheses of Theorem \ref{thm:abstract_convergence} (namely equation \eqref{eq:H2}) holds for Block-PHILA. We refer the reader to the proof of Corollary \ref{cor:convergence_iterates} for a clearer understanding of this issue.}

\subsection{\textcolor{black}{Related block-coordinate approaches}}\label{subsec:related_works}
The Block-PHILA method can be viewed as an alternating block-coor\-dinate version of the Proximal Heavy-ball Inexact Line-search Algorithm (PHILA) proposed in \cite{Bonettini-Prato-Rebegoldi-2024}. Specifically, at each iteration, Block-PHILA applies a PHILA step to one block of coordinates
$x^{(i)}$ at a time while keeping all other blocks fixed. Extending PHILA to the block-coordinate setting is non-trivial, as it requires a careful design of the inertial term across the blocks and the suitable definition of a merit function that decreases throughout the iterations. Regarding the latter, the decreasing merit function used to prove convergence in PHILA coincides with the one employed in the algorithm's line-search. In contrast, the definition of the merit function in Block-PHILA, while still relying on the line-search inequality, must be adapted in order to take into account the block-coordinate structure of the algorithm (see Lemma \ref{lem:descent}).

Beyond Block-PHILA, several other block-coordinate proximal--gradient methods have been proposed for minimizing a possibly non-convex objective function of the form \eqref{eq:problem}, accelerated either by inertial terms or by variable metrics \cite{Bonettini-Prato-Rebegoldi-2018,Chouzenoux-Pesquet-Repetti-2016,Frankel-Garrigos-Peypouquet-2015,Gur-Sabach-shtern-2023,Ochs-2019,Pock-Sabach-2018,Rebegoldi-2024,Xu-Yin-2013,Xu-Yin-2017}.  
    \textcolor{black}{More specifically, the methods in \cite{Bonettini-Prato-Rebegoldi-2018,Chouzenoux-Pesquet-Repetti-2016,Frankel-Garrigos-Peypouquet-2015,Rebegoldi-2024} employ variable-metric strategies, whereas those in \cite{Gur-Sabach-shtern-2023,Pock-Sabach-2018,Xu-Yin-2013,Xu-Yin-2017} rely on inertial extrapolation. To the best of our knowledge, among the cited works, only the framework proposed in \cite{Ochs-2019} allows inertial and variable-metric strategies to be combined. In addition, all the cited approaches, including \cite{Ochs-2019}, are tailored for problems of the form \eqref{eq:problem} where  
    the function $\phi$ consists of a sum of separable and possibly non differentiable components.}
    In contrast, Block-PHILA allows the combination of inertial and variable metric strategies, and is designed to handle objective functions of the form \eqref{eq:problem} where $\phi$ is differentiable and possibly non-separable. Such an assumption on $\phi$ is crucial in practical applications such as the Plug-and-Play framework, where the optimization problem typically involves the sum of two differentiable but non-separable terms.   
Further details on how Block-PHILA can be employed in Plug-and-Play scenarios are provided in Section \ref{sec:Plug-And-Play}. 

On the other hand, there are few  block-coordinate descent methods 
for composite objective functions as in \eqref{eq:problem} where the second term can be supposed non-separable \cite{Aberdam-etal-2021,chorobura2023random,Grishchenko-etal-2021,Hanzely-etal-2018,Latafat-etal-2022}. In \cite{Aberdam-etal-2021}, the authors study the minimization of the sum of a quadratic function and a non-smooth, non-separable function and propose a coordinate gradient descent–type method based on the forward-backward envelope to smooth the original problem. Chorobura and Necoara \cite{chorobura2023random} propose a coordinate proximal--gradient algorithm where, for problem \eqref{eq:problem}, the algorithm updates the variables by performing a descent step using selected components of the gradient of $f$, followed by a proximal step applied to a function restricted to the corresponding subspace, similarly to the one in \eqref{eq:phik}. Furthermore, given the problem \eqref{eq:problem}, \cite{Grishchenko-etal-2021,Hanzely-etal-2018} study proximal coordinate descent whose general iteration reads as
\begin{equation}\label{eq:update_references}
    {x}_{k+1} \in \prox_{\alpha \phi}\bigl(\mathcal{C}\bigl(x_k - \alpha \nabla f(x_k)\bigr)\bigr),
\end{equation}
where $\mathcal{C}(\cdot)$ is a  map defined by the randomly chosen subspace at the current iteration. 
The algorithm proposed in \cite{Latafat-etal-2022}, tailored for problem \eqref{eq:problem} with general non-convex $\phi$, is similar to that in \eqref{eq:update_references}, but  $\mathcal{C}(\cdot)$ is simply the identity map; here only a proper subset of indices $I_{k+1} \subseteq \{1,\ldots,N\}$ is selected and the corresponding block of coordinates is updated by means of \eqref{eq:update_references}, leaving the remaining coordinates unchanged by setting $x_{k+1}^{(i)} = x_k^{(i)}$, $i\in \{1,\ldots,N\}\setminus I_{k+1}$. 
It is also worth noting that \cite{Grishchenko-etal-2021,Hanzely-etal-2018} use only a sketch of the gradient of $f$ restricted to the selected subspace, while \cite{Latafat-etal-2022} assumes that $f$ itself is separable. 

\textcolor{black}{We remark that none of the aforementioned approaches for non-separable optimization adopts a line-search strategy for selecting the steplength. Instead, their convergence guarantees typically require the steplength to satisfy prescribed upper bounds, which may depend on global problem constants and can therefore be conservative or difficult to estimate in practice. By contrast, Block-PHILA uses an Armijo-type backtracking procedure to adapt the actual steplength, thereby potentially allowing the algorithm to accept larger steps than those prescribed by conservative global convergence conditions and avoiding the need to know or accurately estimate such constants in advance. The additional flexibility of the line-search comes at the cost of potentially performing multiple objective-function evaluations at each iteration. Nevertheless, the associated computational overhead is limited by the fact that the proximal point and the search direction are computed only once: during backtracking, Block-PHILA merely evaluates the objective function at the trial points, without recomputing the proximal update. Thus, the line-search introduces a trade-off between additional function evaluations and a more adaptive, potentially less conservative, choice of the steplength. In practice, the number of backtracking reductions is typically small, as will also be highlighted in the numerical experiments.}

Furthermore, the aforementioned algorithms  do not consider inertial and/or variable metric acceleration strategies, and they do not allow for an inexact computation of the proximal point. Moreover in \cite{Aberdam-etal-2021,Grishchenko-etal-2021,Hanzely-etal-2018,Latafat-etal-2022}, the proximal update is not computed coordinate-wise, but instead requires computing a block of coordinates of the full proximal point, by implicitly assuming that the function $\phi$ in \eqref{eq:problem}  admits a low-complexity proximal mapping.

To the best of our knowledge, Block-PHILA is the first block-coordinate proximal--gradient method designed for non-convex and non-separable problems that offers provable convergence guarantees, includes a line-search strategy to adaptively select the steplength without impractical constraints, and can potentially be accelerated by jointly using inertial steps and a variable metric underlying the iterations. 

\subsection{Application to the \textcolor{black}{GS} Plug-and-Play framework}\label{sec:Plug-And-Play}
This section is devoted to detail how Algorithm \ref{alg:block-VMILA} can be applied to image restoration problems of the  form
\begin{equation}\label{eq:PnP_framework}
    \underset{x\in\R^n}{\operatorname{argmin}} \ 
    \textcolor{black}{F(x)\equiv\mathcal{D}(x)+\varrho}g_\sigma(x),
\end{equation}
where  
\textcolor{black}{$\mathcal{D}:\R^n\rightarrow \R$ is a continuously differentiable and bounded from below }data fidelity term, $g_{\sigma}:\R^n\rightarrow \R$ is defined in \eqref{eq:g_sigma}, and $\textcolor{black}{\varrho}>0$ is a regularization parameter. 

\textcolor{black}{Problem \eqref{eq:PnP_framework} can always be reformulated as problem \eqref{eq:problem} under Assumption \ref{ass1} by setting 
\begin{equation}\label{eq:PnP_splitting_1}
    \begin{cases}
        \phi(x) = 0\\
        f(x) = \mathcal{D}(x)+\varrho g_\sigma(x).
    \end{cases}
\end{equation}
If $\mathcal{D}$ is also convex along the coordinates on $\R^n$, we can consider the alternative splitting
\begin{equation}\label{eq:PnP_splitting_2}
\begin{cases}
        \phi(x) = \mathcal{D}(x)\\
        f(x) = \varrho g_\sigma(x).
    \end{cases}    
\end{equation}
In both splittings, the computation of $\nabla g_\sigma(x_k)$ restricted to each block in \eqref{eq:yk} is required.
} 

\textcolor{black}{Regarding the choice of the blocks}, we consider a collection of decomposition submatrices $\{U_i\}_{i=1}^N$ that partition the image into blocks of contiguous pixels. For example, an image can be decomposed into four non-overlapping quadrants, each representing one block of coordinates. 

In the following, we show how the computation of 
\textcolor{black}{the restricted block-gradient $U_{i_k}^T\nabla g_\sigma(x_k)$} can be efficiently handled.

\subsubsection{Block GS denoiser}
In this section, we describe how it is possible to reduce the 
memory overhead required to perform the computation of $\Uk^T\nabla g_\sigma(\xk)$ through the derivation of a restricted block version of the GS denoiser. More in detail, 
in view of \eqref{eq:denoising_step},  $\nabla g_\sigma$ depends on $J_{N_{\sigma}}(x)$, whose computation requires storing the entire computational graph in order to perform the backward pass. The memory needed to perform this intermediate step depends on the size of the images, thus limiting the usability of GS denoisers in limited resources scenarios, for example on a laptop or on a smartphone. For this reason, we introduce a variant of the GS denoisers, named block GS denoisers (BGS), in order to perform the denoising step on a single block of the image with a reduced memory footprint on the GPU, avoiding to compute the entire $J_{N_{\sigma}}(x)$ and the subsequent extraction of the rows associated to the selected block. 
This construction is crucial for the implementation of Block-PHILA in the considered PnP framework.

We recall the concept of receptive field for a convolutional neural network in order to derive our BGS denoisers. In general, the receptive field refers to the region of the input that influences the single pixel in the output. More specifically, the receptive field of a simple convolutional layer is equal to its kernel size \cite{Luo-etal-2016}. The theoretical receptive field of a convolutional neural network (CNN) can be computed by considering how the size, stride, and padding of the kernel of each layer combine to determine the size of the input region that affects a specific output location. Mathematically, this means that a single pixel in the output does not depend on the pixels outside of its receptive field. This independence is reflected also in the partial derivatives of the neural network.

In our setting, we consider the function $N_{\sigma}$ defined as a convolutional neural network, specifically employing a UNet architecture as in \cite{hurault2022gradient}. In general, CNN architectures enable the network to handle images of varying sizes without requiring modification to the model parameters or structure. Let us consider $i_k\in\{1,\ldots,N\}$ as a block selection index and $\hat{i}\in\{1,\ldots,N\}$. Computationally, if we are interested in the partial derivatives of $N_{\sigma}$ in the $\hat{i}$-th component there are two cases 
  \begin{itemize}
    \item $\frac{\partial N_{\sigma}(x)_{r}}{\partial x_j}=0$, if the $j-$th pixel of $x$ 
    is not in the receptive field associated to the pixel $r$ in the block $\hat{i}$;
    \item $\frac{\partial N_{\sigma}(x)_{r}}{\partial x_j}$ \textcolor{black}{may be nonzero otherwise}.
  \end{itemize}
For each block index $i_k$ we can define the restricted neural network acting on a smaller patch of the image, namely $\widetilde{N}^{i_k}_{\sigma}:\R^{n_i+o_k}\rightarrow\R^{n_i+o_k}$, where $o_k$ denotes the number of additional pixels considered in the padded block of coordinates, such that the following condition holds  \begin{equation}\label{eq:condition_restricted}
      U_{i_k}^T N_{\sigma}(x)=U_{i_k}^T\bar{U}_{i_k}\widetilde{N}^{i_k}_{\sigma}(\bar{U}^T_{i_k}x),
  \end{equation}
where $\bar{U}_{i_k}$ denotes the masking matrix which considers the indexes of the pixels within the receptive field associated to the pixels in $I_{i_k}$. In practice, this means that extracting a block of coordinates of the neural network output $N_{\sigma}(x)$ is equivalent to extracting the same block of the restricted neural network output $\widetilde{N}_\sigma^{i_k}(\bar{U}_{i_k}^Tx)$. Now, considering the following potential function:

\begin{equation}
    \widetilde{g}^{i_k}_{\sigma}(x) = \frac{1}{2}\|\bar{U}^T_{i_k} x - \widetilde{N}^{i_k}_{\sigma}(\bar{U}^T_{i_k}x)\|^2
\end{equation}
we can compute its gradient
\begin{equation}
    \nabla  \widetilde{g}^{i_k}_{\sigma}(x)  =  \bar{U}_{i_k}\left(\bar{U}^T_{i_k} x - \widetilde{N}^{i_k}_{\sigma}(\bar{U}^T_{i_k}x)- J_{\widetilde{N}^{i_k}_{\sigma}}(\bar{U}^T_{i_k}x)^T
    (\bar{U}^T_{i_k} x - \widetilde{N}^{i_k}_{\sigma}(\bar{U}^T_{i_k}x))\right).
\end{equation}
We point out that the computation of $\widetilde{N}^{i_k}_{\sigma}$ and its associated Jacobian  $J_{\widetilde{N}^{i_k}_{\sigma}}$ require less memory since its input belongs to a smaller space. Furthermore, the $i_k-$th block of coordinates of $\nabla  \widetilde{g}^{i_k}_{\sigma}(x)$ writes as
\begin{align}\label{eq:nabla_tildeg_intermediate}
\Uk^T\nabla  \widetilde{g}^{i_k}_{\sigma}(x)&=\Uk^T\bar{U}_{i_k}\left(\bar{U}^T_{i_k} x - \widetilde{N}^{i_k}_{\sigma}(\bar{U}^T_{i_k}x)- J_{\widetilde{N}^{i_k}_{\sigma}}(\bar{U}^T_{i_k}x)^T
    (\bar{U}^T_{i_k} x - \widetilde{N}^{i_k}_{\sigma}(\bar{U}^T_{i_k}x))\right)\nonumber\\
&=\Uk^T\left(x-N_{\sigma}(x)-\bar{U}_{i_k}J_{\widetilde{N}^{i_k}_{\sigma}}(\bar{U}^T_{i_k}x)^T
    (\bar{U}^T_{i_k} x - \widetilde{N}^{i_k}_{\sigma}(\bar{U}^T_{i_k}x))\right),
\end{align}
where we used \eqref{eq:condition_restricted} and property $U_{i_k}^T\bar{U}_{i_k}\bar{U}_{i_k}^T = U_{i_k}^T$ of the masking matrices. Finally, noting that
\begin{equation}\label{eq:nabla_tildeg_fact}
U_{i_k}^T \bar{U}_{i_k} J_{\widetilde{N}^{i_k}_{\sigma}}(\bar{U}^T_{i_k}x)^T
    (\bar{U}^T_{i_k} x - \widetilde{N}^{i_k}_{\sigma}(\bar{U}^T_{i_k}x))=U_{i_k}^T J_{N_{\sigma}}(x)^T(x - N_\sigma(x)),  
\end{equation}
and plugging \eqref{eq:nabla_tildeg_fact} into \eqref{eq:nabla_tildeg_intermediate}, we conclude that 
\begin{equation}\label{eq:block_gradient_step}
U_{i_k}^T \nabla  \widetilde{g}^{i_k}_{\sigma}(x) = U_{i_k}^T \nabla  g_{\sigma}(x).
\end{equation}
In Figure \ref{fig:cheapgradient_diagram}, we present a diagram illustrating the dimensions of the tensors involved in the gradient computation under the proposed approach, considering a set of masking matrices that select one image quadrant at a time.

\textcolor{black}{Finally, we emphasize that the arguments presented in this section for efficiently computing the GS denoising step on a single block rely on the fact that each output pixel depends only on a finite neighborhood of input pixels. They therefore apply to CNNs and, more generally, to architectures with a bounded receptive field. Unlike CNNs, Transformers possess a global receptive field, meaning the output leverages the entire global context. For this type of architectures with global dependencies, the exact evaluation of a single block may require processing the entire image, in this case the proposed padding strategy does not provide the same computational advantage. The GS denoiser proposed in \cite{hurault2022gradient} and used in the numerical experiments of Section~\ref{sec:numerical_experiments} is based on a convolutional architecture and therefore falls within the setting considered here. A related empirical padding strategy for block-wise denoising was also considered in \cite{Sun-etal-2020}, where the employed neural network denoiser is likewise CNN-based.}

 \begin{figure}
    \centering
    \begin{tikzpicture}[scale=0.3]
    \draw[help lines] (-5,-5) grid (5,5);
    \node at (0, -6) {$x$};
        \fill[green!30, semitransparent] (-5,-1) -- (-5, 5) -- (1,5) -- (1,-1);
        \fill[cyan!20, nearly opaque] (-5,0) -- (-5, 5) -- (0,5) -- (0,0);
        \draw[->] (5,0)   --  (8,0) node[midway,above] () {{$\widetilde{N}^{i_k}_{\sigma}$}};
        \begin{scope}[shift={(11,0)}]
            \draw[help lines] (-3,-3) grid (3,3);
            \node at (0, -4.5) {$\widetilde{N}^{i_k}_{\sigma}(\bar{U}^T_{i_k}x)$};
        \end{scope}
        \draw[->] (14,0)   --  (17,0) node[midway,above] () {{$\nabla\widetilde{g}^{i_k}_{\sigma}$}};
        \begin{scope}[shift={(20,0)}]
            \draw[help lines] (-3,-3) grid (3,3);
            \fill[red!30, semitransparent] (-3,-3) -- (-3, 3) -- (3,3) -- (3,-3);
            \fill[cyan!20, nearly opaque] (-3,-2) -- (-3, 3) -- (2,3) -- (2,-2);
            \node at (0, -4.5) {$J_{\widetilde{N}^{i_k}_{\sigma}}(\bar{U}^T_{i_k}x)$};
        \end{scope}
        \draw[->] (23,0)   --  (27,0) node[midway,above] () {$U_{i_k}^T$};
        \begin{scope}[shift={(27,-2.5)}]
            \draw[help lines] (0,0) grid (5,5);
            \node at (2.5, -1.5) {$U_{i_k}^T \nabla  \widetilde{g}^{i_k}_{\sigma}(x)$};
        \end{scope}
    \end{tikzpicture}
    \caption{Action of the block GS denoiser, illustrated on a $10\times 10$ image decomposed into four $5\times 5$ quadrants, $i_k$ being the index of the top-left one. Left: the input image $x$; the cyan area is the block  selected by $U_{i_k}$, while the green contour collects the additional pixels belonging to the receptive field associated to the patch, so that the cyan and green areas together form the padded block selected by $\bar{U}_{i_k}$. Center: the restricted network $\widetilde{N}^{i_k}_{\sigma}$ is evaluated on the padded block only. Right: the gradient $\nabla \widetilde{g}^{i_k}_{\sigma}$  is computed exploiting the Jacobian $J_{\widetilde{N}^{i_k}_{\sigma}}(\bar{U}^T_{i_k}x)$ on the whole padded block (red area), and the block of interest (cyan area) is finally extracted by means of $U_{i_k}^T$, which returns  $U_{i_k}^T \nabla  \widetilde{g}^{i_k}_{\sigma}(x)$ that corresponds to $U_{i_k}^T\nabla g_{\sigma}(x)$, as described in \eqref{eq:nabla_tildeg_fact}.}
    \label{fig:cheapgradient_diagram}
\end{figure}

\section{Convergence analysis of Block-PHILA}\label{subsec:convergence_analysis}
In this section, we carry out the convergence analysis of Algorithm \ref{alg:block-VMILA}. In the first part, we show a sublinear convergence rate and the stationarity of limit points under Assumption \ref{ass1} and some smoothness requirements on both $f$ and $\phi$ (Section \ref{subsubsec:convergence_part1}). Then, we prove the convergence of the whole sequence to a stationary point and related convergence rates, by additionally assuming that some appropriate surrogate functions of $F$ satisfy the KL property on their domains (Section \ref{subsubsec:convergence_part2}).

\subsection{Stationarity of limit points and sublinear rate}\label{subsubsec:convergence_part1}
We start by deriving the following lemma, which will be frequently used in the convergence analysis of Algorithm \ref{alg:block-VMILA}.

\begin{lemma}\label{lem:technical_1}
Suppose Assumption \ref{ass1} holds. For all $k\in\N$, the following inequalities hold:
\begin{align}
    \|\hyk-\Uk^T\xk\|^2&\leq 2\alpha_{\max}\mu\left(1+\frac{\tau}{2}\right)(-h_k(\yk))\label{eq:ine_phila_1}\\
    \|\yk-\hyk\|^2&\leq \alpha_{\max} \mu\tau(-h_k(\yk))\label{eq:ine_phila_2}\\
    \|\dk\|^2=\|\yk-\Uk^T\xk\|^2&\leq 2\alpha_{\max}\mu\left(\sqrt{1+\frac{\tau}{2}}+\sqrt{\frac{\tau}{2}}\right)^2(-h_k(\yk)).\label{eq:ine_phila_3}
\end{align}
\end{lemma}

\begin{proof}
In Lemma \ref{lem:general_inequalities}, set $f(\textcolor{black}{\cdot}) = f(\xk+\Uk(\textcolor{black}{\cdot}-\textcolor{black}{\Uk^T}\xk))$, $\phi(\textcolor{black}{\cdot}) = \phi_{i_k}^{\xk}(\textcolor{black}{\cdot})$, $x = \Uk^T\xk$, $w=\Uk^Tx_{k-N}$, $\alpha=\alpha_k$, $\beta=\beta_k$ and $D=D_k$. By using these settings, we have $\hat{y}=\hyk$, $\tilde{y}=\yk$, with $\hyk,\yk$ defined in \eqref{eq:hyk}-\eqref{eq:tilde_yki} and $h(\cdot;x,w)=h_k(\cdot)$, being $h_k$ defined in \eqref{eq:hki}. Hence, the inequalities \eqref{eq:ine_phila_1}-\eqref{eq:ine_phila_2}-\eqref{eq:ine_phila_3} follow as mere applications of inequalities \eqref{eq:ine_gene_1}-\eqref{eq:ine_gene_2}-\eqref{eq:ine_gene_3} from Lemma \ref{lem:general_inequalities}.
\end{proof}

Next, we state the well-posedness of the line-search performed at {\sc Step 5} of Algorithm \ref{alg:block-VMILA}. To this aim, we require the function $f$ to have a Lipschitz continuous gradient. 
\begin{assumptions}\label{ass2}
$\nabla f$ is Lipschitz continuous with constant $L_f>0$, i.e., 
\begin{equation}\label{eq:Lipschitz}
\|\nabla f(x)-\nabla f(y)\|\leq L_f\|x-y\|, \quad \forall \ x,y\in \R^n. 
\end{equation}
\end{assumptions}

\begin{remark}
Assumption \ref{ass2} holds \textcolor{black}{for the GS PnP model \eqref{eq:PnP_framework} with the splitting \eqref{eq:PnP_splitting_2}, i.e., by setting }
$f=\textcolor{black}{\varrho} g_\sigma$, being $\textcolor{black}{\varrho}>0$ and $g_\sigma$ defined in \eqref{eq:g_sigma}
, provided that the neural network $N_\sigma$ in \eqref{eq:g_sigma} is given by the composition of continuously differentiable functions whose derivatives are bounded and Lipschitz \cite[Proposition 2]{hurault2022gradient}, as is the case of the convolutional neural network adopted in the numerical experiments of Section \ref{sec:numerical_experiments}. \textcolor{black}{Further, Assumption \ref{ass2} is satisfied for model \eqref{eq:PnP_framework} also with the splitting \eqref{eq:PnP_splitting_1}, i.e., setting $f(x)=\mathcal{D}(x)+\varrho g_\sigma(x)$, provided that the continuously differentiable data fidelity term $\mathcal{D}$ has a Lipschitz continuous gradient (see e.g. the data fidelities in Remark \ref{remark_assumption1}).}
\end{remark}

\begin{lemma}\label{lem:lambdamin} 
Suppose Assumptions \ref{ass1}-\ref{ass2} hold.
\begin{itemize}
    \item[(i)] For all $k\in\N$, the line-search defined at {\sc Step 5} of Algorithm \ref{alg:block-VMILA} terminates in a finite number of steps, i.e., $m_k<\infty$ for all $k\in\N$.
    \item[(ii)] There exists $\lambda_{\min}>0$ such that
\begin{equation*}
    \lambda_k=\delta^{m_k}\geq \lambda_{\min}, \quad \forall \ k\in\mathbb{N}.
\end{equation*} 
\end{itemize}
\end{lemma}

\begin{proof}
The line-search at {\sc Step 5} of Algorithm \ref{alg:block-VMILA} is exactly the one described in \cite[Algorithm 1]{Bonettini-Prato-Rebegoldi-2024} when applied to the functions $f_0(\cdot)=f(\xk+\Uk(\cdot-\Uk^T\xk))$ and $f_1(\cdot)= \phi_{i_k}^{\xk}(\cdot)$, using the iterates $x^{(k)} = U_{i_k}^T\xk$ and $x^{(k-1)} = \Uk^Tx_{k-N}$, and the search direction $d^{(k)} = \yk - \Uk^T\xk$. Hence, the thesis follows by applying \cite[Lemma 12]{Bonettini-Prato-Rebegoldi-2024}. 
\end{proof}
 
We now show that the line-search at {\sc Step 5} enforces the sufficient decrease of a suitable merit function $\Psi$, which incorporates both the objective function $F$ and the information related to the inertial steps. Similar merit functions have previously appeared in the analysis of other block-coordinate proximal--gradient methods with inertial terms \cite{Ochs-2019,Pock-Sabach-2018}.

\begin{lemma}\label{lem:descent}
Suppose Assumptions \ref{ass1}-\ref{ass2} hold. Define the merit function \begin{equation}\label{eq:psi}
    \Psi(z_1,\ldots,z_{N+1})=F(z_1) + \frac{\gamma}{2}\sum_{i=1}^N\|z_i-z_{i+1}\|^2, \quad z_i\in\R^n, \ i=1,\ldots,\textcolor{black}{N}+1.
\end{equation}
Then, the following statements hold.
\begin{itemize}
    \item[(i)] For all $k\in\N$, we have 
    \begin{equation}\label{eq:suff_decr_phi}
    \Psi(\xkk,x_k,\ldots,x_{k-N+1})
    \leq \Psi(\xk,x_{k-1},\ldots,x_{k-N})+\textcolor{black}{\vartheta}\lambda_{\min} h_k(\yk).   
    \end{equation}

    \item[(ii)] We have
    \begin{equation}\label{eq:summable_h}
        \sum_{k=0}^{\infty}(-h_k(\yk))<\infty.
    \end{equation}
\end{itemize}    
\end{lemma}

\begin{proof}
(i) We start by observing that
\begin{equation}\label{eq:inertial_technical}
\Uk^T(\xk-x_{k-N}) 
=\Uk^T(x_{k-N+1}-x_{k-N}),
\end{equation}
as the $i_k-$th block of coordinates \textcolor{black}{is not updated during any of} the previous $N-1$ iterations $k-1,k-2,\ldots,k-N+1$ preceeding the $k-$th iteration. Furthermore $\|\Uk^T(x_{k-N+1}-x_{k-N})\| = \|x_{k-N+1}-x_{k-N}\|$ as the $i_k-$th block of coordinates is the only one to be updated at iteration $k-N$. Then, the Armijo-like condition \eqref{eq:armijo} in {\sc Step 5} of Algorithm \ref{alg:block-VMILA} rewrites as
\begin{equation*}
F(\xk + \lambda_k \Uk \dk)+\frac{\gamma}{2}\lambda_k^2\|d_k\|^2 \leq F(\xk)+\frac{\gamma}{2}\|x_{k-N+1}-x_{k-N}\|^2+\textcolor{black}{\vartheta} \lambda_kh_k(\yk).    
\end{equation*}
By summing the quantity $\frac{\gamma}{2}\sum_{i=1}^{N-1}\|x_{k-i+1}-x_{k-i}\|^2$ to both sides of the previous inequality, we obtain 
\begin{equation*}
F(\xk + \lambda_k \Uk \dk)+\frac{\gamma}{2}\lambda_k^2\|d_k\|^2+\frac{\gamma}{2}\sum_{i=1}^{N-1}\|x_{k-i+1}-x_{k-i}\|^2 \leq F(\xk)+\frac{\gamma}{2}\sum_{i=1}^{N}\|x_{k-i+1}-x_{k-i}\|^2+\textcolor{black}{\vartheta} \lambda_kh_k(\yk).   
\end{equation*}
By recalling the definition of the merit function \eqref{eq:psi}, the Armijo-like condition is easily rewritten as
\begin{equation}\label{eq:almost_suff_decr_phi}
\Psi(\xk + \lambda_k \Uk \dk,x_k,x_{k-1},\ldots,x_{k-N+1})
\leq \Psi(\xk,x_{k-1},x_{k-2}\ldots,x_{k-N})+\textcolor{black}{\vartheta} \lambda_k h_k(\yk).    
\end{equation}
Furthermore, by following a similar reasoning as the one just conducted for the Armijo-like condition, one can rewrite {\sc Step 6} in terms of the function $\Psi$ as
\begin{equation}\label{eq:Step6_rewritten}
\xkk = \xk + \begin{cases}
\Uk \dk, \quad &\text{if }\Psi(\xk+ \Uk\dk,\xk,\ldots,x_{k-N+1}) < \Psi(\xk+ \lambda_k\Uk\dk,\xk,\ldots,x_{k-N+1})\\
\lambda_k\Uk \dk, \quad &\text{otherwise}.
\end{cases}
\end{equation}
Hence, in both cases, the following property holds
\begin{equation}\label{eq:step6_property}
\Psi(x_{k+1},x_k,x_{k-1},\ldots,x_{k-N+1})\leq \Psi(\xk + \lambda_k \Uk \dk,x_k,x_{k-1},\ldots,x_{k-N+1}).
\end{equation}
Combining \eqref{eq:step6_property} and \eqref{eq:almost_suff_decr_phi} leads to the inequality
\begin{equation}\label{eq:almost_done}
\Psi(\xkk,x_k,\ldots,x_{k-N+1})
\leq \Psi(\xk,x_{k-1},\ldots,x_{k-N})+\textcolor{black}{\vartheta} \lambda_k h_k(\yk).    
\end{equation}
By recalling that $h_k(\yk)\leq 0$ and applying Lemma \ref{lem:lambdamin}(ii) to \eqref{eq:almost_done}, we obtain the thesis of item (i).

(ii) By summing inequality \eqref{eq:suff_decr_phi} for $k=0,\ldots,K$, and noting the telescopic sum at the right-hand side, we get the following
\begin{align}
    \textcolor{black}{\vartheta}\lambda_{\min}\sum_{k=0}^{K}(-h_k(\yk))&\leq \sum_{k=0}^{K}(\Psi(x_k,x_{k-1},\ldots,x_{k-N})-\Psi(x_{k+1},\xk,\ldots,x_{k-N+1}))\nonumber\\
    &=\Psi(x_0,x_{-1},\ldots,x_{-N})-\Psi(x_{K+1},x_{K},\ldots,x_{K-N+1})\nonumber\\
    &\leq \Psi(x_0,x_{-1},\ldots,x_{-N})-F_{low},\label{eq:hk_almost_summable}
\end{align}
where the last inequality follows from Assumption \ref{ass1}(iii) and the definition of the merit function $\Psi$ in \eqref{eq:psi}. The thesis follows by taking the limit for $K\rightarrow \infty$ on inequality \eqref{eq:hk_almost_summable} and noting that its right-hand side is independent of $K$.
\end{proof}

To proceed further with the analysis, it is crucial to assume that also $\phi$ is continuously differentiable with locally Lipschitz continuous gradient. Such a differentiability assumption could be neglected if $\phi$ was a separable sum with respect to the blocks, i.e., in the case where $\phi(x) = \sum_{i=1}^N\phi_i(U_i^Tx)$. Assuming that $\phi$ is a separable sum is quite common in the literature of block-coordinate proximal--gradient methods \cite{Bolte-etal-2014,Bonettini-Prato-Rebegoldi-2018,Chouzenoux-Pesquet-Repetti-2016,Frankel-Garrigos-Peypouquet-2015,Gur-Sabach-shtern-2023,Ochs-2019}; however, we prefer to assume the differentiability of $\phi$ and drop the separability assumption, as the PnP framework based on model \eqref{eq:PnP_framework} does not allow neither $f$ nor $\phi$ to be separable.

\begin{assumptions}\label{ass3}
$\phi$ \textcolor{black}{is finite and} continuously differentiable \textcolor{black}{on $\R^n$} and $\nabla \phi$ is locally Lipschitz continuous \textcolor{black}{on $\R^n$}, i.e., for all compact subsets $K\subseteq \R^n$ there exists $L_{\phi,K}>0$ such that  
\begin{equation}\label{eq:phi_Lipschitz}
\|\nabla \phi(x)-\nabla \phi(y)\|\leq L_{\phi,K}\|x-y\|, \quad \forall \ x,y\in K. 
\end{equation}
\end{assumptions}

\begin{remark}\label{remark_assumption3}
\textcolor{black}{Assumption \ref{ass3} is trivially satisfied for the GS PnP model \eqref{eq:PnP_framework} with the splitting \eqref{eq:PnP_splitting_1}, namely setting $\phi(x)=0$. Otherwise, if we employ splitting \eqref{eq:PnP_splitting_2} and set $\phi(x) = \mathcal{D}(x)$, the data fidelity term needs to have a locally Lipschitz continuous gradient, which is the case for the examples reported in Remark \ref{remark_assumption1}.\\
One model not complying with our assumptions is the GS PnP framework \eqref{eq:PnP_framework} with the generalized Kullback-Leibler divergence as data fidelity, which covers the case of linear imaging inverse problems under Poisson noise according to the Maximum A Posteriori approach \cite{Bertero2018}. Although this function is convex along the coordinates, it can not be included in our analysis because it is not continuously differentiable on the entire space $\R^n$, thus contrasting the differentiability requirements made in Assumptions \ref{ass1} and \ref{ass3}.}
\end{remark}

\textcolor{black}{Another crucial requirement is the boundedness of the Block-PHILA iterates.}

\textcolor{black}{
\begin{assumptions}\label{ass4}
The sequence $\{x_k\}_{k\in\N}$ generated by Algorithm \ref{alg:block-VMILA} is bounded.    
\end{assumptions}
}

\textcolor{black}{
\begin{remark}
Assumption \ref{ass4} holds whenever the objective function $F$ in the PnP model \eqref{eq:PnP_framework} is coercive. Indeed, by combining the definition of the merit function $\Psi$ in \eqref{eq:psi} and item (i) of Lemma \ref{lem:descent}, it follows that
\begin{equation*}
F(x_k)\leq \Psi(x_k,x_{k-1},\ldots,x_{k-N})\leq \Psi(x_0,x_{-1},\ldots,x_{-N})=\Psi(x_0,\ldots,x_0)=F(x_0),    
\end{equation*}
therefore $\{x_k\}_{k\in\N}$ is contained in the sublevel set $\{x\in\R^n: \ F(x)\leq F(x_0)\}$, which is bounded provided that $F$ is coercive. As noted in \cite[Appendix D]{hurault2022gradient}, we can always make the function $F$ in \eqref{eq:PnP_framework} coercive by choosing a convex, compact subset $K\subseteq \R^n$ and replacing the regularization term $g_\sigma(x)$ with $\tilde{g}_\sigma(x) = g_\sigma(x) + \frac{1}{2}\|x-P_{K}(x)\|^2$, where $P_K$ is the Euclidean projection onto $K$.   
\end{remark}
}

The following Lemma shows an upper bound on the norm of the gradient of the objective function along the iterates $\{\hat{u}_k\}_{k\in\N}$, which are defined so that $\Uk^T\hat{u}_k = \hyk$ and $U_j^T\hat{u}_k=U_j^T\xk$ for any $j\neq \ik$. The result is inspired by \cite[Lemma 3.2]{chorobura2023random}, although unlike in \cite{chorobura2023random} we need to take into account the inexactness of the proximal--gradient point, as well as the presence of inertial terms, variable metrics and a linesearch procedure. As in \cite{chorobura2023random}, we require the sequence $\{\xk\}_{k\in\N}$ to be bounded, which is used in combination with the local Lipschitz continuity of $\nabla \phi$ to ensure the existence of the upper bound.

\begin{lemma}\label{lem:gradient_bound}
    Suppose Assumptions \ref{ass1}-\ref{ass2}-\ref{ass3}-\ref{ass4} hold. Set
    \begin{equation}\label{eq:huk}
    \hat{u}_k = x_k + U_{i_k}(\hyk-U_{i_k}^Tx_k), \quad \forall \ k\in \N.
    \end{equation}
    Then, for all $k\geq N$, there exists $b>0$ such that
\begin{equation}\label{eq:gradient_norm_bound}
        \|\nabla F(\hat{u}_k)\|^2\leq b^2 \sum_{j=-(N-1)}^N(-h_{k+j-1}(\tilde{y}_{k+j-1})). 
    \end{equation}
\end{lemma}

\begin{proof}
We start by noting that the sequence $\{\hat{u}_k\}_{k\in\N}$ is bounded, due to \textcolor{black}{Assumption \ref{ass4}, \textsc{Step 2} of Algorithm \ref{alg:block-VMILA}} and the continuity of the proximal operator with respect to its parameters. Then, there exists a compact $K\subseteq \R^n$ that contains both $\{x_k\}_{k\in\N}$ and $\{\hat{u}_k\}_{k\in\N}$ and a corresponding Lipschitz constant $L_{\phi,K}$ satisfying Assumption \ref{ass3}. As also Assumption \ref{ass2} holds, we can conclude that the following local, Lipschitz-block conditions are valid for all $i=1,\dots,N$:
\begin{align}
\|U_i^T(\nabla f(x)-\nabla f(y))\|\leq L_f\|x-y\|, \quad \|U_i^T(\nabla \phi(x)-\nabla \phi(y))\|\leq L_{\phi,K}\|x-y\|, \quad \forall \ x,y\in K.\label{eq:Lipschitz-like}
\end{align}

Furthermore, we recall that the proximal--gradient point $\hyk$ is the unique minimum point of the function $h_k$ defined in \eqref{eq:hki}, i.e., $\hyk=\operatorname{argmin}_{y\in\R^{n_{i_k}}}h_k(y)$, and the corresponding optimality condition can be rewritten as follows
    \begin{align}\label{eq:optimality}
        \nabla h_k(\hyk) = 0 \quad &\Leftrightarrow \quad  U_{i_k}^T\nabla f(x_k)-\frac{\bk}{\ak}\Dk \Uk^T(\xk-x_{k-N})+\frac{1}{\alpha_k}D_k(\hyk-U_{i_k}^Tx_k)+\nabla \phi_{i_k}^{x_k}(\hyk)=0\nonumber\\
        &\Leftrightarrow \quad U_{i_k}^T\nabla \phi(x_k+U_{i_k}(\hyk-U_{i_k}^Tx_k))\nonumber\\
        &\phantom{\Leftrightarrow \quad}=-U_{i_k}^T\nabla f(x_k)+\frac{\bk}{\ak}\Dk \Uk^T(\xk-x_{k-N})-\frac{1}{\alpha_k}D_k(\hyk-U_{i_k}^Tx_k)\nonumber\\
        &\Leftrightarrow \quad U_{i_k}^T\nabla \phi(\hat{u}_k)=-U_{i_k}^T\nabla f(x_k)+\frac{\bk}{\ak}\Dk \Uk^T(\xk-x_{k-N})-\frac{1}{\alpha_k}D_k(\hyk-U_{i_k}^Tx_k),
    \end{align}
    where the second equivalence follows from the definition of $\phi_{i_k}^{x_k}$ in \eqref{eq:phik} and the chain rule holding for differentiable functions, whereas the third is obtained by definition of the point $\hat{u}_k$ in \eqref{eq:huk}.
    
    Based on \eqref{eq:optimality}, the quantity $\|\nabla F(\hat{u}_k)\|^2$ can be rewritten as follows:
    \begin{align*}
        &\|\nabla F(\hat{u}_k)\|^2=\sum_{j=1}^{N}\|U_{i_{k+j-1}}^T\nabla F(\hat{u}_k)\|^2=\sum_{j=1}^{N}\|U_{i_{k+j-1}}^T(\nabla \ff(\hat{u}_k)+\nabla \fg(\hat{u}_k))\|^2\\
        &=\sum_{j=1}^{N}\|U_{i_{k+j-1}}^T(\nabla \ff(\hat{u}_k)-\nabla f(x_{k+j-1})+\nabla \fg(\hat{u}_k)-\nabla \phi(\hat{u}_{k+j-1}))\\
        &-\frac{1}{\alpha_{k+j-1}}D_{k+j-1}(\hat{y}_{k+j-1}-U_{i_{k+j-1}}^Tx_{k+j-1})+\frac{\beta_{k+j-1}}{\alpha_{k+j-1}}D_{k+j-1}U_{i_{k+j-1}}^T(x_{k+j-1}-x_{k+j-1-N})\|^2,
\end{align*}
where the third equality is due to the optimality condition \eqref{eq:optimality} with $k$ replaced by $k+j-1$. Applying the relation $\|a+b\|^2\leq 2\|a\|^2+2\|b\|^2$ \textcolor{black}{two} times and {\sc Step 2} of Algorithm \ref{alg:block-VMILA} yields
\begin{align*}        
 &\|\nabla F(\hat{u}_k)\|^2\leq 4 \sum_{j=1}^{N}\left(\|U_{i_{k+j-1}}^T(\nabla \ff(\hat{u}_k)-\nabla f(x_{k+j-1}))\|^2+\frac{1}{\alpha_{k+j-1}}\|D_{k+j-1}(\hat{y}_{k+j-1}-U_{i_{k+j-1}}^Tx_{k+j-1})\|^2\right)\\
&+4\sum_{j=1}^{N}\left(\|U_{i_{k+j-1}}^T(\nabla \phi(\hat{u}_k)-\nabla \phi(\hat{u}_{k+j-1}))\|^2+\frac{\beta_{k+j-1}}{\alpha_{k+j-1}}\|D_{k+j-1}U_{i_{k+j-1}}^T(x_{k+j-1}-x_{k+j-1-N})\|^2\right)\\ 
&\leq 4 \sum_{j=1}^{N}\left(\|U_{i_{k+j-1}}^T(\nabla \ff(\hat{u}_k)-\nabla f(x_{k+j-1}))\|^2+\frac{\textcolor{black}{\mu^2}}{\alpha_{\min}}\|\hat{y}_{k+j-1}-U_{i_{k+j-1}}^Tx_{k+j-1}\|^2\right)\\
&+4 \sum_{j=1}^{N}\left(\|U_{i_{k+j-1}}^T(\nabla \phi(\hat{u}_k)-\nabla \phi(\hat{u}_{k+j-1}))\|^2+\frac{\beta_{\max}\textcolor{black}{\mu^2}}{\alpha_{\min}}\|U_{i_{k+j-1}}^T(x_{k+j-1}-x_{k+j-1-N})\|^2\right).
\end{align*}
By applying the inequalities in \eqref{eq:Lipschitz-like}, we get
\begin{align}
\|\nabla F(\hat{u}_k)\|^2&\leq 4L_f^2\sum_{j=1}^{N}\|\hat{u}_k-x_{k+j-1}\|^2+\frac{4\textcolor{black}{\mu^2}}{\alpha_{\min}}\sum_{j=1}^N\|\hat{y}_{k+j-1}-U_{i_{k+j-1}}^Tx_{k+j-1}\|^2\nonumber\\
&+4L_{\phi,K}^2\sum_{j=1}^{N}\|\hat{u}_k-\hat{u}_{k+j-1}\|^2+\frac{4\beta_{\max}\textcolor{black}{\mu^2}}{\alpha_{\min}}\sum_{j=1}^N\|U_{i_{k+j-1}}^T(x_{k+j-1}-x_{k+j-1-N})\|^2\nonumber\\
&= 4L_f^2\sum_{j=1}^{N}\|\hat{u}_k-x_k+\sum_{t=1}^{j-1}(x_{k+t-1}-x_{k+t})\|^2+\frac{4\textcolor{black}{\mu^2}}{\alpha_{\min}}\sum_{j=1}^N\|\hat{y}_{k+j-1}-U_{i_{k+j-1}}^Tx_{k+j-1}\|^2\nonumber\\
&+4L_{\phi,K}^2\sum_{j=1}^{N}\|\hat{u}_k-x_k+\sum_{t=1}^{j-1}(x_{k+t-1}-x_{k+t})+x_{k+j-1}-\hat{u}_{k+j-1}\|^2\nonumber\\
&+\frac{4\beta_{\max}\textcolor{black}{\mu^2}}{\alpha_{\min}}\sum_{j=1}^N\|U_{i_{k+j-1}}^T(x_{k+j-1}-x_{k+j-1-N})\|^2.\nonumber
\end{align}
Repeatedly applying Jensen's inequality to the above inequality yields
\begin{align}
\|\nabla F(\hat{u}_k)\|^2&\leq 4L_f^2\sum_{j=1}^{N}\left(2\|\hat{u}_k-x_k\|^2+2(j-1)\sum_{t=1}^{j-1}\|x_{k+t-1}-x_{k+t}\|^2\right)\nonumber\\
&+4L_{\phi,K}^2\sum_{j=1}^{N}\left(2\|\hat{u}_k-x_k\|^2+4\|x_{k+j-1}-\hat{u}_{k+j-1}\|^2+4(j-1)\sum_{t=1}^{j-1}\|x_{k+t-1}-x_{k+t}\|^2\right)\nonumber\\
&+\frac{4\textcolor{black}{\mu^2}}{\alpha_{\min}}\sum_{j=1}^N\|\hat{y}_{k+j-1}-U_{i_{k+j-1}}^Tx_{k+j-1}\|^2+\frac{4\beta_{\max}\textcolor{black}{\mu^2}}{\alpha_{\min}}\sum_{j=1}^N\|U_{i_{k+j-1}}^T(x_{k+j-1}-x_{k+j-1-N})\|^2.\label{eq:so_technical}
\end{align}

Next, we upper-bound some of the squared norms appearing in \eqref{eq:so_technical}. We note that
\begin{align}
    \|\hat{u}_k-x_k\|&=\|U_{i_k}(\hyk-\Uk^T\xk)\|=\|\hyk-\Uk^T\xk\|\label{eq:norm1}\\
    \|\hat{u}_{k+j-1}-x_{k+j-1}\|&=\|\hat{y}_{k+j-1}-U_{i_{k+j-1}}^Tx_{k+j-1}\|.\label{eq:norm1_bis}
\end{align}
Furthermore, by {\sc Step 6} of Algorithm \ref{alg:block-VMILA}, it is either $x_{k+t}-x_{k+t-1}=\lambda_{k+t-1} U_{i_{k+t-1}} d_{k+t-1}$ or $x_{k+t}-x_{k+t-1}=U_{i_{k+t-1}} d_{k+t-1}$, so that in both cases
\begin{equation}\label{eq:norm2}
 \|x_{k+t}-x_{k+t-1}\|\leq \|U_{i_{k+t-1}}d_{k+t-1}\|=\|d_{k+t-1}\|.   
\end{equation}
In addition, following a similar reasoning as in \eqref{eq:inertial_technical}, we can show that
\begin{align}
\|U_{i_{k+j-1}}^T(x_{k+j-1}-x_{k+j-1-N})\|&=\|U_{i_{k+j-1}}^T(x_{k+j-N}-x_{k+j-1-N})\|\nonumber\\
&=\|x_{k+j-N}-x_{k+j-N-1}\|\leq \|d_{k+j-N-1}\|.\label{eq:norm3}
\end{align}
Plugging \eqref{eq:norm1}-\eqref{eq:norm1_bis}-\eqref{eq:norm2}-\eqref{eq:norm3} into \eqref{eq:so_technical} yields
\begin{align*}
&\|\nabla F(\hat{u}_k)\|^2\leq 4L_f^2\sum_{j=1}^{N}\left(2\|\hyk-\Uk^T\xk\|^2+2(N-1)\sum_{t=1}^{N}\|d_{k+t-1}\|^2\right)\\
&+4L_{\phi,K}^2\sum_{j=1}^{N}\left(2\|\hyk-\Uk^T\xk\|^2+4\|\hat{y}_{k+j-1}-U_{i_{k+j-1}}^Tx_{k+j-1}\|^2+4(N-1)\sum_{t=1}^{N}\|d_{k+t-1}\|^2\right)\\
&+\frac{4\textcolor{black}{\mu^2}}{\alpha_{\min}}\sum_{j=1}^N\|\hat{y}_{k+j-1}-U_{i_{k+j-1}}^Tx_{k+j-1}\|^2+\frac{4\beta_{\max}\textcolor{black}{\mu^2}}{\alpha_{\min}}\sum_{j=1}^N\|d_{k+j-N-1}\|^2\\
&=8N(N-1)(L_f^2+2L_{\phi,K}^2)\sum_{t=1}^N\|d_{k+t-1}\|^2+4\left(\frac{\textcolor{black}{\mu^2}}{\alpha_{\min}}+4L_{\phi,K}^2\right)\sum_{j=1}^N\|\hat{y}_{k+j-1}-U_{i_{k+j-1}}^Tx_{k+j-1}\|^2\\
&+8N(L_f^2+L_{\phi,K}^2)\|\hyk-\Uk^T\xk\|^2+\frac{4\beta_{\max}\textcolor{black}{\mu^2}}{\alpha_{\min}}\sum_{j=1}^N\|d_{k+j-N-1}\|^2.
\end{align*}
Now, we can apply Lemma \ref{lem:technical_1} to all the squared norms in the previous inequality, thus concluding there exist $b_1>0$ and $b_2>0$ such that
\begin{equation*}
    \|\nabla F(\hat{u}_k)\|^2 \leq b_1^2 \sum_{j=1}^N(-h_{k+j-1}(\tilde{y}_{k+j-1})) + b_2^2\sum_{j=1}^N(-h_{k+j-N-1}(\tilde{y}_{k+j-N-1})).
\end{equation*}
Finally, the thesis follows by setting $b=\sqrt{\max\{b_1^2,b_2^2\}}$.
\end{proof}

Thanks to the previous Lemma, it is possible to prove the first mild convergence result for Algorithm \ref{alg:block-VMILA}, as detailed in the following Theorem.

\begin{theorem}
Suppose Assumptions \ref{ass1}-\ref{ass2}-\ref{ass3}-\ref{ass4} hold. The following statements hold true.
\begin{itemize}
\item[(i)] We have $\sum_{k=0}^{\infty}\|\nabla F(x_k)\|^2<\infty$, thus each limit point of $\{x_k\}_{k\in\N}$ is stationary for $F$.
\item[(ii)] There exists $C>0$ such that
\begin{equation*}
\min_{0\leq i\leq \mathcal{K}}\|\nabla F(x_i)\|^2\leq \frac{C}{\mathcal{K}+1}.
\end{equation*}
\end{itemize}
\end{theorem}

\begin{proof}
(i) Let $\mathcal{K}\geq N$. We sum inequality \eqref{eq:gradient_norm_bound} for $k=N,N+1,\ldots,\mathcal{K}$, so as to obtain
\begin{align}
    \sum_{k=N}^\mathcal{K}\|\nabla F(\hat{u}_k)\|^2\leq \sum_{k=N}^{\mathcal{K}}\left(b^2\sum_{j=-(N-1)}^N(-h_{k+j-1}(\tilde{y}_{k+j-1}))\right)\nonumber&=b^2\sum_{j=-(N-1)}^N\left(\sum_{k=N}^{\mathcal{K}}(-h_{k+j-1}(\tilde{y}_{k+j-1}))\right)\nonumber\\
    &=b^2 \sum_{j=-(N-1)}^N\sum_{k'={N+j-1}}^{\mathcal{K}+j-1}(-h_{k'}(\tilde{y}_{k'}))\nonumber\\
    &\leq 2b^2N\sum_{k'=0}^{\mathcal{K}+N-1}(-h_{k'}(\tilde{y}_{k'})).\label{eq:sum_N_K}
\end{align}
By taking the limit for $\mathcal{K}\rightarrow \infty$ on both sides of  \eqref{eq:sum_N_K} and exploiting property \eqref{eq:summable_h}, we get
\begin{equation}\label{eq:normguk_summable}
    \sum_{k=0}^{\infty}\|\nabla F(\hat{u}_k)\|^2<\infty.
\end{equation}
In order to conclude, we observe that \textcolor{black}{Assumption \ref{ass4} ensures the existence of a compact set $K\subseteq \mathbb{R}^n$ containing both $\{x_k\}_{k\in\mathbb{N}}$ and $\{\hat{u}_k\}_{k\in\mathbb{N}}$, and of a corresponding Lipschitz constant $L_{\phi,K}$ satisfying Assumption \ref{ass3}. Together with Assumption \ref{ass2}, this implies that there exists $L_{F,K}>0$ such that $\nabla F$ is Lipschitz continuous on $K$. Hence, similarly to property \eqref{eq:Lipschitz-like}, it holds that
 $$
 \|U_i^T(\nabla F(x)-\nabla F(y)\|\leq L_{F,K}\|x-y\|, \quad \forall x,y\in K.
 $$
 As a consequence, we get}
\begin{align}
    \|\nabla F(\xk)\|^2&\leq 2\|\nabla F(\xk)-\nabla F(\hat{u}_k)\|^2+2\|\nabla F(\hat{u}_k)\|^2\nonumber\\
    &=2\left\|\sum_{i=1}^NU_i^T (\nabla F(\xk)-\nabla F(\xk+\Uk(\hyk-\Uk^T\xk)))\right\|^2+2\|\nabla F(\hat{u}_k)\|^2\nonumber\\
    &\leq 2N\sum_{i=1}^N\left\|U_i^T (\nabla F(\xk)-\nabla F(\xk+\Uk(\hyk-\Uk^T\xk)))\right\|^2+2\|\nabla F(\hat{u}_k)\|^2\nonumber\\
    &\leq 2N^2\textcolor{black}{L_{F,K}^2}\|\hyk-\Uk^T\xk\|^2+2\|\nabla F(\hat{u}_k)\|^2\nonumber\\
    &\leq 4N^2\textcolor{black}{L_{F,K}^2}\alpha_{\max}\mu\left(1+\frac{\tau}{2}\right)(-h_k(\yk))+2\|\nabla F(\hat{u}_k)\|^2,\label{eq:ine_squared}
\end{align}
where the first two inequalities follow by using Jensen's inequality, 
and the fourth is implied by inequality \eqref{eq:ine_phila_1}. Summing \eqref{eq:ine_squared} for $k=0,\ldots,\mathcal{K}$ leads to
\begin{equation}\label{eq:crucial_ine_for_mild_convergence}
 \sum_{k=0}^\mathcal{K}\|\nabla F(\xk)\|^2\leq  4N^2\textcolor{black}{L_{F,K}^2}\alpha_{\max}\mu\left(1+\frac{\tau}{2}\right)\sum_{k=0}^{\mathcal{K}} (-h_k(\yk)) +2\sum_{k=0}^K\|\nabla F(\hat{u}_k)\|^2.
\end{equation}
By taking the limit of \eqref{eq:crucial_ine_for_mild_convergence} for $\mathcal{K}\rightarrow \infty$, and applying \eqref{eq:normguk_summable} and \textcolor{black}{item (ii) of Lemma \ref{lem:descent}}, we get that $\{\|\nabla F(\xk)\|^2\}_{k\in\N}$ is summable. Hence, the limit $\lim_{k\rightarrow \infty}\|\nabla F(\xk)\|=0$ holds, and by continuity of $\nabla F$, we conclude that each limit point is stationary.

(ii) We bound the finite sums in the right-hand side of \eqref{eq:crucial_ine_for_mild_convergence} with their (finite) limits to obtain
\begin{equation*}
\sum_{k=0}^{\mathcal{K}}\|\nabla F(\xk)\|^2\leq  4N^2\textcolor{black}{L_{F,K}^2}\alpha_{\max}\mu\left(1+\frac{\tau}{2}\right)\sum_{k=0}^\infty (-h_k(\yk)) +2\sum_{k=0}^\infty\|\nabla F(\hat{u}_k)\|^2<\infty.
\end{equation*}
The thesis then follows by applying the definition of minimum to the above inequality.
\end{proof}

\subsection{Convergence under the KL property}\label{subsubsec:convergence_part2}

In the following, we are interested in proving the convergence of the sequence $\{x_k\}_{k\in\N}$ generated by Algorithm \ref{alg:block-VMILA} to a stationary point of $F$. To this aim, we rely on the following abstract convergence result. 

\begin{theorem}\cite[Theorem 14]{Bonettini-Prato-Rebegoldi-2024}\label{thm:abstract_convergence}
Let $\F:\R^{n}\times\R^m\to \bR$ be a proper, lower semicontinuous KL function and $\Psi:\R^n\times \R^q\to\bR$ a proper, lower semicontinuous function that is bounded from below. Consider a sequence $\{(\xk,\rho_k)\}_{k\in\N}\subset \R^n\times\R^m$ such that $\{\xk\}_{k\in\N}$ is bounded and $\{\rho_k\}_{k\in\N}$ converges, and let $\{\uk\}_{k\in\N}\subset \R^n$, $\{r_k\}_{k\in\N}\subset \R^q$, $\{\sk\}_{k\in\N}\subset \R^q$ be chosen such that the following properties hold.
\begin{enumerate}
\item[(i)] There exists a  positive real number $a$ such that 
\begin{equation}\label{eq:H1}
\Psi(\xkk,s_{k+1})+a r_k^2 \leq \Psi(\xk,\sk), \quad \forall \ k\in \N.
\end{equation}
\item[(ii)] There exists a sequence of non--negative real numbers $\{t_k\}_{k\in\N}$ with $\lim\limits_{k\to\infty} t_k = 0$ such that
\begin{equation}\label{eq:H2}
\Psi(\xkk,s_{k+1})\leq \F(\uk,\rho_k)\leq \Psi(\xk,\sk)+t_k, \quad \forall \ k\in \N.
\end{equation}
\item[(iii)] There exists a subgradient $w_k\in\partial \F(\uk,\rho_k)$ such that 
\begin{equation}\label{eq:H3}
\|w_k\|\leq p \sum_{i\in \mathcal{I}}r_{k+1-i}, \quad \forall \ k\in \N,
\end{equation}
where $p$ is a positive real number, ${\mathcal I}\subset \mathbb{Z}$ is a non-empty, finite index set, and $r_j=0$ for $j\leq 0$.
\item[(iv)] If $\{(x_{k_j},\rho_{k_j})\}_{j\in\N}$ is a subsequence of $\{(\xk,\rho_k)\}_{k\in\N}$ converging to some $ (x^*,\rho^*)\in\R^n\times \R^m$, then we have 
\begin{equation}\label{eq:H4}
\lim_{j\to\infty}\|u_{k_j}-x_{k_j}\|=0, \quad \lim_{j\to\infty} \F(u_{k_j},\rho_{k_j}) = \F(x^*,\rho^*).
\end{equation}
\item[(v)] There exists a positive real number $q>0$ such that
\begin{equation}\label{eq:H5}
\|\xkk-\xk\|\leq q r_{k}, \quad \forall \ k\in \N.
\end{equation}
\end{enumerate}
Then, the sequence $\{(\xk,\rho_k)\}_{k\in\N}$ converges to a stationary point of $\mathcal F$.    
\end{theorem}

Theorem \ref{thm:abstract_convergence} guarantees the convergence of the abstract iterates $\{(\xk,\rho_k)\}_{k\in\N}$ to a stationary point of the KL function $\F$, provided that the iterates comply with properties \eqref{eq:H1}-\eqref{eq:H5}. Such a result generalizes several abstract convergence results holding under the KL inequality and similar assumptions on the iterates \cite{Attouch-etal-2013,Bolte-etal-2014,Ochs-2019}. By relying on Theorem \ref{thm:abstract_convergence}, we can ensure the convergence of the iterates generated by Algorithm \ref{alg:block-VMILA} to a stationary point of $F$. To this aim, the role of the merit function $\Psi$ appearing in Theorem \ref{thm:abstract_convergence} is played by the function in \eqref{eq:psi}, whereas the other merit function $\F$ is defined as
\begin{equation}\label{eq:matF}
\F(x,\rho)= F(x)+\frac{1}{2}\rho^2, \quad \forall \ (x,\rho)\in\R^n\times\R.    
\end{equation}
Note that the function \eqref{eq:matF} is a KL function whenever $F$ is definable in an $o-$minimal structure \cite[Definition 6]{Bolte-etal-2007b}, as definable functions do satisfy the KL inequality on their domains and finite sums of definable functions in some $o-$minimal structure remain in the same structure \cite[Remark 5]{Bolte-etal-2007b}. When the specific PnP model \eqref{eq:PnP_framework} is considered, we can ensure $F$ is definable whenever the data fidelity $\phi$ and the regularizer $g_\sigma$ in \eqref{eq:g_sigma} are definable in the same structure. Interestingly, if $\phi$ is chosen as a least-squares functional, then $g_\sigma$ is definable in the same structure as $\phi$ provided that the employed neural network $N_\sigma$ is built upon definable functions such as ReLU, eLU, quadratics and SoftPlus functions, as the composition of definable functions remains definable \cite[Section 5.2]{Davis-etal-2020}. This is exactly the scenario addressed in our numerical experiments, where $\phi(x)=\frac{1}{2}\|Ax-b\|^2$ with $b\in\R^n$, $A\in\R^{m\times n}$, and the neural network used in $g_\sigma$ is defined by means of linear and eLU functions (see Section \ref{sec:numerical_experiments}).

We now state the convergence result holding for Algorithm \ref{alg:block-VMILA} and based on Theorem \ref{thm:abstract_convergence}.

\begin{cor}\label{cor:convergence_iterates}
Suppose Assumptions \ref{ass1}-\ref{ass2}-\ref{ass3}-\ref{ass4} hold. If the function $\mathcal{F}$ defined in \eqref{eq:matF} is a KL function, then $\{\xk\}_{k\in\N}$ converges to a stationary point of $F$.
\end{cor}

\begin{proof}
The proof consists in showing that Algorithm \ref{alg:block-VMILA} satisfies the properties from \eqref{eq:H1} to \eqref{eq:H5} listed in Theorem \ref{thm:abstract_convergence}.

According to Lemma \ref{lem:descent}, we immediately obtain the descent property \eqref{eq:H1} with $\Psi$ defined as in \eqref{eq:psi}, $r_k = \sqrt{-h_k(\yk)}$, $s_k = (x_{k-1},\ldots,x_{k-N})$ and $a=\lambda_{\min}\textcolor{black}{\vartheta}$.

Next, we show that property \eqref{eq:H2} holds. As noted in the proof of Lemma \ref{lem:gradient_bound}, the sequence $\{\hat{u}_k\}_{k\in\N}$ defined in \eqref{eq:huk} is bounded. Furthermore, let $\tilde{u}_k = x_k+\Uk d_k$ for all $k\in\N$. From inequalities \eqref{eq:ine_phila_3} and \eqref{eq:summable_h}, it follows that $\lim_{k\rightarrow \infty}d_k= 0$; by combining this limit with \textcolor{black}{Assumption \ref{ass4}}, we conclude that also $\{\tilde{u}_k\}_{k\in\N}$ is bounded. Now, let $\tilde{K}\subseteq \R^n$ be any compact set containing both $\{\hat{u}_k\}_{k\in\N}$ and $\{\tilde{u}_k\}_{k\in\N}$. Since $\nabla F$ is locally Lipschitz continuous by Assumptions \ref{ass2} and \ref{ass3}, the Descent Lemma for $F$ locally holds on $\tilde{K}$ \cite[Lemma 5.7]{Beck-2017}, i.e., there exists $L_{F,\tilde{K}}>0$ such that 
\begin{equation}\label{eq:descent_lemma}
F(y)\leq F(x)+ \nabla F(x)^T(y-x)+\frac{L_{F,\tilde{K}}}{2}\|y-x\|^2, \quad \forall \ x,y\in \tilde{K}.    
\end{equation}
By setting $x=\hat{u}_k$ and $y=\tilde{u}_k$ in equation \eqref{eq:descent_lemma}, we obtain
\begin{align}
F(\tilde{u}_k)&\leq F(\hat{u}_k)+\nabla F(\hat{u}_k)^T(\tilde{u}_k-\hat{u}_k)+\frac{L_{F,\tilde{K}}}{2}\|\tilde{u}_k-\hat{u}_k\|^2\nonumber\\
&=\textcolor{black}{F(\hat{u}_k)+\nabla F(\hat{u}_k)^T(x_k+U_{i_k}d_k-x_k-U_{i_k}(\hat{y}_k-U_{i_k}^Tx_k))}\nonumber\\
&\quad\textcolor{black}{+\frac{L_{F,\tilde{K}}}{2}\|x_k+U_{i_k}d_k-x_k-U_{i_k}(\hat{y}_k-U_{i_k}^Tx_k))\|^2}\nonumber\\
&=\textcolor{black}{F(\hat{u}_k)+\nabla F(\hat{u}_k)^T(U_{i_k}(\tilde{y}_k-U_{i_k}^Tx_k)-U_{i_k}(\hat{y}_k-U_{i_k}^Tx_k))}\nonumber\\
&\quad\textcolor{black}{+\frac{L_{F,\tilde{K}}}{2}\|U_{i_k}(\tilde{y}_k-U_{i_k}^Tx_k)-U_{i_k}(\hat{y}_k-U_{i_k}^Tx_k))\|^2}\nonumber\\
&=F(\hat{u}_k)+(\Uk^T \nabla F(\hat{u}_k))^T(\tilde{y}_k-\hat{y}_k)+\frac{L_{F,\tilde{K}}}{2}\|\tilde{y}_k-\hat{y}_k\|^2,\label{eq:descent_lemma_huk}
\end{align}
\textcolor{black}{where the first equality follows from the definitions of $\tilde{u}_k$ and $\hat{u}_k$ in \eqref{eq:huk}, while the second equality follows from the definition of $d_k$ in \textsc{Step 4} of Algorithm \ref{alg:block-VMILA}}.\\
By summing the term $\Uk^T\nabla f(\hat{u}_k)$ to both sides of \eqref{eq:optimality}, we can express the $\ik-$block of coordinates of $\nabla F(\hat{u}_k)$ through the following equality
\begin{equation}
U_{i_k}^T\nabla F(\hat{u}_k)=U_{i_k}^T(\nabla f(\hat{u}_k)-\nabla f(x_k))+\frac{\bk}{\ak}\Dk \Uk^T(\xk-x_{k-N})-\frac{1}{\alpha_k}D_k(\hyk-U_{i_k}^Tx_k).\label{eq:block_gradient_huk}    
\end{equation}
Plugging \eqref{eq:block_gradient_huk} in \eqref{eq:descent_lemma_huk} yields
\begin{align*}
F(\tilde{u}_k)&\leq F(\hat{u}_k)+(\nabla f(\hat{u}_k)-\nabla f(x_k))^T\Uk(\tilde{y}_k-\hat{y}_k)+\frac{L_{F,\tilde{K}}}{2}\|\tilde{y}_k-\hat{y}_k\|^2\\
&+\frac{\bk}{\ak}(x_k-x_{k-N})^T\Uk \Dk (\tilde{y}_k-\hat{y}_k)-\frac{1}{\ak}(\hat{y}_k-\Uk^T\xk)^T\Dk(\tilde{y}_k-\hat{y}_k)\\
&\leq F(\hat{u}_k)+L_f\|\hat{u}_k-x_k\| \|\tilde{y}_k-\hat{y}_k\|+\frac{L_{F,\tilde{K}}}{2}\|\tilde{y}_k-\hat{y}_k\|^2\\
&+\frac{\bk}{\ak}\|\Uk^T(\xk-x_{k-N})\|\|\Dk(\tilde{y}_k-\hat{y}_k)\|+\frac{1}{\ak}\|\hat{y}_k-\Uk^T\xk\|\|\Dk(\tilde{y}_k-\hat{y}_k)\|\\
& \leq F(\hat{u}_k)+\left(L_f+\frac{\mu}{\alpha_{\min}}\right)\|\hat{y}_k-\Uk^T\xk\|\|\tilde{y}_k-\hat{y}_k\|+\frac{L_{F,\tilde{K}}}{2}\|\tilde{y}_k-\hat{y}_k\|^2\\
&+\frac{\beta_{\max}\mu}{\alpha_{\min}}\|\tilde{y}_k-\hat{y}_k\|\|x_{k-N+1}-x_{k-N}\|,
\end{align*}
where the second inequality follows by applying the Cauchy-Schwarz inequality and the Lipschitz continuity of $\nabla f$, and the third is due to the definition of $\hat{u}_k$ in \eqref{eq:huk}, {\sc Step 2} of Algorithm \ref{alg:block-VMILA}, and equality \eqref{eq:inertial_technical}. We now sum the term $\frac{\gamma}{2}\|d_k\|^2 + \frac{\gamma}{2}\sum_{i=1}^{N-1}\|x_{k-i+1}-x_{k-i}\|^2$ to both sides of the obtained inequality and recall the definition of the function $\Psi$ in \eqref{eq:psi}, hence obtaining
\begin{align}\label{eq:Psi_uk}
\Psi(\tilde{u}_k,\xk,\ldots,x_{k-N+1})&\leq F(\hat{u}_k)+\frac{\gamma}{2}\|d_k\|^2+\frac{\gamma}{2}\sum_{i=1}^{N-1}\|x_{k-i+1}-x_{k-i}\|^2\nonumber\\
&+\frac{L_{F,\tilde{K}}}{2}\|\tilde{y}_k-\hat{y}_k\|^2+\left(L_f+\frac{\mu}{\alpha_{\min}}\right)\|\hat{y}_k-\Uk^T\xk\|\|\tilde{y}_k-\hat{y}_k\|\nonumber\\
&+\frac{\beta_{\max}\mu}{\alpha_{\min}}\|\tilde{y}_k-\hat{y}_k\|\|x_{k-N+1}-x_{k-N}\|\nonumber\\
&\leq F(\hat{u}_k)+\frac{\gamma}{2}\|d_k\|^2+\frac{\gamma}{2}\sum_{i=1}^{N-1}\|d_{k-i}\|^2+\frac{L_{F,\tilde{K}}}{2}\|\tilde{y}_k-\hat{y}_k\|^2\nonumber\\
&+\left(L_f+\frac{\mu}{\alpha_{\min}}\right)\|\hat{y}_k-\Uk^T\xk\|\|\tilde{y}_k-\hat{y}_k\|+\frac{\beta_{\max}\mu}{\alpha_{\min}}\|\tilde{y}_k-\hat{y}_k\|\|d_{k-N}\|,
\end{align}
where each norm $\|x_{k-i+1}-x_{k-i}\|$ has been bounded with $\|d_{k-i}\|$ thanks to inequality \eqref{eq:norm2}. Thus, we define $\{\rho_k\}_{k\in\N}$ as the nonnegative sequence such that
\begin{align}\label{eq:rho}
\frac{\rho_k^2}{2} &= \frac{\gamma}{2}\|d_k\|^2+\frac{\gamma}{2}\sum_{i=1}^{N-1}\|d_{k-i}\|^2+\frac{L_{F,\tilde{K}}}{2}\|\tilde{y}_k-\hat{y}_k\|^2\nonumber\\
&+\left(L_f+\frac{\mu}{\alpha_{\min}}\right)\|\hat{y}_k-\Uk^T\xk\|\|\tilde{y}_k-\hat{y}_k\|+\frac{\beta_{\max}\mu}{\alpha_{\min}}\|\tilde{y}_k-\hat{y}_k\|\|d_{k-N}\|.
\end{align}
Recalling the definition of $\F$ in \eqref{eq:matF} and using \eqref{eq:rho} in \eqref{eq:Psi_uk} leads to
\begin{equation}
\Psi(\tilde{u}_k,\xk,x_{k-1},\ldots,x_{k-N+1})\leq \F(\hat{u}_k,\rho_k).    
\end{equation}
By using \textcolor{black}{Step 6, reformulated as in \eqref{eq:Step6_rewritten},} and the definition of $\tilde{u}_k$, it immediately follows that
\begin{equation}\label{eq:left_hand_almost_there}
\Psi(x_{k+1},\xk,x_{k-1},\ldots,x_{k-N+1})\leq \Psi(\tilde{u}_k,\xk,x_{k-1},\ldots,x_{k-N+1})\leq \mathcal{F}(\hat{u}_k,\rho_k),  
\end{equation}
which is exactly the left-hand inequality of property \eqref{eq:H2} with $u_k = \hat{u}_k$. Furthermore, if we replace $F$ with $f$ and $L_{F,\tilde{K}}$ with $L_f$ in the Descent Lemma \eqref{eq:descent_lemma} and set $y=\hat{u}_k$ and $x=\xk$, we can write
\begin{align*}
f(\hat{u}_k)&\leq f(\xk)+\nabla f(\xk)^T(\hat{u}_k-\xk)+\frac{L_f}{2}\|\hat{u}_k-\xk\|^2\\
&= f(\xk)+(\Uk^T\nabla f(\xk))^T(\hat{y}_k-\Uk^T\xk)+\frac{L_{f}}{2}\|\hat{y}_k-\Uk^T\xk\|^2.
\end{align*}
By summing the term $\phi(\hat{u}_k)$ to both sides of the above inequality, and observing that $\phi(\hat{u}_k)=\phik(\hat{y}_k)$ and $\phi(\xk)=\phik(\Uk^T\xk)$ (see \eqref{eq:phik}), the following inequalities hold
\begin{align}\label{eq:so_technical_-1}
F(\hat{u}_k)
& \leq F(\xk)+ (\Uk^T\nabla f(\xk))^T(\hat{y}_k-\Uk^T\xk)+\phik(\hat{y}_k)-\phik(\Uk^T\xk)+\frac{L_{f}}{2}\|\hat{y}_k-\Uk^T\xk\|^2\nonumber\\
&\leq F(\xk) + h_k(\hat{y}_k) + \frac{L_{f}}{2}\|\hat{y}_k-\Uk^T\xk\|^2+\frac{\beta_k}{\alpha_k}(x_k-x_{k-N})^TU_{i_k}D_k(\hat{y}_k-U_{i_k}^Tx_k)\nonumber\\
&\leq F(\xk) + h_k(\hat{y}_k)+\frac{L_{f}}{2}\|\hat{y}_k-\Uk^T\xk\|^2 +\frac{\beta_{\max} \mu}{\alpha_{\min}} \|x_{k-N+1}-x_{k-N}\|\|\hat{y}_k-U_{i_k}^Tx_k\|\nonumber\\
& \leq F(\xk) + \frac{L_{f}}{2}\|\hat{y}_k-\Uk^T\xk\|^2 +\frac{\beta_{\max} \mu}{\alpha_{\min}} \|d_{k-N}\|\|\hat{y}_k-U_{i_k}^Tx_k\|,
\end{align}
where the second inequality is obtained by adding the term $\frac{1}{2\alpha_k}\|\hat{y}_k-\Uk^T\xk\|^2_{D_k}$, summing and subtracting the term $\langle\frac{\beta_k}{\alpha_k}D_kU_{i_k}^T(x_k-x_{k-N}),\hat{y}_k-U_{i_k}^Tx_k\rangle$ and recalling the definition of $h_k$ in \eqref{eq:hki}, the third follows by applying the Cauchy-Schwarz inequality, {\sc Step 2} of Algorithm \ref{alg:block-VMILA} and equality \eqref{eq:inertial_technical}, and the fourth is due to $h_k(\hat{y}_k)\leq h_k(\Uk^T\xk)=0$ and \eqref{eq:norm2}. We sum the term $\rho_k^2/2$ to both sides of \eqref{eq:so_technical_-1} and observe that $F(\xk)\leq \Psi(\xk,\xkm,\ldots,x_{k-N})$ by definition of $\Psi$ in \eqref{eq:psi}, so as to obtain
\begin{equation*}
\F(\hat{u}_k,\rho_k)\leq \Psi(\xk,\xkm,\ldots,x_{k-N}) + \frac{L_{f}}{2}\|\hat{y}_k-\Uk^T\xk\|^2 +\frac{\beta_{\max} \mu}{\alpha_{\min}} \|d_{k-N}\|\|\hat{y}_k-U_{i_k}^Tx_k\|+\frac{\rho_k^2}{2}.
\end{equation*}
At this point, we conveniently define the sequence $\{t_k\}_{k\in\N}$ as
\begin{equation}\label{eq:tk}
t_k = \frac{L_{f}}{2}\|\hat{y}_k-\Uk^T\xk\|^2 +\frac{\beta_{\max} \mu}{\alpha_{\min}} \|d_{k-N}\|\|\hat{y}_k-U_{i_k}^Tx_k\|+\frac{\rho_k^2}{2}.
\end{equation} 
Note that from \eqref{eq:rho}, Lemma \ref{lem:technical_1} and the property $\lim_{k\rightarrow \infty}h_k(\tilde{y}_k)=0$ (due to \eqref{eq:summable_h}), it follows that 
\begin{equation}\label{eq:limit_rhok}
\lim_{k\rightarrow \infty}\rho_k=0.   
\end{equation}
Based on \eqref{eq:tk}, \eqref{eq:limit_rhok} and Lemma \ref{lem:technical_1}, we also obtain that $\lim_{k\rightarrow \infty}t_k=0$, and the right-hand inequality of property \eqref{eq:H2} is finally proved.

We move on to property \eqref{eq:H3}. By applying Lemma \ref{lem:gradient_bound} and relation $\sqrt{u+v}\leq \sqrt{u}+\sqrt{v}$, we have
\begin{align*}
\|\nabla \F(\hat{u}_k,\rho_k)\|\leq \|\nabla F(\hat{u}_k)\|+\rho_k
\leq b\sum_{j=-(N-1)}^{N}\sqrt{-h_{k+j-1}(\tilde{y}_{k+j-1})}+\rho_k.
\end{align*}
From the definition of $\rho_k$ in \eqref{eq:rho} and Lemma \ref{lem:technical_1}, it is easy to prove the existence of $b_{\rho}>0$ such that
\begin{equation*}
    \rho_k \leq b_{\rho}\sum_{j=-(N-1)}^{N}\sqrt{-h_{k+j-1}(\tilde{y}_{k+j-1})}.
\end{equation*}
Then, property \eqref{eq:H3} follows by setting $p=b+b_{\rho}$ and $\mathcal{I} = \{-N+2,\ldots,N+1\}$. 

Property \eqref{eq:H4} is obtained as follows. From Lemma \ref{lem:technical_1}, we have that $\|\hat{u}_k-\xk\|=\|\hat{y}_k-\Uk^T\xk\|\leq \sqrt{2\alpha_{\max}\mu \left(1+\frac{\tau}{2}\right)(-h_k(\tilde{y}_k))}$; since equation \eqref{eq:summable_h} implies $\lim_{k\rightarrow \infty} h_k(\tilde{y}_k)=0$, it must be $\lim_{k\rightarrow \infty} \|\hat{u}_k-\xk\|=0$ and the first limit in \eqref{eq:H4} holds. The second limit in \eqref{eq:H4} trivially follows by continuity of the function $\F$.

As for property \eqref{eq:H5}, it suffices to note that 
$$
\|\xkk-\xk\|\leq \|d_k\|\leq \sqrt{2\alpha_{\max}\mu}\left(\sqrt{1+\frac{\tau}{2}}+\sqrt{\frac{\tau}{2}}\right)\sqrt{-h_k(\tilde{y}_k)},
$$
which follows by {\sc Step 6} of Algorithm \ref{alg:block-VMILA} and Lemma \ref{lem:technical_1}, and set $q = \sqrt{2\alpha_{\max}\mu}\left(\sqrt{1+\frac{\tau}{2}}+\sqrt{\frac{\tau}{2}}\right)$.

Since properties from \eqref{eq:H1} to \eqref{eq:H5} hold, $\{x_k\}_{k\in\N}$ is bounded by \textcolor{black}{Assumption \ref{ass4}}, and $\{\rho_k\}_{k\in\N}$ converges as concluded in \eqref{eq:limit_rhok}, the assumptions of Theorem \ref{thm:abstract_convergence} hold, and thus the iterates $\{(\xk,\rho_k)\}_{k\in\N}$ converge to a stationary point $(x^*,0)\in\R^n\times \R$ of $\F$. Since $\nabla \F(x,\rho) = (\nabla F(x),\rho)$ for all $(x,\rho)\in \R^{n}\times \R$, it follows that the limit point $x^*$ of $\{\xk\}_{k\in\N}$ must be stationary for $F$, and hence the proof is complete.
\end{proof}

In the following, we are interested in deriving convergence rates on the function values along the iterates, by assuming that the merit function $\Psi$ defined in \eqref{eq:psi} satisfies the KL property at the limit point with $\phi(t) = ct^{1-\theta}$, being $c>0$ and $\theta\in [0,1)$ depending on the point of interest. This type of result is quite common in the KL literature, see e.g. \cite{chorobura2023random,Chouzenoux-Pesquet-Repetti-2016,Frankel-Garrigos-Peypouquet-2015}. 

Our analysis relies on the following Lemma, which contains convergence rates for a non-negative non-increasing sequence complying with a certain recurrence relation. Such a result is similar to \cite[Lemma 8]{Necoara-2025}, which regards any positive decreasing sequence $\{\Delta_k\}_{k\in\N}$ satisfying $\Delta_{k}^{1+\zeta}\leq \tilde{c}(\Delta_k-\Delta_{k+1})$ with $\tilde{c}>0$, $\zeta>-1$. In our result, we replace $\Delta_k-\Delta_{k+1}$ with $\Delta_{k-1}-\Delta_{k+1}$, which is needed to take into account the presence of an inertial term in Algorithm \ref{alg:block-VMILA}.

\begin{lemma}\label{lem:attouch_modified}
Let $\{\Delta_k\}_{k\in\N}\subseteq \R$ be a non-negative, monotone non-increasing sequence such that $\lim\limits_{k\rightarrow \infty} \Delta_k = 0$. Suppose there exist $\theta\in(0,1)$ and $\tilde{c}>0$ such that the following inequality holds:
\begin{equation}\label{eq:crucial_for_rate}
    \Delta_{k+1}^{2\theta} \leq \tilde{c}(\Delta_{k-1}-\Delta_{k+1}), \quad \forall \ k\in\N.
\end{equation}
Then, the following statements hold true.
\begin{itemize}
    \item[(i)] If $\theta\in (\frac{1}{2},1)$, there exists $C>0$ such that, for all $k\in\N$, we have
    \begin{equation*}
        \Delta_k\leq C k^{-\frac{1}{2\theta-1}}.
    \end{equation*}
    \item[(ii)] If $\theta \in (0,\frac{1}{2}]$, there exists $C>0$, $\omega\in(0,1)$ such that, for all $k\in\N$, we have
    \begin{equation*}
        \Delta_k\leq C\omega^k.
    \end{equation*}
\end{itemize}
\end{lemma}

\begin{proof}
See Appendix \ref{app:proof_of_lemma}.
\end{proof}

The convergence rates under the KL assumption on $\Psi$ are stated in the following Theorem. 

\begin{theorem}\label{thm:convergence_rate}
Suppose the assumptions of Corollary \ref{cor:convergence_iterates} are satisfied and let $x^*\in\R^n$ be the unique limit point of the sequence $\{\xk\}_{k\in\N}$ generated by Algorithm \ref{alg:block-VMILA}. Assume that the function $\Psi:\R^n\times\cdots\times \R^n\rightarrow \R$ defined in \eqref{eq:psi} satisfies the KL property at $(x^*,\ldots,x^*)$ with $\xi(t) = c t^{1-\theta}$, with $c>0$, $\theta\in [0,1)$. Then the following convergence rates hold.
	\begin{itemize}
		\item[(i)] If $\theta=0$, the sequence $\{x_k\}_{k\in\N}$ terminates in a finite number of iterations.
		\item[(ii)]If $\theta\in(0,\frac{1}{2}]$, there exist $C>0 $, $\omega\in(0,1)$ such that
	\begin{equation}\label{eq:linear}
			F(x_{kN})-F(x^*)\leq C \omega^k.
		\end{equation}
		\item[(iii)] If $\theta\in(\frac{1}{2},1)$, there exists $C>0$ such that
		\begin{equation}\label{eq:sublinear}
			F(x_{kN})-F(x^*)\leq Ck^{-\frac{1}{2\theta-1}}.
		\end{equation}
	\end{itemize}
\end{theorem}

\begin{proof}
Regarding item (i), we assume by contradiction that there exists an infinite subset of indices $\{k_j\}_{j\in\N}$ such that $x_{k_j+1}\neq x_{k_j}$ for all $j\in\N$. From property \eqref{eq:H5}, we have that $r_{k_j}>0$ for all $j\in\N$. Since $\lim_{k\rightarrow \infty}(-h_k(\tilde{y}_k))=0$ due to \eqref{eq:summable_h}, it follows by \eqref{eq:ine_phila_1} that $\lim_{k\rightarrow \infty} \|\hat{u}_k-x_k\|=\lim_{k\rightarrow \infty}\|\hat{y}_k-\Uk^T\xk\| =0$. Then, since Corollary \ref{cor:convergence_iterates} guarantees that $\lim_{k\rightarrow \infty}x_k=x^*$, we have $\lim_{k\rightarrow \infty}\hat{u}_k=x^*$. Let $U$ and $\nu$ be such that Definition \ref{def:KL} is satisfied for the function $\F$ defined in \eqref{eq:matF} at the point $(x^*,0)$. Since $\{\hat{u}_k\}_{k\in\N}$ is converging to $x^*$, $\{\rho_k\}_{k\in\N}$ is converging to zero and $\F(\hat{u}_{k_j},\rho_{k_j})>\F(x^*,0)$ thanks to \eqref{eq:left_hand_almost_there}, we have $(\hat{u}_{k_j},\rho_{k_j})\in U\cap \{(x,\rho)\in\R^n\times \R: \ \F(x^*,0) < \F(x,\rho) < \F(x^*,0)+\nu\}$ for all sufficiently large $j\in\N$. By an abuse of notation, we assume such an inclusion holds for all $j\in\N$. Then, the KL inequality can be evaluated at the point $(\hat{u}_{k_j},\rho_{k_j})$, namely
\begin{equation}\label{eq:gradient_lower_bound}
    \|\nabla \F(\hat{u}_{k_j},\rho_{k_j})\|\geq \frac{1}{\xi'(\F(\hat{u}_{k_j},\rho_{k_j})-\F(x^*,0))}=\frac{1}{c}, \quad \forall \ j\in\N,
\end{equation}
where the equality follows from $\xi(t)=ct^{1-\theta}=ct$. By combining \eqref{eq:gradient_lower_bound} and property \eqref{eq:H3}, we obtain
\begin{equation}\label{eq:r_lower_bound}
\sum_{i\in\mathcal{I}}r_{k_j+1-i}\geq \frac{1}{pc}, \quad \forall \ j\in\N.
\end{equation}
Squaring both sides of equation \eqref{eq:r_lower_bound} and applying Jensen's inequality leads to
\begin{equation}\label{eq:sumr_lower_bound}
\sum_{i\in\mathcal{I}}r^2_{k_j+1-i}\geq \frac{1}{|\mathcal{I}|p^2c^2}, \quad \forall \ j\in \N.    
\end{equation}
By summing property \eqref{eq:H1} with $\Psi$ introduced in \eqref{eq:psi}, $k=k_j+1-i$ and for $i\in \mathcal{I}$, and using \eqref{eq:sumr_lower_bound}, we get the following inequality
\begin{equation}\label{eq:absurd}
\frac{a}{|\mathcal{I}|p^2c^2}\leq \sum_{i\in \mathcal{I}}(\Psi(x_{k_j+1-i},\ldots,x_{k_j+1-i-N})-\Psi(x_{k_j+2-i},\ldots,x_{k_j+2-i-N})), \quad \forall \ j\in\N. 
\end{equation}
On the other hand, since $\lim_{k\rightarrow \infty} x_{k}=x^*$ and $\Psi$ is a continuous function, we have that $$
\lim_{j\rightarrow \infty}(\Psi(x_{k_j+1-i},\ldots,x_{k_+1-i-N})-\Psi(x_{k_j+2-i},\ldots,x_{k_j+2-i-N}))=0,
$$
which contradicts \eqref{eq:absurd}. Then item (i) follows.

We now consider items (ii) and (iii). From Corollary \ref{cor:convergence_iterates}, there exists a unique point $x^*\in\R^n$ such that $\lim_{k\rightarrow \infty}x_k=x^*$. Let $U$ and $\nu$ be such that Definition \ref{def:KL} is satisfied for the function $\Psi$ at the point $(x^*,\ldots,x^*)$. Thanks to Lemma \ref{lem:descent} and Corollary \ref{cor:convergence_iterates}, we have $(x_k,x_{k-1},\ldots,x_{k-N})\in U \cap \{y\in\R^n\times \ldots \times \R^n: \ \Psi(x^*,\ldots,x^*)< \Psi(y) < \Psi(x^*,\ldots,x^*)+\nu\}$ for all sufficiently large $k\in\N$; with abuse of notation, we will assume such an inclusion holds for all $k\in\N$. Then, the KL inequality can be evaluated at the point $(x_k,x_{k-1},\ldots,x_{k-N})$, i.e., 
\begin{equation*}
\xi'(\Psi(x_k,x_{k-1},\ldots,x_{k-N})-\Psi(x^*,x^*,\ldots,x^*))\|\nabla \Psi(x_k,x_{k-1},\ldots,x_{k-N})\|\geq 1.
\end{equation*}
By recalling that $\xi(t) = c t^{1-\theta}$, and squaring both sides of the above inequality, we obtain
\begin{equation}\label{eq:KL_Psi}
(\Psi(x_k,x_{k-1},\ldots,x_{k-N})-\Psi(x^*,x^*,\ldots,x^*))^{2\theta}\leq c^2 \|\nabla \Psi(x_k,x_{k-1},\ldots,x_{k-N})\|^2.    
\end{equation}
Note that
\begin{equation*}
\nabla \Psi(z_1,z_2,\ldots,z_{N+1}) = \left(\begin{array}{c}
       \nabla F(z_1)+\gamma(z_1-z_2)\\
        \gamma({z_2-z_1}+z_2-z_3)\\
        \vdots\\
        \gamma(z_N-z_{N-1}+z_N-z_{N+1})\\
        \gamma(z_{N+1}-z_N)
\end{array}\right).
\end{equation*}
Therefore, from inequality \eqref{eq:KL_Psi} we get 
\begin{align*}
(\Psi(x_k,x_{k-1},\ldots&,x_{k-N})-\Psi(x^*,x^*,\ldots,x^*))^{2\theta}\leq c^2 \left(\|\nabla F(\xk)\|+2\gamma \sum_{j=1}^{N}\|x_{k+1-j}-x_{k-j}\|  \right)^2\\
&\leq 2c^2\|\nabla F(\xk)\|^2+8c^2\gamma^2N\sum_{j=1}^{N}\|x_{k+1-j}-x_{k-j}\|^2\\
& \leq 2c^2\|\nabla F(\xk)\|^2+16c^2\alpha_{\max}\mu\left(\sqrt{1+\frac{\tau}{2}}+\sqrt{\frac{\tau}{2}}\right)^2\gamma^2 N\sum_{j=1}^{N}(-h_{k-j}(\tilde{y}_{k-j})),
\end{align*}
where the second inequality follows by  
applying relation $(u+v)^2\leq 2u^2+2v^2$ and Jensen's inequality, and the third inequality holds by combining \eqref{eq:norm2} and \eqref{eq:ine_phila_3}. We continue by applying \eqref{eq:ine_squared} and Lemma \ref{lem:gradient_bound} to the above inequality, thus obtaining
\begin{align*}
&(\Psi(x_k,x_{k-1},\ldots,x_{k-N})-\Psi(x^*,x^*,\ldots,x^*))^{2\theta}\leq 4b^2c^2\sum_{j=-(N-1)}^N(-h_{k+j-1}(\tilde{y}_{k+j-1}))\\
&+8N^2L_f^2\alpha_{\max}\mu c^2\left(1+\frac{\tau}{2}\right)(-h_k(\tilde{y}_k))+16c^2\alpha_{\max}\mu\left(\sqrt{1+\frac{\tau}{2}}+\sqrt{\frac{\tau}{2}}\right)^2\gamma^2N\sum_{j=1}^{N}(-h_{k-j}(\tilde{y}_{k-j})).
\end{align*}
In conclusion, there exists a sufficiently large $\tilde{b}>0$ such that 
\begin{equation*}
(\Psi(x_k,x_{k-1},\ldots,x_{k-N})-\Psi(x^*,x^*,\ldots,x^*))^{2\theta}\leq \tilde{b}\sum_{j=-(N-1)}^N(-h_{k+j-1}(\tilde{y}_{k+j-1})),    
\end{equation*}
and an application of Lemma \ref{lem:descent} implies
\begin{align*}
&(\Psi(x_k,x_{k-1},\ldots,x_{k-N})-\Psi(x^*,x^*,\ldots,x^*))^{2\theta}\\
&\leq \frac{\tilde{b}}{\textcolor{black}{\vartheta} \lambda_{\min}}\sum_{j=-(N-1)}^N(\Psi(x_{k+j-1},x_{k+j-2},\ldots,x_{k+j-1-N})-\Psi(x_{k+j},x_{k+j-1},\ldots,x_{k+j-N}))\\
&= \frac{\tilde{b}}{\textcolor{black}{\vartheta} \lambda_{\min}}(\Psi(x_{k-N},x_{k-N-1},\ldots,x_{k-2N})-\Psi(x_{k+N},x_{k+N-1},\ldots,x_{k})).
\end{align*}
Setting $k=\tilde{k}N$ with $\tilde{k}\in\mathbb{N}$ yields
\begin{align}
&(\Psi(x_{\tilde{k}N},x_{\tilde{k}N-1},\ldots,x_{(\tilde{k}-1)N})-\Psi(x^*,x^*,\ldots,x^*))^{2\theta}\nonumber\\
&\leq \frac{\tilde{b}}{\textcolor{black}{\vartheta} \lambda_{\min}}(\Psi(x_{(\tilde{k}-1)N},x_{(\tilde{k}-1)N-1},\ldots,x_{(\tilde{k}-2)N})-\Psi(x_{(\tilde{k}+1)N},x_{(\tilde{k}+1)N-1},\ldots,x_{\tilde{k}N}).\label{eq:tech_rate_1}
\end{align}
By defining the following sequence
\begin{equation}\label{eq:Delta_Psi}
\Delta_{\tilde{k}} = \Psi(x_{\tilde{k}N},x_{\tilde{k}N-1},\ldots,x_{(\tilde{k}-1)N})-\Psi(x^*,x^*,\ldots,x^*), \quad \forall \ \tilde{k}\in\N,    
\end{equation}
we can rewrite inequality \eqref{eq:tech_rate_1} as follows
\begin{equation}\label{eq:tech_rate_2}
\Delta_{\tilde{k}}^{2\theta}\leq \frac{\tilde{b}}{\textcolor{black}{\vartheta}\lambda_{\min}}(\Delta_{\tilde{k}-1}-\Delta_{\tilde{k}+1}), \quad \forall \ \tilde{k}\in\N.    
\end{equation}
Since $\{\Delta_{\tilde{k}}\}_{\tilde{k}\in\N}$ is monotone non increasing by Lemma \ref{lem:descent}, we have $\Delta_{\tilde{k}+1}^{2\theta}\leq \Delta_{\tilde{k}}^{2\theta}$. Hence, by setting $\tilde{c}=\tilde{b}/(\textcolor{black}{\vartheta}\lambda_{\min})$, inequality \eqref{eq:tech_rate_2} yields
\begin{equation*}
\Delta_{\tilde{k}+1}^{2\theta}\leq \tilde{c}(\Delta_{\tilde{k}-1}-\Delta_{\tilde{k}+1}), \quad \forall \ \tilde{k}\in\N.   
\end{equation*}
At this point, items (ii) and (iii) follow by applying Lemma \ref{lem:attouch_modified} and observing that
\begin{equation*}
\Delta_{\tilde{k}}= \Psi(x_{\tilde{k}N},x_{\tilde{k}N-1},\ldots,x_{(\tilde{k}-1)N})-F(x^*)\geq F(x_{\tilde{k}N})-F(x^*). 
\end{equation*}
\end{proof}

\section{Numerical experiments}\label{sec:numerical_experiments}
\textcolor{black}{In this section, we assess the effectiveness of Block-PHILA on two imaging inverse problems formulated within the GS Plug-and-Play framework: image deblurring and image super-resolution. Specifically, we consider the following GS PnP optimization problem 
\begin{equation}\label{eq:GS_PNP_problem}
    \underset{x\in\R^n}{\operatorname{argmin}} \ F(x)\equiv \mathcal{D}(x) +\frac{\textcolor{black}{\varrho}}{2}\|x-N_{\sigma}(x)\|^2,
\end{equation}
where the data fidelity term $\mathcal{D}:\mathbb{R}^n\rightarrow\mathbb{R}$ is continuously differentiable and bounded from below. Its specific expression depends on the noise corrupting the data. Hereafter we consider both additive Gaussian noise (Section \ref{sec:image_deblurring} and Section \ref{sec:image_SR}) and Cauchy noise (Section \ref{sec:largescale}). If $\mathcal{D}$ is convex along the coordinates on $\mathbb{R}^n$, the optimization problem \eqref{eq:GS_PNP_problem} can be numerically solved by Block-PHILA adopting the following splitting of the objective function:
\begin{equation}\label{eq:splitting_PnP}
    \left\{
    \begin{array}{l}
    \displaystyle
    \phi(x) = \mathcal{D}(x) \\
    \\
    \displaystyle
     f(x) =  \frac{\textcolor{black}{\varrho}}{2}\|x-N_{\sigma}(x)\|^2  \\
    \end{array}
    \right..
\end{equation} } 
\textcolor{black}{However, since the objective function in \eqref{eq:GS_PNP_problem} is continuously differentiable, regardless of whether $\mathcal{D}$ is convex along coordinates or not, the same problem can alternatively be cast into formulation \eqref{eq:problem} by treating the whole objective as a smooth term, namely: 
\begin{equation}\label{eq:all_diff_splitting_PnP}
    \left\{
    \begin{array}{l}
    \displaystyle
    \phi(x) = 0\\
    \\
    \displaystyle
     f(x) = \mathcal{D}(x)
     + \frac{\textcolor{black}{\varrho}}{2}\|x-N_{\sigma}(x)\|^2  \\
    \end{array}
    \right..
\end{equation} 
When possible, comparing the two splittings \eqref{eq:splitting_PnP} and \eqref{eq:all_diff_splitting_PnP} allows us to quantify the benefit of retaining the forward--backward structure, as opposed to performing a plain gradient step on the whole objective function. In addition, we provide a preliminary analysis on the GPU memory required to compute the block GS denoiser with respect to the image size (Section \ref{subsec:memory_consumption}). }

\subsection{\textcolor{black}{Experimental setting}}\label{sec:setting} 
\textcolor{black}{This section outlines the experimental setup, implementation details, and baseline methods considered in our numerical evaluation.}

\paragraph{{GS denoiser}} The function $N_\sigma$ is the UNet proposed in \cite{Zhang-et-al-2021} and subsequently trained within the context of GS denoiser in \cite{hurault2022gradient}. We remark that the authors of \cite{hurault2022gradient} replaced the ReLU activation functions in the original UNet in \cite{Zhang-et-al-2021} with eLU functions in order to ensure the differentiability of $N_\sigma$ with respect to its input. Hence, in our tests, we use the network's weights released with the code of \cite{hurault2022gradient} within the DeepInverse library \cite{tachella2025deepinverse}. For the restricted denoiser, we selected the patches of the image extending them with a padding contour of size 16 pixels. {Finally, we remark that under this choice $\phi$ and $f$ in \eqref{eq:splitting_PnP} readily satisfy Assumption \ref{ass1}.}

\paragraph{{The Block-PHILA variants}} {Block-PHILA is determined by four independent design choices, each with two possible options.
}

\begin{itemize}
    \item {\emph{The splitting of the objective function:}
    this choice is denoted by {FB} when adopting the forward--backward splitting \eqref{eq:splitting_PnP}, and by {GD} when employing the all-smooth splitting \eqref{eq:all_diff_splitting_PnP}.}
    \item {\emph{The steplength selection rule:} this choice is denoted by {BB} when $\alpha_k$ is computed adaptively as the geometric mean of the two well-known Barzilai–Borwein rules \cite{An-etal-2025}, i.e.,   }
    \begin{equation}\label{eq:alpha_BB_geom}
    \alpha_k = \max\left\{\alpha_{\min},\min\left\{\alpha_{\max},\frac{\|U_{i_k}^T(x_k - x_{k-N})\|_{D_k^{-1}}}{\|U_{i_k}^T(\nabla f(x_k) - \nabla f(x_{k-N}))\|_{D_k^{-1}}}\right\}\right\}, 
    \end{equation}
    {and by {C} when the steplength is kept constant throughout the iterations, namely $\alpha_k=1/\varrho$.}  
    \item {\emph{The inertial term:} this choice is denoted by the suffix {I} and the inertial parameters correspond to \cite{Chambolle-etal-2015}:}    \begin{equation}\label{eq:beta_FISTA}
        \beta_k = \frac{\mathrm{div}(k,N)-1}{\mathrm{div}(k,N)+2},
    \end{equation}
    {where $\mathrm{div}(\cdot,\cdot)$ denotes the integer division, together with $\beta_{\max} = 1$ and  $\gamma = 10^{-4}$.  The absence of the suffix means that the inertial term is discarded, i.e.  $\beta_{k} = 0$ and  $\gamma = 0$. Although rule \eqref{eq:beta_FISTA} is usually chosen for Nesterov-like accelerated methods, it has also been employed for Heavy-Ball-like inertial terms, see e.g. \cite{Pock-Sabach-2018}.} 
\end{itemize}
{The combination of these three binary choices yields the eight variants of Block-PHILA listed in Table \ref{tab:variants}. 
Comparing these eight settings allows us to isolate the contribution of each single ingredient of the proposed framework. {Finally, we remark that the scaling matrix is set to the identity matrix in most of the experiments. The only exception is in Section \ref{sec:variable_metric}, where the use of a variable metric is denoted by adding the suffix {VM} to the names reported in the first column of Table \ref{tab:variants}.} 
}

\begin{table}
\centering
\renewcommand{\arraystretch}{1.2}
\begin{tabular}{l l l l l}
\toprule
variant & splitting & steplength $\alpha_k$ & inertia $\beta_k$ & $N=1$ reduces to \\
\midrule
{FB-BB-I} & \eqref{eq:splitting_PnP} & \eqref{eq:alpha_BB_geom} & \eqref{eq:beta_FISTA} & PHILA \cite{Bonettini-Prato-Rebegoldi-2024}\\
{FB-BB}   & \eqref{eq:splitting_PnP} & \eqref{eq:alpha_BB_geom} & $0$ & VMILA \cite{Bonettini-Loris-Porta-Prato-2016}\\
{FB-C-I}  & \eqref{eq:splitting_PnP} & $1/\varrho$ & \eqref{eq:beta_FISTA} & PHILA \cite{Bonettini-Prato-Rebegoldi-2024}\\
{FB-C}    & \eqref{eq:splitting_PnP} & $1/\varrho$ & $0$ & VMILA \cite{Bonettini-Loris-Porta-Prato-2016}\\
\midrule
{GD-BB-I} & \eqref{eq:all_diff_splitting_PnP} & \eqref{eq:alpha_BB_geom} & \eqref{eq:beta_FISTA} & full-gradient PHILA\\
{GD-BB}   & \eqref{eq:all_diff_splitting_PnP} & \eqref{eq:alpha_BB_geom} & $0$ & full-gradient VMILA\\
{GD-C-I}  & \eqref{eq:all_diff_splitting_PnP} & $1/\varrho$ & \eqref{eq:beta_FISTA} & full-gradient PHILA\\
{GD-C}    & \eqref{eq:all_diff_splitting_PnP} & $1/\varrho$ & $0$ & full-gradient VMILA\\
\bottomrule
\end{tabular}
\caption{The variants of Block-PHILA compared in the numerical experiments. Each variant is identified by the three independent design choices.}
\label{tab:variants}
\end{table}

{As reported in the last column of Table \ref{tab:variants}, for $N=1$ the variants { FB-BB-I} and {FB-C-I} reduce to the PHILA approach proposed in \cite{Bonettini-Prato-Rebegoldi-2024} equipped with two particular hyperparameter settings. Similarly, {FB-BB} and {FB-C} are special instances of the VMILA method developed in \cite{Bonettini-Loris-Porta-Prato-2016}. The variants {GD-$\ast$} reduce, for $N=1$, to the corresponding full-gradient line-search schemes obtained by setting $\phi\equiv 0$.}

{The remaining parameters are fixed across all versions as follows: $\alpha_{min} = 10^{-2}$, $\alpha_{max} = 10^{3}$,  $\delta = 0.5$, $\vartheta = 10^{-4}$ and $\tau = 10^{6}$.}

{\paragraph{Competitors} Block-PHILA is compared with three PnP approaches from the literature.
\begin{itemize}
    \item \textbf{GS-PnP} a state-of-the-art Plug-and-Play approach that has recently gained popularity for solving image restoration problems of the form \eqref{eq:GS_PNP_problem}. More specifically, GS-PnP is based on the forward-backward iteration \eqref{eq:PnP_original} applied to the splitting \eqref{eq:splitting_PnP}, and therefore minimizes exactly the same objective function as Block-PHILA. We used the default hyperparameters setting of the official implementation\footnote{\url{https://github.com/samuro95/GSPnP}}.
    \item \textbf{DPIR} \cite{Zhang-et-al-2021}, a half-quadratic splitting scheme in which the denoising strength is progressively decreased along a prescribed schedule and a fixed, small number of iterations is performed. We recall that the neural network architecture employed in this work is the one introduced in \cite{Zhang-et-al-2021} for DPIR, and subsequently retrained in \cite{hurault2022gradient} within the Gradient Step framework; the comparison is therefore carried out at comparable network capacity. We considered the default schedule and number of iterations of the released implementation\footnote{\url{https://github.com/cszn/DPIR}}.
    \item \textbf{BCRED} \cite{Sun-etal-2020}, a block-coordinate Regularization-by-Denoising method, which is the approach closest in spirit to the one proposed here, as it also updates one block of coordinates at a time in the presence of a non-separable denoiser. In this setting, we used the DnCNN\cite{zhang2017beyond} as the denoiser and we considered the hyperparameters setting of the original implementation\footnote{\url{https://github.com/wustl-cig/bcred}}.
\end{itemize}
    It is important to remark that, among the compared approaches, only GS-PnP minimizes the objective function \eqref{eq:GS_PNP_problem}. DPIR varies the denoising strength along the iterations, so that no fixed functional is being decreased, whereas BCRED targets a fixed point of the RED operator, for which an explicit objective function is available only under suitable condition on the denoiser \cite{Reehorst-etal-2018,Romano-etal-2017}. 
    Consequently, neither the objective function values nor the stopping criterion adopted for Block-PHILA and GS-PnP are meaningful for DPIR and BCRED. 
    \begin{remark}
    Algorithmically, BCRED would correspond to a specific instance of the GD variant of Block-PHILA (without line-search) only if the DnCNN adopted in \cite{Sun-etal-2020} were replaced by the GS denoiser $\denoiser_\sigma$ defined in \eqref{eq:D_sigma}.
    Indeed, with such a choice, the BCRED iteration reads
\begin{equation}\label{eq:bcred_iteration}
    x_{k+1} = x_k-\gamma \Uk\Uk^T\left(\nabla \phi(x_k)+\varrho\left(x_k-\denoiser_{\sigma}(x_k)\right)\right),
\end{equation}
where $\gamma>0$ is a fixed steplength and the regularization parameter of the RED model has been identified with $\varrho$. Recalling from \eqref{eq:denoising_step} that $x-\denoiser_\sigma(x)=\nabla g_\sigma(x)$, the vector between brackets in \eqref{eq:bcred_iteration} is nothing but $\nabla \phi(x_k)+\varrho\nabla g_\sigma(x_k)=\nabla F(x_k)$, so that \eqref{eq:bcred_iteration} is a block-coordinate gradient step on the whole objective function \eqref{eq:GS_PNP_problem}
    \end{remark}
}

{\paragraph{Hardware and reproducibility}
All the experiments in Sections \ref{sec:image_deblurring}, \ref{sec:image_SR} and \ref{sec:largescale} have been performed on a laptop equipped with a 12th Intel i7-1280P, 16 GB of RAM and an NVIDIA GeForce RTX 3050 Laptop GPU with 4 GB of dedicated memory.
On the other hand, the memory profiling in Section \ref{subsec:memory_consumption} has been carried out on a computational cluster having a NVIDIA H200 GPU, in order to check the memory occupation of single-block configurations when dealing with large size images.
The implementation\footnote{\url{https://github.com/sedaboni/BlockPHILA_PnP}} of all the proposed algorithms (and the competitors) mainly relies on PyTorch and the DeepInverse library \cite{tachella2025deepinverse}.
}

\subsection{GPU memory consumption}\label{subsec:memory_consumption}
{In this section, we compare the memory required to evaluate the block GS denoiser and its single-block counterpart with respect to the image size. 
For each image size and each number of blocks, we record the peak amount of GPU memory allocated during the computation of single gradient step within the Block-PHILA framework, by means of \eqref{eq:block_gradient_step}.
Figure \ref{fig:mem_consumption} reports the peak allocated memory as a function of the image size, for $N\in\{1,2,4,8,16\}$. As expected, for a fixed number of blocks the memory grows proportionally to the number of pixels, since the computational graph that has to be stored to evaluate $J_{N_\sigma}$ is proportional to the size of the input. More importantly, it is evident that at every image size the memory is reduced monotonically as $N$ grows. The practical consequence is made explicit by the hatched band of Figure \ref{fig:mem_consumption}. On a device with a $4$ GB memory budget, it is not possible to employ the single-block algorithm on large size images (such as the $2048\times 2048$), whereas Block-PHILA with $N=16$ stays below the same memory budget for all the sizes considered here.
}

\begin{figure}
    \centering
    \includegraphics[width=0.62\linewidth]{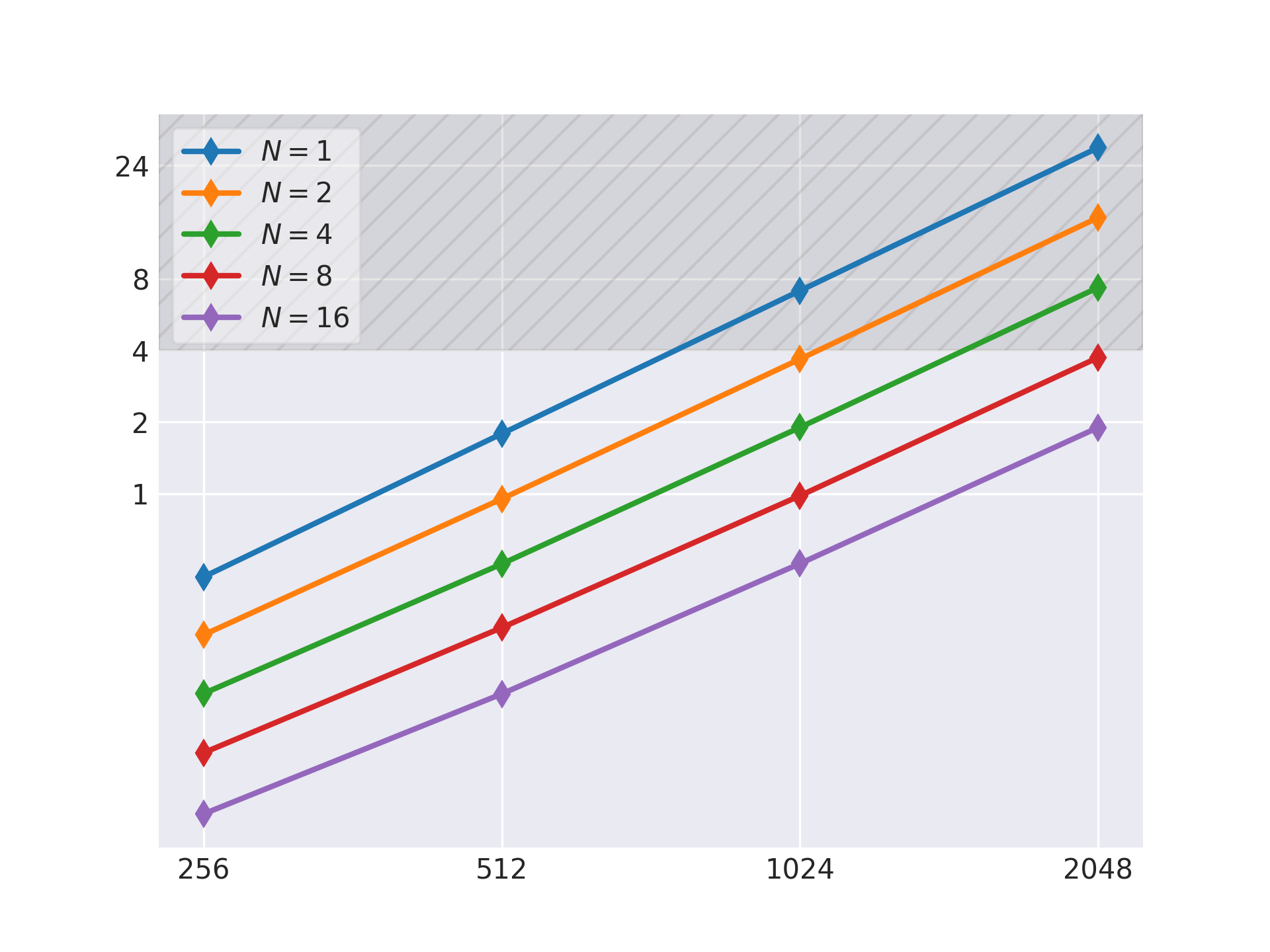}
    \caption{Peak GPU memory allocated during the computation of the gradient in one iteration of Block-PHILA, as a function of the image size and for an increasing number of blocks $N$. Both axes are in logarithmic scale. The hatched band marks the region above $4$ GB i.e. the amount of memory of the hardware used in all the other experiments.}
    \label{fig:mem_consumption}
\end{figure}

\subsection{Image deblurring}\label{sec:image_deblurring}
In this section, \textcolor{black}{we assume additive Gaussian noise with standard deviation $\nu\in\R^+$. As a consequence, the data fidelity term in \eqref{eq:GS_PNP_problem} has the form $\mathcal{D}(x) = \frac{1}{2}\|Ax-b\|^2$}, 
where $A = H$ is defined as a convolution operator with circular boundary conditions.  
In particular, the operator $H$ corresponds to a Gaussian blur kernel 
of size $25 \times 25$ with standard deviation $1.6$  \cite{hurault2022gradient,Romano-etal-2017}. As reference images, we consider those from the Set3C dataset 
(\textit{Butterfly}, \textit{Starfish}, and \textit{Leaves}) {having size $256\times 256$}.
  Furthermore, the observed data $b$ are assumed to be corrupted by Gaussian noise with noise level $\nu = 0.03$. By following \cite{hurault2022gradient}, where the authors consider the same test problems, we fixed $\sigma = 1.8\nu$ and $\textcolor{black}{\varrho} = 0.075$ in \eqref{eq:GS_PNP_problem}. All the compared algorithms were initialized using $x_0 = b$. The FB variants of Block-PHILA and GS-PnP require the computation of the proximal operator of the data fidelity term in the backward step. For $N=1$, a well-known closed-form solution can be applied, while for $N>1$, an inexact procedure is used in the implementation of \textsc{Step 3} of Block-PHILA. Further details are provided in Appendix \ref{app:prox_computation_exact} and Appendix \ref{app:prox_computation}, respectively.
  
  Table \ref{tab:deblur} presents the results achieved by Block-PHILA variants across different numbers of blocks, setting as
  stopping criterion
\begin{equation}\label{eq:stop_Crit}
\frac{|F(x_{k+1})-F(x_k)|}{|F(x_k)|} \leq \varepsilon,
\end{equation}
with $\varepsilon = 10^{-5}$. The same stopping criterion is used for GS-PnP. {Instead, both DPIR and BCRED are executed for a fixed number of iterations, namely 8 and 200, respectively}.
More in detail, Table \ref{tab:deblur} reports the values of the objective function (when available) and PSNR obtained by the different methods once \eqref{eq:stop_Crit} is satisfied, together with the corresponding number of iterations and computational time. For each combination of image and number of blocks, the best PSNR values are highlighted in bold, and the lowest computational times are underlined.
\begin{table}
\renewcommand{\arraystretch}{1.2}
\setlength{\tabcolsep}{5pt}
\adjustbox{max width=\textwidth}{%
\centering
\begin{tabular}{p{1cm} l *{3}{c c c c}}
\toprule
 &     & \multicolumn{4}{c}{\textit{Butterfly}} & \multicolumn{4}{c}{\textit{Leaves}} & \multicolumn{4}{c}{\textit{Starfish}} \\
 \cmidrule(lr){3-6}\cmidrule(lr){7-10}\cmidrule(lr){11-14}
  &  & $F$ & PSNR & itr & time & $F$ & PSNR & itr & time & $F$ & PSNR & itr & time \\
\midrule

\multirow{11}{*}{\parbox{1cm}{\centering \rotatebox{90}{$N=1$}}} & GS-PnP & 97.50 & 28.17 & 30 & 4.79 & 97.72 & 28.41 & 42 & 6.46 & 97.73 & 29.00 & 24 & 3.71 \\
 & DPIR & - & 27.23 & 8 & \underline{1.22} & - & 27.38 & 8 & \underline{0.89} & - & 28.67 & 8 & \underline{1.03} \\
 \cmidrule(l){2-14}  & BCRED & - & 26.84 & 200 & 21.51 & - & 25.86 & 200 & 21.38 & - & 27.79 & 200 & 20.01 \\
 & {FB-BB-I} & 97.58 & 28.04 & 8 & 1.86 & 97.93 & 27.91 & 8 & 1.49 & 97.76 & 29.12 & 11 & 8.29 \\
 & {FB-BB} & 97.50 & 28.22 & 16 & 3.64 & 97.72 & 28.46 & 23 & 14.27 & 97.73 & 29.03 & 15 & 9.35 \\
 & {FB-C-I} & 97.50 & \textbf{28.42} & 13 & 4.38 & 97.73 & \textbf{28.49} & 18 & 11.14 & 97.73 & \textbf{29.13} & 15 & 9.63 \\
 & {FB-C} & 97.50 & 28.17 & 30 & 15.08 & 97.72 & 28.41 & 42 & 26.14 & 97.73 & 29.00 & 24 & 15.58 \\
\cmidrule(l){2-14}  &{GD-BB-I} & 97.53 & 27.78 & 89 & 92.41 & 97.78 & 28.08 & 125 & 139.75 & 97.75 & 28.87 & 78 & 87.81 \\
 & {GD-BB} & 97.54 & 27.73 & 105 & 27.92 & 97.78 & 28.07 & 156 & 78.53 & 97.76 & 28.81 & 85 & 88.70 \\
 & {GD-C-I} & 97.71 & 26.89 & 28 & 7.11 & 97.86 & 27.66 & 60 & 34.69 & 97.78 & 28.76 & 40 & 38.01 \\
 & {GD-C} & 97.69 & 26.93 & 43 & 11.14 & 98.27 & 26.37 & 37 & 9.32 & 98.40 & 27.58 & 8 & 1.99 \\
\midrule
\multirow{9}{*}{\parbox{1cm}{\centering \rotatebox{90}{$N=2$}}} & BCRED & - & 26.44 & 200 & 18.94 & - & 25.47 & 200 & 27.66 & - & 27.65 & 200 & 50.39 \\
 & {FB-BB-I} & 97.59 & 28.02 & 15 & \underline{9.07} & 97.87 & 28.17 & 19 & \underline{11.71} & 97.75 & \textbf{29.10} & 22 & 12.40 \\
 & {FB-BB} & 97.51 & 28.18 & 27 & 15.12 & 97.73 & 28.41 & 41 & 22.29 & 97.75 & 28.93 & 20 & 10.79 \\
 & {FB-C-I} & 97.51 & \textbf{28.36} & 25 & 13.78 & 97.74 & \textbf{28.49} & 37 & 19.81 & 97.80 & 28.84 & 13 & \underline{7.17} \\
 & {FB-C} & 97.52 & 28.05 & 48 & 25.77 & 97.74 & 28.35 & 70 & 37.24 & 97.74 & 28.91 & 34 & 18.46 \\
\cmidrule(l){2-14}  &{GD-BB-I} & 97.56 & 27.57 & 140 & 133.94 & 97.84 & 27.72 & 170 & 166.24 & 97.77 & 28.68 & 96 & 93.07 \\
 & {GD-BB} & 97.57 & 27.46 & 156 & 97.38 & 97.86 & 27.64 & 200 & 190.56 & 97.78 & 28.63 & 110 & 106.03 \\
 & {GD-C-I} & 97.84 & 26.48 & 41 & 35.22 & 98.16 & 26.67 & 61 & 53.05 & 97.78 & 28.62 & 68 & 62.39 \\
 & {GD-C} & 98.71 & 25.42 & 15 & \underline{3.48} & 98.67 & 25.63 & 42 & \underline{9.42} & 97.90 & 28.19 & 46 & 31.56 \\
\midrule
\multirow{9}{*}{\parbox{1cm}{\centering \rotatebox{90}{$N=4$}}} & BCRED & - & 25.69 & 200 & 43.07 & - & 24.72 & 200 & 43.16 & - & 27.36 & 200 & 43.70 \\
 & {FB-BB-I} & 97.62 & 27.94 & 30 & \underline{14.98} & 97.86 & 28.11 & 37 & \underline{19.48} & 97.83 & 28.85 & 24 & \underline{11.73} \\
 & {FB-BB} & 97.54 & 27.96 & 37 & 17.11 & 97.76 & 28.26 & 56 & 26.16 & 97.75 & \textbf{28.93} & 39 & 17.97 \\
 & {FB-C-I} & 97.54 & \textbf{28.13} & 37 & 16.99 & 97.78 & \textbf{28.39} & 52 & 23.46 & 97.81 & 28.80 & 25 & 11.86 \\
 & {FB-C} & 97.53 & 27.82 & 71 & 34.98 & 97.78 & 28.10 & 96 & 44.27 & 97.75 & 28.86 & 59 & 27.10 \\
\cmidrule(l){2-14}  &{GD-BB-I} & 97.66 & 27.00 & 163 & 143.64 & 98.03 & 27.00 & 200 & 179.61 & 97.84 & 28.36 & 108 & 97.26 \\
 & {GD-BB} & 97.69 & 26.89 & 183 & 162.96 & 98.14 & 26.68 & 200 & 185.61 & 97.86 & 28.28 & 116 & 107.27 \\
 & {GD-C-I} & 98.56 & 25.60 & 29 & 22.10 & 98.28 & 26.33 & 98 & 76.85 & 97.96 & 28.02 & 52 & 41.36 \\
 & {GD-C} & 98.78 & 25.40 & 31 & 24.46 & 98.57 & 25.77 & 94 & 76.55 & 98.22 & 27.72 & 36 & 29.36 \\
\bottomrule
\end{tabular}}
\caption{Results achieved by the Block-PHILA variants and BCRED in solving the image deblurring problems by varying the number $N$ of blocks. For $N=1$ the results obtained by GS-PnP and DPIR are also included. For each image and value of $N$, the best PSNR is shown in bold and the lowest time is  underlined.}\label{tab:deblur}
\end{table}

The following considerations can be deduced from the analysis of Table \ref{tab:deblur}.
\begin{itemize}
    \item[(i)] For $N=1$, it is evident that DPIR requires the least computational time across all tested images. Nevertheless, the final PSNR it achieves is generally lower than that provided by the FB versions of Block-PHILA, regardless of the chosen number of blocks. As concerns 
    Block-PHILA, FB-C-I, FB-BB,  FB-BB-I  meets the stopping criterion earlier than the competitors, while still providing comparable or better values for both the objective function and PSNR. These results highlight the advantages, in terms of convergence speed, of adopting an adaptive steplength $\alpha_k$ and/or incorporating an inertial term, features absent in Block-PHILA-FB-C and GS-PnP. The acceleration of standard forward-backward algorithms through adaptive steplengths and inertial techniques is well established in the literature, with their effectiveness having been repeatedly observed in imaging tasks (see, for example, \cite{Bonettini-Loris-Porta-Prato-2016,Bonettini-Prato-Rebegoldi-2024,Bonettini-Rebegoldi-Ruggiero-2018,Chouzenoux-Pesquet-Repetti-2016} and references therein), and our findings provide further confirmation of these benefits. Moreover, for $N=1$, the FB versions of Block-PHILA, based on the splitting \eqref{eq:splitting_PnP}, typically outperform the GD variants, which instead exploit the differentiability of both the fidelity and regularization terms via the splitting \eqref{eq:all_diff_splitting_PnP}. This behaviour has already been observed in \cite{Combettes-etal-2019}.  
    \item[(ii)] For $N=2$ and $N=4$, Block-PHILA FB consistently outperforms Block-PHILA GD in terms of PSNR achieved when the stopping criterion is satisfied. Although Block-PHILA-GD-C sometimes satisfies the stopping criterion earlier than its forward-backward counterpart Block-PHILA-FB-C, this is not due to faster convergence; rather, it results from stalled progress in the objective function values between successive iterations. Indeed, the PSNR values provided by Block-PHILA GD-C are consistently the lowest.
    \item[(iii)] The performance of Block-PHILA for 
$N>1$ is comparable to that obtained for $N=1$ in terms of objective function reduction and  PSNR values. Notably, the block-wise structure facilitates the processing of large-scale images more efficiently, enabling scalable optimization without compromising performance. This design also significantly reduces GPU memory usage, making the method well-suited for high-resolution data and resource-constrained environments.
\end{itemize}
In all the block configurations, it is evident that the behavior of BCRED in terms of PSNR is comparable to that of the Block-PHILA GD variants, given the similarity of their underlying iterations. 
Similar conclusions can be drawn from Figure \ref{fig:deblur_graph}, which reports the PSNR  and objective function values achieved by some of the compared methods on the \textit{Leaves} image with respect to the computational time. In particular, panels (a)–(c) and (d)–(f) report the PSNR and objective function values obtained by Block-PHILA-FB-BB-I and Block-PHILA-GD-C-I, respectively, for different numbers of blocks, in a fixed time budget of 25 seconds. For $N=1$, the results obtained by GS-PnP are also included.  
\begin{figure}
    \centering
    \begin{subfigure}[b]{0.3\textwidth}
        \centering
        \includegraphics[width=\textwidth]{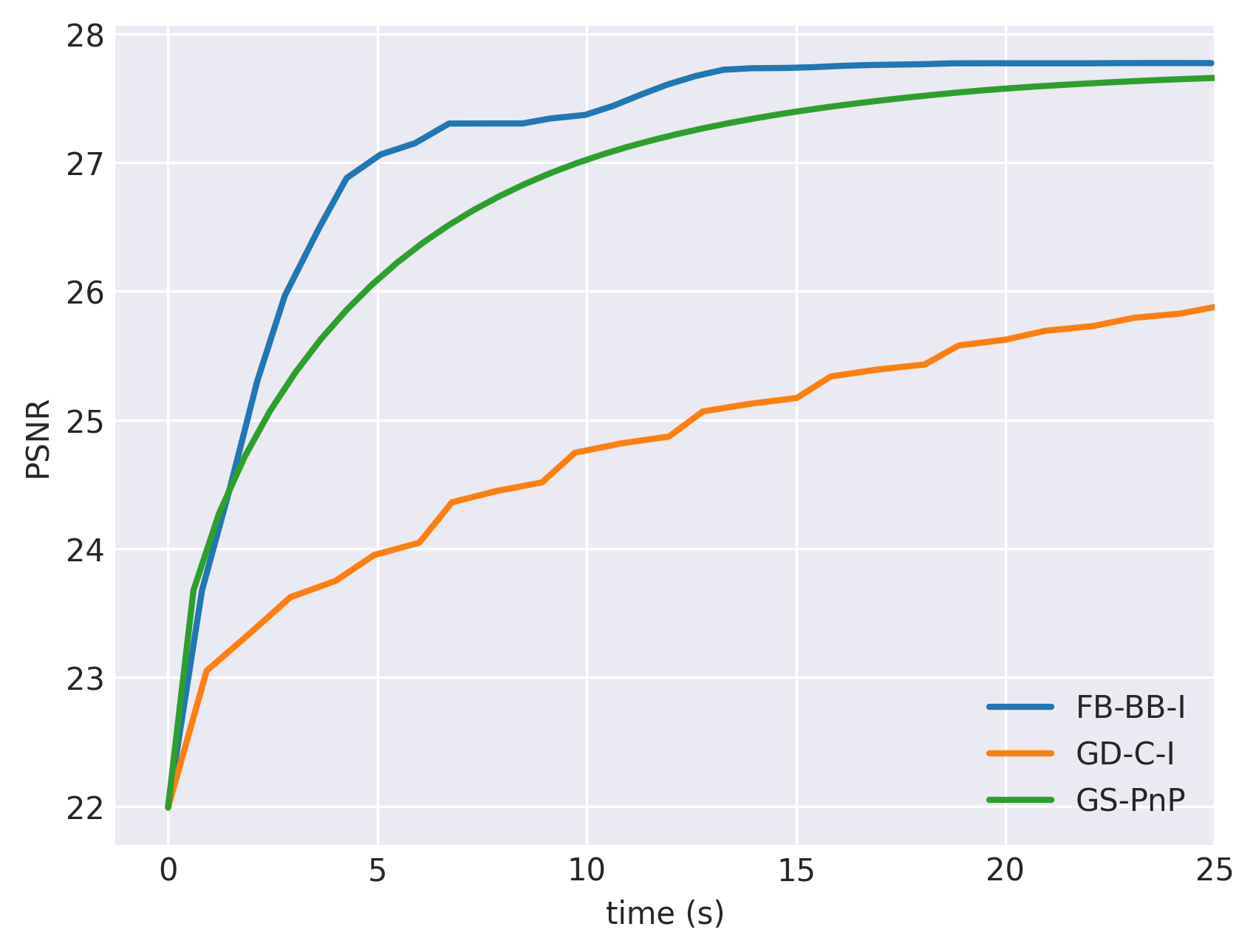}
        \caption{$N=1$}
    \end{subfigure}\hfill
    \begin{subfigure}[b]{0.3\textwidth}
        \centering
        \includegraphics[width=\textwidth]{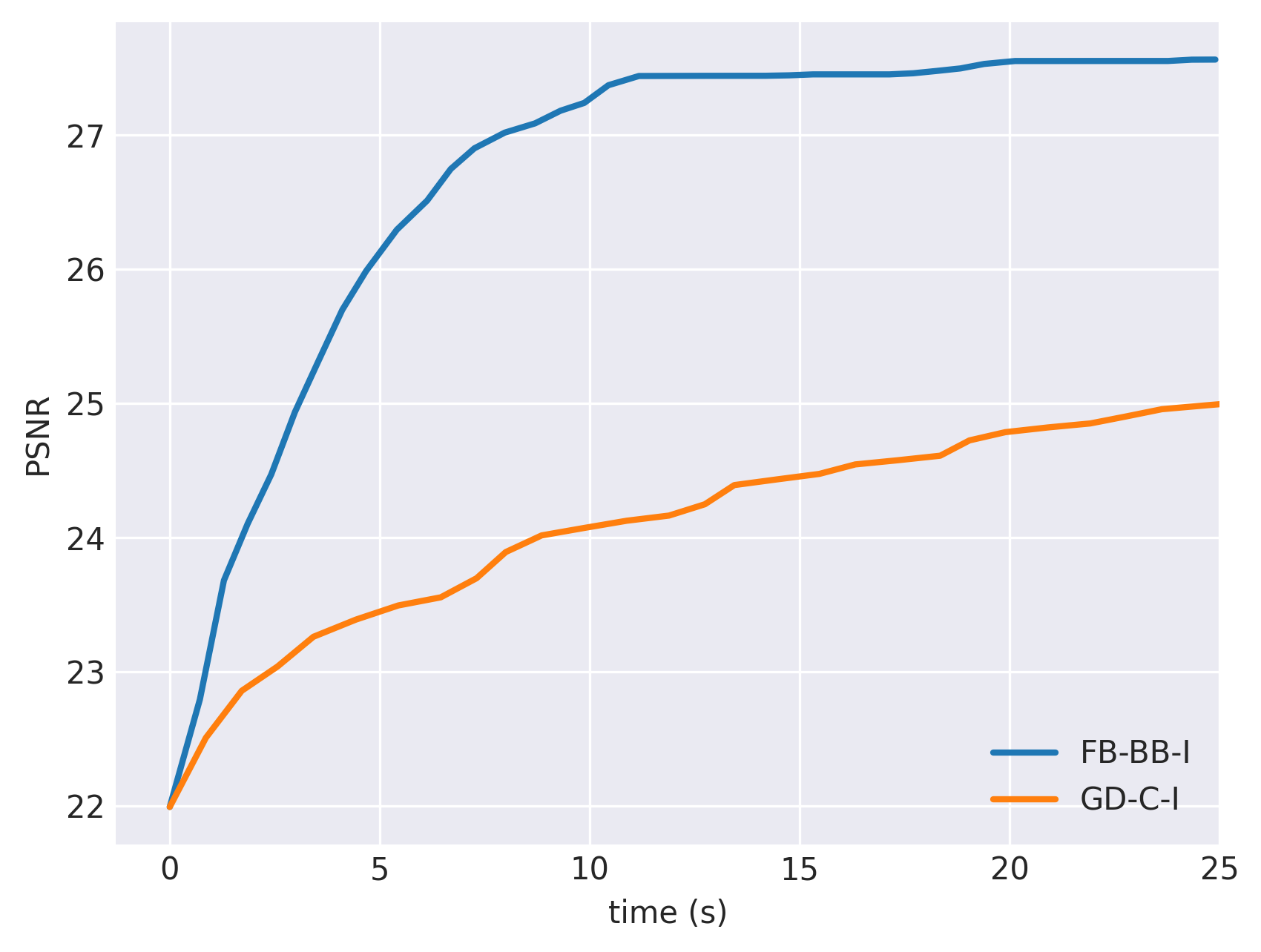}
        \caption{$N=2$}
    \end{subfigure}\hfill
    \begin{subfigure}[b]{0.3\textwidth}
        \centering
        \includegraphics[width=\textwidth]{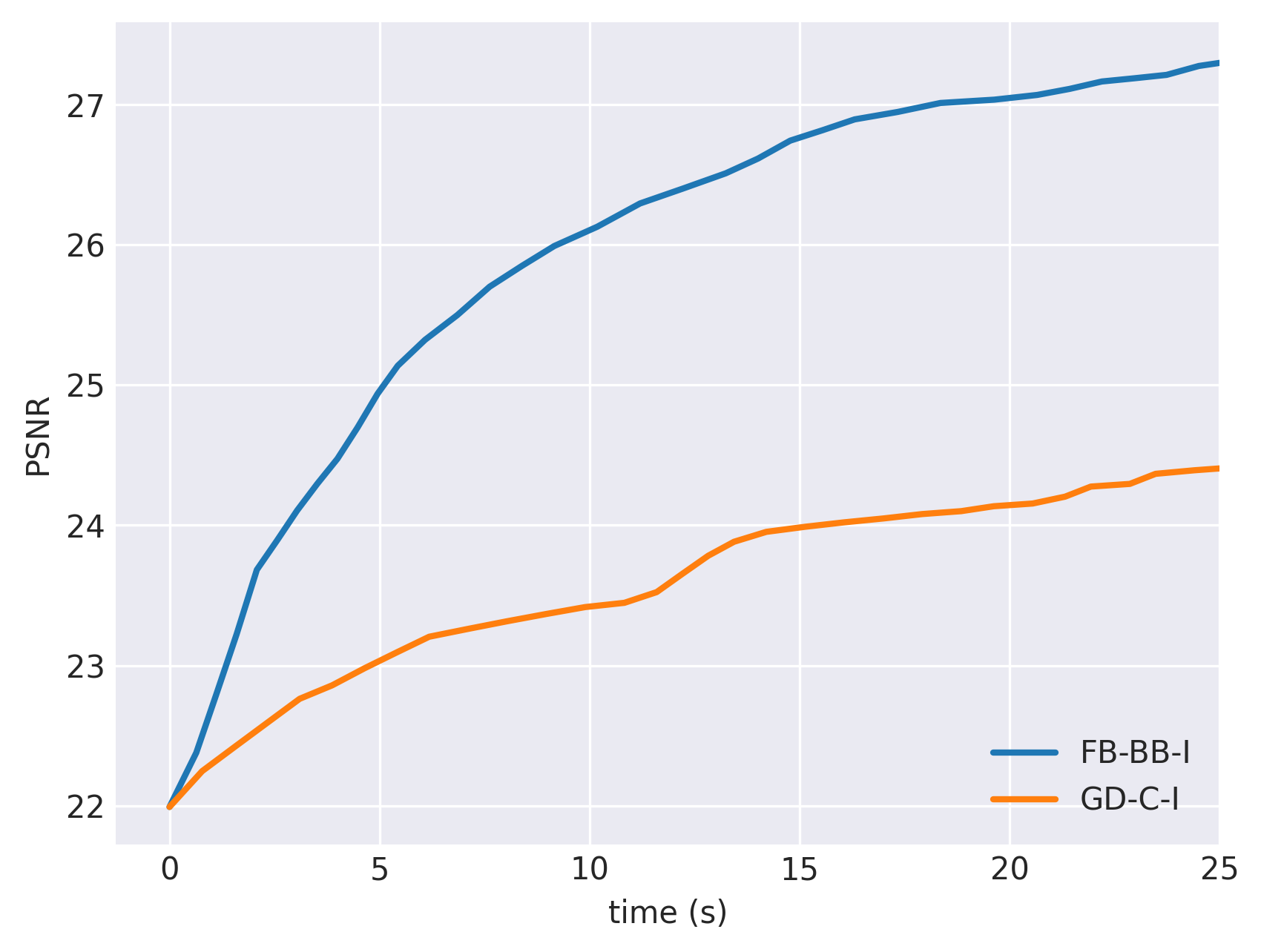}
        \caption{$N=4$}
    \end{subfigure}\\
    \begin{subfigure}[b]{0.3\textwidth}
        \centering
        \includegraphics[width=\textwidth]{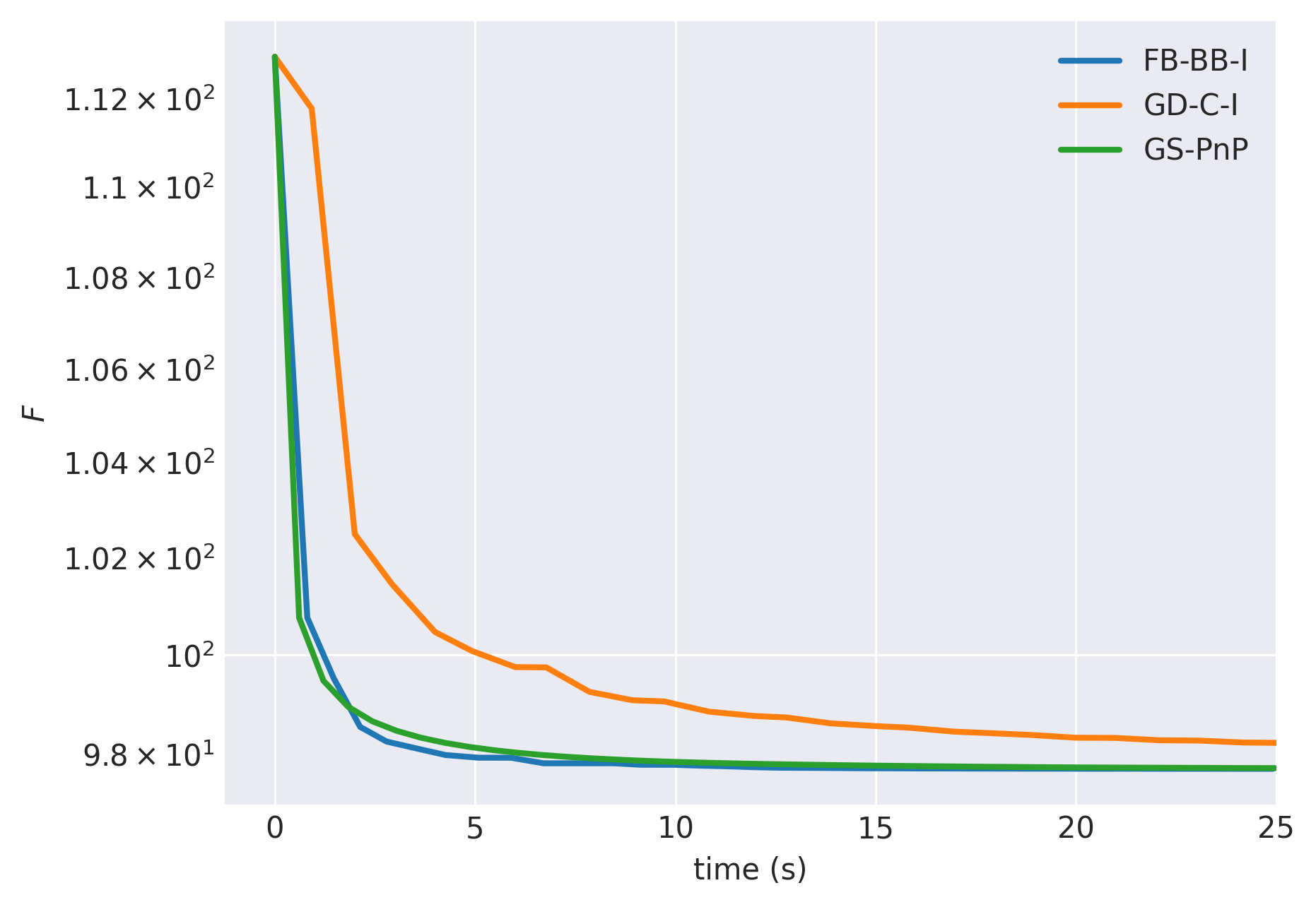}
        \caption{$N=1$}
    \end{subfigure}\hfill
    \begin{subfigure}[b]{0.3\textwidth}
        \centering
        \includegraphics[width=\textwidth]{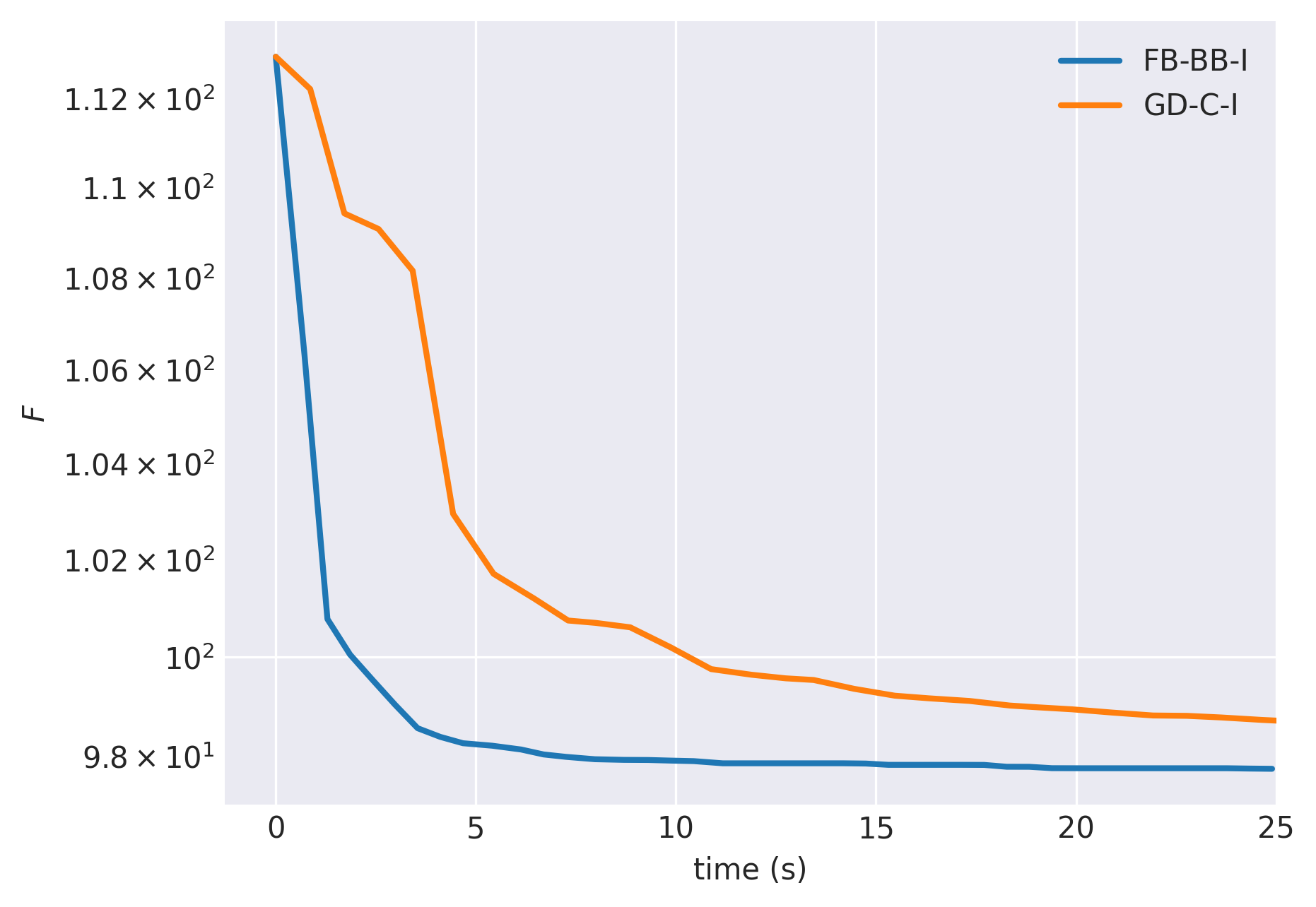}
        \caption{$N=2$}
    \end{subfigure}\hfill
    \begin{subfigure}[b]{0.3\textwidth}
        \centering
        \includegraphics[width=\textwidth]{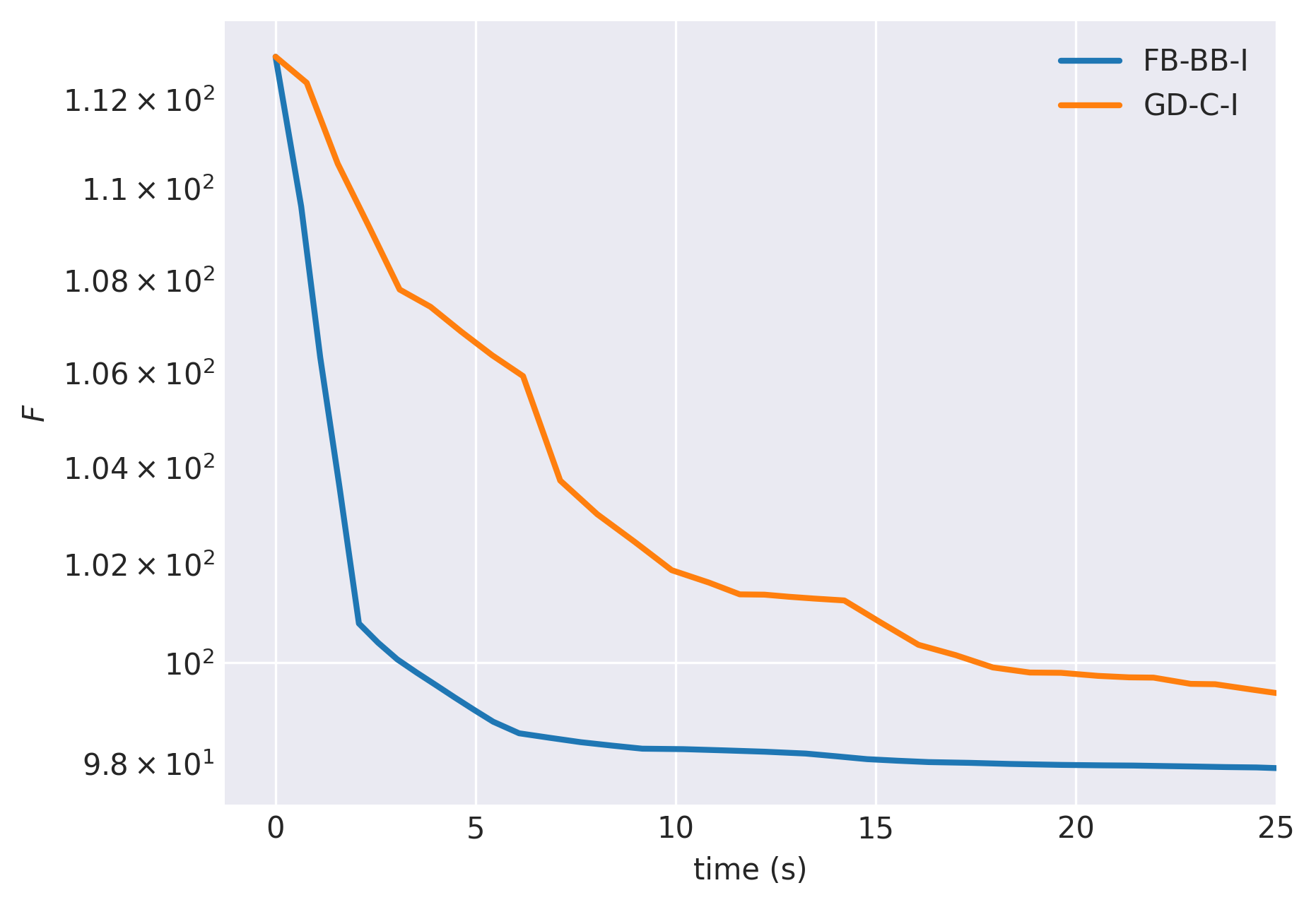}
        \caption{$N=4$}
    \end{subfigure}
    \caption{Results achieved solving the image deblurring problem on \textit{leaves} by varying the number $N$ of blocks. For $N=1$ the results obtained by GS-PnP are also included.}
    \label{fig:deblur_graph}
\end{figure}
It is evident that Block-PHILA-FB-BB-I consistently outperforms both Block-PHILA-GD-C-I and GS-PnP for $N=1$. Moreover, its performance does not degrade as the number of blocks increases. 

In Figure \ref{fig:blur_results} we compare the reconstructions obtained by GS-PnP, Block-PHILA-FB-BB-I and Block-PHILA-GD-C-I (for different numbers of blocks) after 2 seconds of computation. Following \cite{hurault2022gradient}, a further application of the denoiser 
$\denoiser_\sigma$, defined in \eqref{eq:D_sigma}, is also performed on the last iterate. For $N=1$ Block-PHILA-FB-BB-I yields the best approximation of the solution, achieving high overall sharpness and well-preserved edges. For $N=2$ and $N=4$, Block-PHILA-FB-BB-I also provides satisfactory reconstructions, effectively removing noise and blur without introducing block artifacts. In contrast, Block-PHILA-GD-C-I exhibits the worst reconstruction behavior. This is particularly evident for 
$N=1$, where the solution obtained with Block-PHILA FB-BB-I is of significantly higher quality.

\begin{figure}[]
	\centering
	\subfloat[Ground truth]{
	   \scalebox{1}{
    	\begin{tikzpicture}
    	\begin{scope}[spy using outlines={rectangle,blue,magnification=2,size=1.5cm}]
    	\node [name=c]{	\includegraphics[height=3cm]{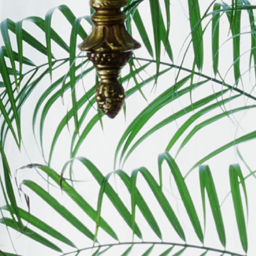}};
    	\spy on (1.,0.8) in node [name=c1]  at (0.75,-2.25);
    	\spy on (-0.75,0.75) in node [name=c1]  at (-0.75,-2.25);
    	\end{scope}
    	\end{tikzpicture}}
        }
	\subfloat[Corrupted data]{
	\scalebox{1}{
	\begin{tikzpicture}
	\begin{scope}[spy using outlines={rectangle,blue,magnification=2,size=1.5cm}]
	\node [name=c]{	\includegraphics[height=3cm]{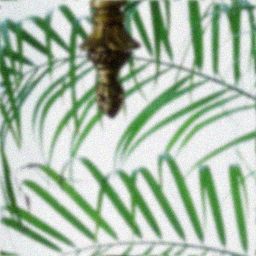}};
	\spy on (1.,0.8) in node [name=c1]  at (0.75,-2.25);
	\spy on (-0.75,0.75) in node [name=c1]  at (-0.75,-2.25);
	\end{scope}
	\end{tikzpicture}}
    }
	\subfloat[GS-PnP]{
	\scalebox{1}{
	\begin{tikzpicture}
	\begin{scope}[spy using outlines={rectangle,blue,magnification=2,size=1.5cm}]
	\node [name=c]{	\includegraphics[height=3cm]{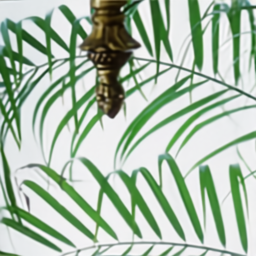}};
	\spy on (1.,0.8) in node [name=c1]  at (0.75,-2.25);
	\spy on (-0.75,0.75) in node [name=c1]  at (-0.75,-2.25);
	\end{scope}
	\end{tikzpicture}}
    }\\

	\subfloat[$N=1$]{
	\scalebox{1}{
	\begin{tikzpicture}
	\begin{scope}[spy using outlines={rectangle,blue,magnification=2,size=1.5cm}]
	\node [name=c]{	\includegraphics[height=3cm]{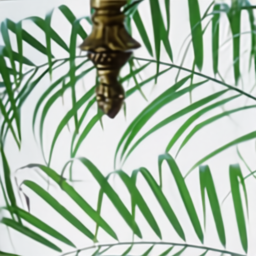}};
	\spy on (1.,0.8) in node [name=c1]  at (0.75,-2.25);
	\spy on (-0.75,0.75) in node [name=c1]  at (-0.75,-2.25);
	\end{scope}
	\end{tikzpicture}}
    }
	\subfloat[$N=2$]{
	\scalebox{1}{
	\begin{tikzpicture}
	\begin{scope}[spy using outlines={rectangle,blue,magnification=2,size=1.5cm}]
	\node [name=c]{	\includegraphics[height=3cm]{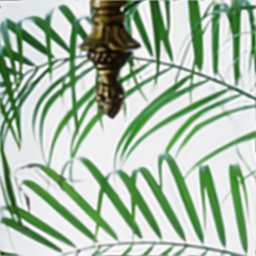}};
	\spy on (1.,0.8) in node [name=c1]  at (0.75,-2.25);
	\spy on (-0.75,0.75) in node [name=c1]  at (-0.75,-2.25);
	\end{scope}
	\end{tikzpicture}}
    }
	\subfloat[$N=4$]{
	\scalebox{1}{
	\begin{tikzpicture}
	\begin{scope}[spy using outlines={rectangle,blue,magnification=2,size=1.5cm}]
	\node [name=c]{	\includegraphics[height=3cm]{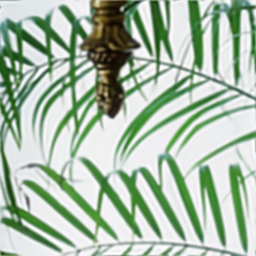}};
	\spy on (1.,0.8) in node [name=c1]  at (0.75,-2.25);
	\spy on (-0.75,0.75) in node [name=c1]  at (-0.75,-2.25);
	\end{scope}
	\end{tikzpicture}}
    }\\

	\subfloat[$N=1$]{
	\scalebox{1}{
	\begin{tikzpicture}
	\begin{scope}[spy using outlines={rectangle,blue,magnification=2,size=1.5cm}]
	\node [name=c]{	\includegraphics[height=3cm]{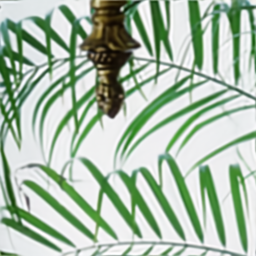}};
	\spy on (1.,0.8) in node [name=c1]  at (0.75,-2.25);
	\spy on (-0.75,0.75) in node [name=c1]  at (-0.75,-2.25);
	\end{scope}
	\end{tikzpicture}}
    }
	\subfloat[$N=2$]{
	\scalebox{1}{
	\begin{tikzpicture}
	\begin{scope}[spy using outlines={rectangle,blue,magnification=2,size=1.5cm}]
	\node [name=c]{	\includegraphics[height=3cm]{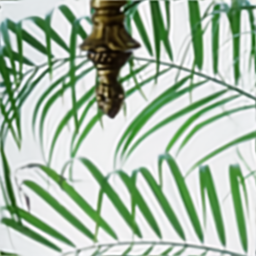}};
	\spy on (1.,0.8) in node [name=c1]  at (0.75,-2.25);
	\spy on (-0.75,0.75) in node [name=c1]  at (-0.75,-2.25);
	\end{scope}
	\end{tikzpicture}}
    }
	\subfloat[$N=4$]{
	\scalebox{1}{
	\begin{tikzpicture}
	\begin{scope}[spy using outlines={rectangle,blue,magnification=2,size=1.5cm}]
	\node [name=c]{	\includegraphics[height=3cm]{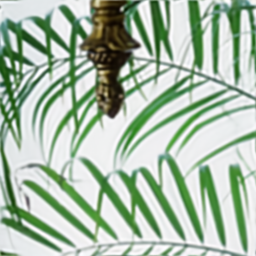}};
	\spy on (1.,0.8) in node [name=c1]  at (0.75,-2.25);
	\spy on (-0.75,0.75) in node [name=c1]  at (-0.75,-2.25);
	\end{scope}
	\end{tikzpicture}}
    }
    \caption{Reconstructions provided by GS-PnP, Block-PHILA-FB-BB-I (second row) and Block-PHILA-GD-C-I (third row) for the image deblurring problem.}
    \label{fig:blur_results}
\end{figure}

\textcolor{black}{Finally, in Table \ref{tab:placeholder_2}, we report the average number of backtracking steps per iteration performed by the different versions of Block-PHILA. The results clearly show that a very small number of backtracking steps is required overall, especially in the case of the FB implementations. This is likely due to the fact that the Lipschitz constant of $\nabla f$ is smaller than that of $\nabla f+\nabla \phi$. Consequently, the initial stepsize is more frequently accepted, reducing the computational burden of the Armijo line search.
}

\begin{table}[]
    \centering
    \resizebox{\textwidth}{!}{%
   \textcolor{black}{ \begin{tabular}{l rrr rrr rrr}
    \toprule
    \cmidrule(lr){2-10}
    & \multicolumn{3}{c}{$N=1$} & \multicolumn{3}{c}{$N=2$} & \multicolumn{3}{c}{$N=4$} \\
    \cmidrule(lr){2-4} \cmidrule(lr){5-7} \cmidrule(lr){8-10}
     & \textit{Butterfly} & \textit{Leaves} & \textit{Starfish} & \textit{Butterfly} & \textit{Leaves} & \textit{Starfish} & \textit{Butterfly} & \textit{Leaves} & \textit{Starfish} \\
    \midrule
    FB-BB-I & 0.88 & 0.38 & 0.82 & 0.40 & 0.58 & 0.18 & 0.27 & 0.49 & 0.21 \\
    FB-BB   & 0.12 & 0.00 & 0.00 & 0.07 & 0.05 & 0.00 & 0.00 & 0.02 & 0.00 \\
    FB-C-I  & 0.08 & 0.00 & 0.13 & 0.08 & 0.03 & 0.08 & 0.00 & 0.00 & 0.08 \\
    FB-C    & 0.00 & 0.00 & 0.00 & 0.00 & 0.00 & 0.00 & 0.00 & 0.00 & 0.00 \\
    GD-BB-I & 2.99 & 2.99 & 2.99 & 2.99 & 2.99 & 2.98 & 2.99 & 3.00 & 2.96 \\
    GD-BB   & 2.99 & 2.99 & 2.99 & 2.99 & 3.00 & 2.98 & 2.99 & 3.00 & 2.97 \\
    GD-C-I  & 2.36 & 2.38 & 2.33 & 2.37 & 2.38 & 2.43 & 2.21 & 2.38 & 2.38 \\
    GD-C    & 2.63 & 2.62 & 2.38 & 2.40 & 2.57 & 2.61 & 2.42 & 2.60 & 2.44 \\
    \bottomrule
    \end{tabular}}
    }
    \caption{Average number of backtracking steps per iteration performed by the different versions of Block-PHILA on the deblurring problems.}
    \label{tab:placeholder_2}
\end{table}
\subsubsection{\textcolor{black}{The effect of the variable metric}}\label{sec:variable_metric}
\textcolor{black}{In all the experiments reported so far, the scaling matrices have been set to $D_k=I_{n_{i_k}}$, meaning that the variable metric allowed by Block-PHILA is not exploited in the results shown in Table \ref{tab:deblur}.
 In this section, we investigate whether a non-trivial choice of $D_k$ can be beneficial in the case of the image deblurring problem considered before. Specifically, we consider a scaling matrix inspired by Adagrad \cite{Duchi-etal-2009} whose entries accumulate the squared components of the gradient observed on the selected block, namely
\begin{equation}\label{eq:adagrad_metric}
g_k^{i_k} = g_k^{i_k} + (\Uk^T\nabla \ff(\xk))^2
\qquad
S_k = \mathrm{diag}\left(\sqrt{g_k^{i_k}+\varepsilon}\right)
\qquad
D_k = \min\left\{\mu_k, \frac{1}{\mu_k\max{(S_k)}}S_k\right\},
\end{equation}
where the square operation is applied element-wise, $\varepsilon>0$ is a small constant to ensure stability and the vectors $g_0^{i}$ for $i=1,\ldots,N$ are defined as the null vectors. The sequence $\mu_k$ is defined as $\mu_k = 1+\frac{t_1}{{(\mathrm{div}(k,N)+1)}^{1+t_2}}$, where $t_1=10$ and $t_2=1.1$ \cite{Bonettini-Loris-Porta-Prato-2016}. The variable metric \eqref{eq:adagrad_metric} has been implemented and tested specifically for the FB configurations of Block-PHILA. As defined in Table \ref{tab:variants}, this yields four additional variants denoted by the suffix {VM}, namely Block-PHILA-{FB-C-VM}, {FB-BB-VM}, {FB-C-I-VM}, and {FB-BB-I-VM}.
We point out that, even in the presence of a scaling matrix, the inexact computation of the proximal operator is performed by means of an inexact dual procedure detailed in Appendix \ref{app:prox_computation}.\\
In Figure \ref{fig:vm_graph}, we compare the PSNR results obtained by solving the deblurring problem on the \textit{leaves} image, contrasting the use of the identity matrix with the variable metric \eqref{eq:adagrad_metric} across different values of $N$. The top row displays the {FB} variants without inertia (with and without the scaling matrix), while the bottom row reports the {FB-I} versions (with and without the scaling matrix). In both cases, for $N=1$, employing the variable metric improves the performance when a constant steplength is used. However, the adoption of the BB stepsize generally outperforms the scaled versions, effectively closing the performance gap between the standard and VM approaches. Finally, concerning the block-coordinate configurations ($N=2$ and $N=4$), the advantage of using non-constant scaling matrices is no longer evident, which might be attributed to the iterative procedure required to compute the inexact proximal operator in the presence of a non-trivial scaling matrix.}
\begin{figure}
    \centering
    \begin{subfigure}[b]{0.3\textwidth}
        \centering
        \includegraphics[width=\textwidth]{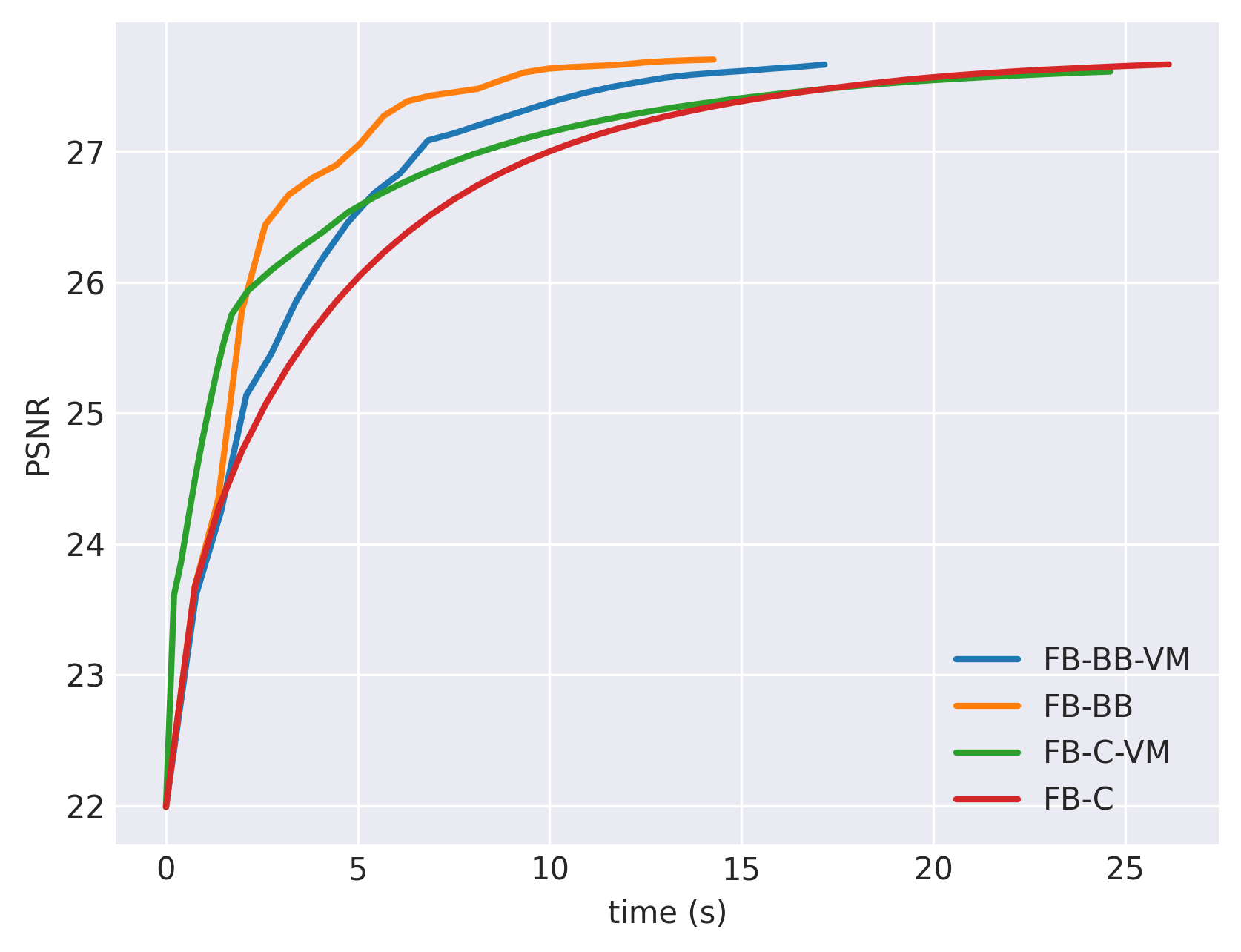}
        \caption{$N=1$}
    \end{subfigure}\hfill
    \begin{subfigure}[b]{0.3\textwidth}
        \centering
        \includegraphics[width=\textwidth]{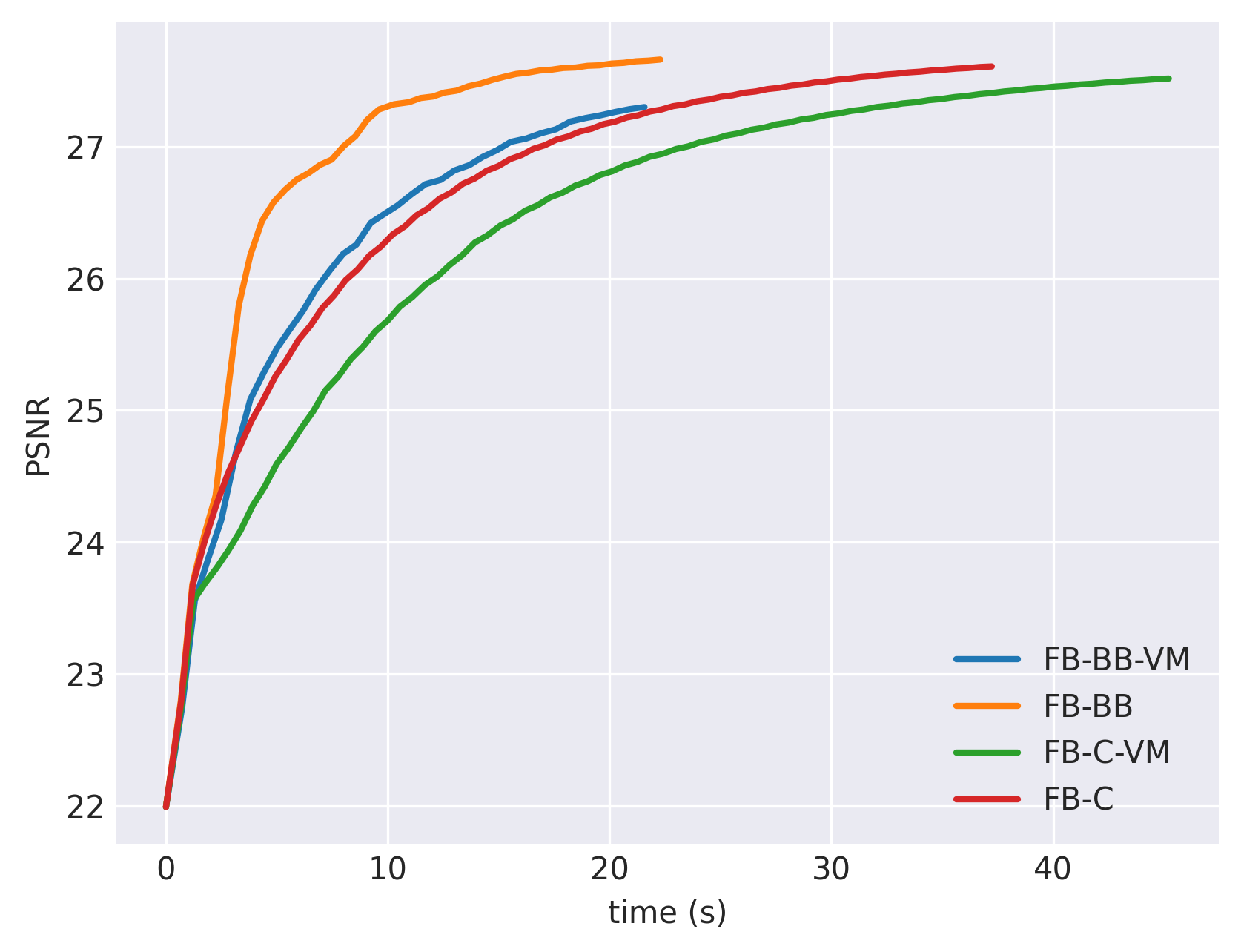}
        \caption{$N=2$}
    \end{subfigure}\hfill
    \begin{subfigure}[b]{0.3\textwidth}
        \centering
        \includegraphics[width=\textwidth]{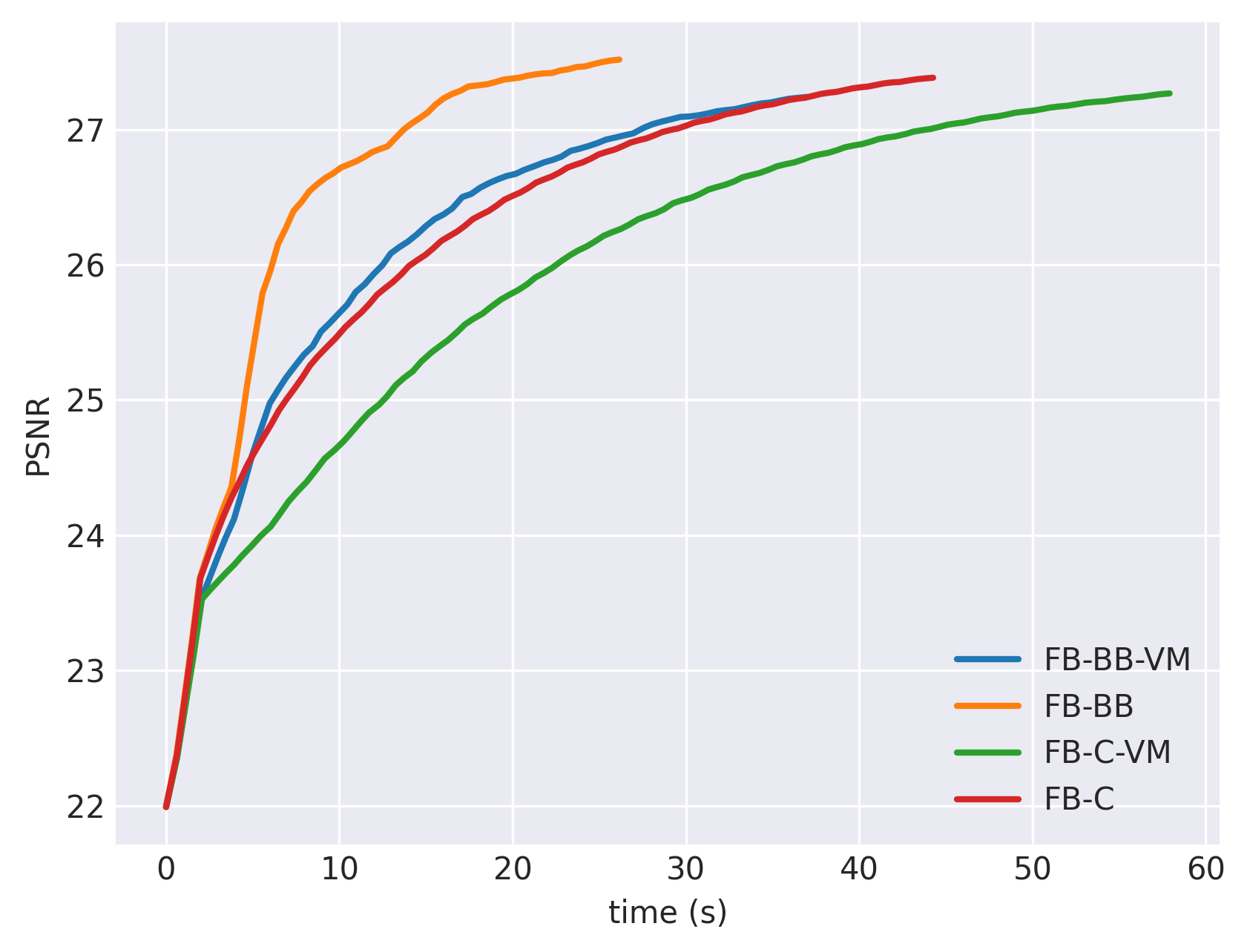}
        \caption{$N=4$}
    \end{subfigure}\\
    \begin{subfigure}[b]{0.3\textwidth}
        \centering
        \includegraphics[width=\textwidth]{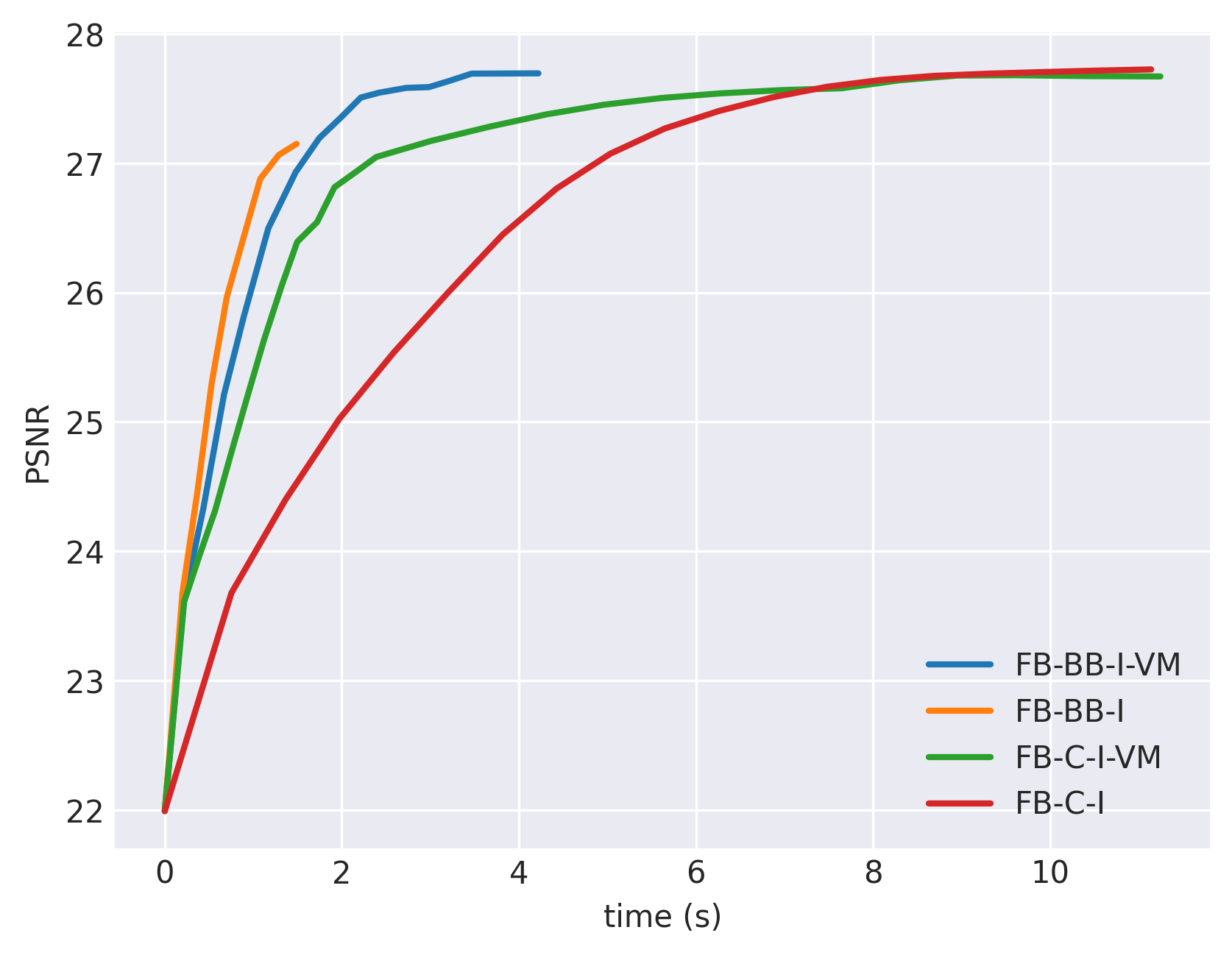}
        \caption{$N=1$}
    \end{subfigure}\hfill
    \begin{subfigure}[b]{0.3\textwidth}
        \centering
        \includegraphics[width=\textwidth]{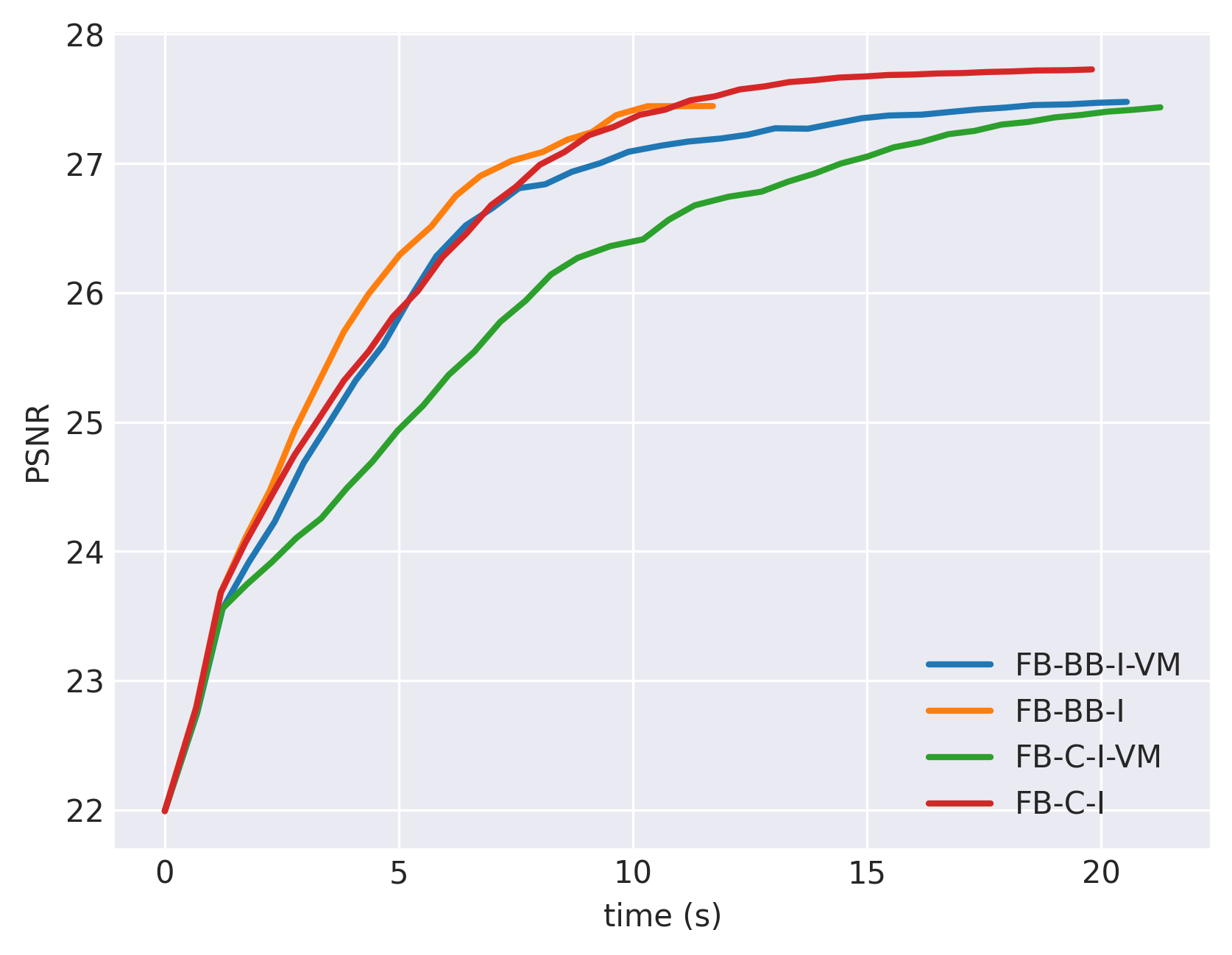}
        \caption{$N=2$}
    \end{subfigure}\hfill
    \begin{subfigure}[b]{0.3\textwidth}
        \centering
        \includegraphics[width=\textwidth]{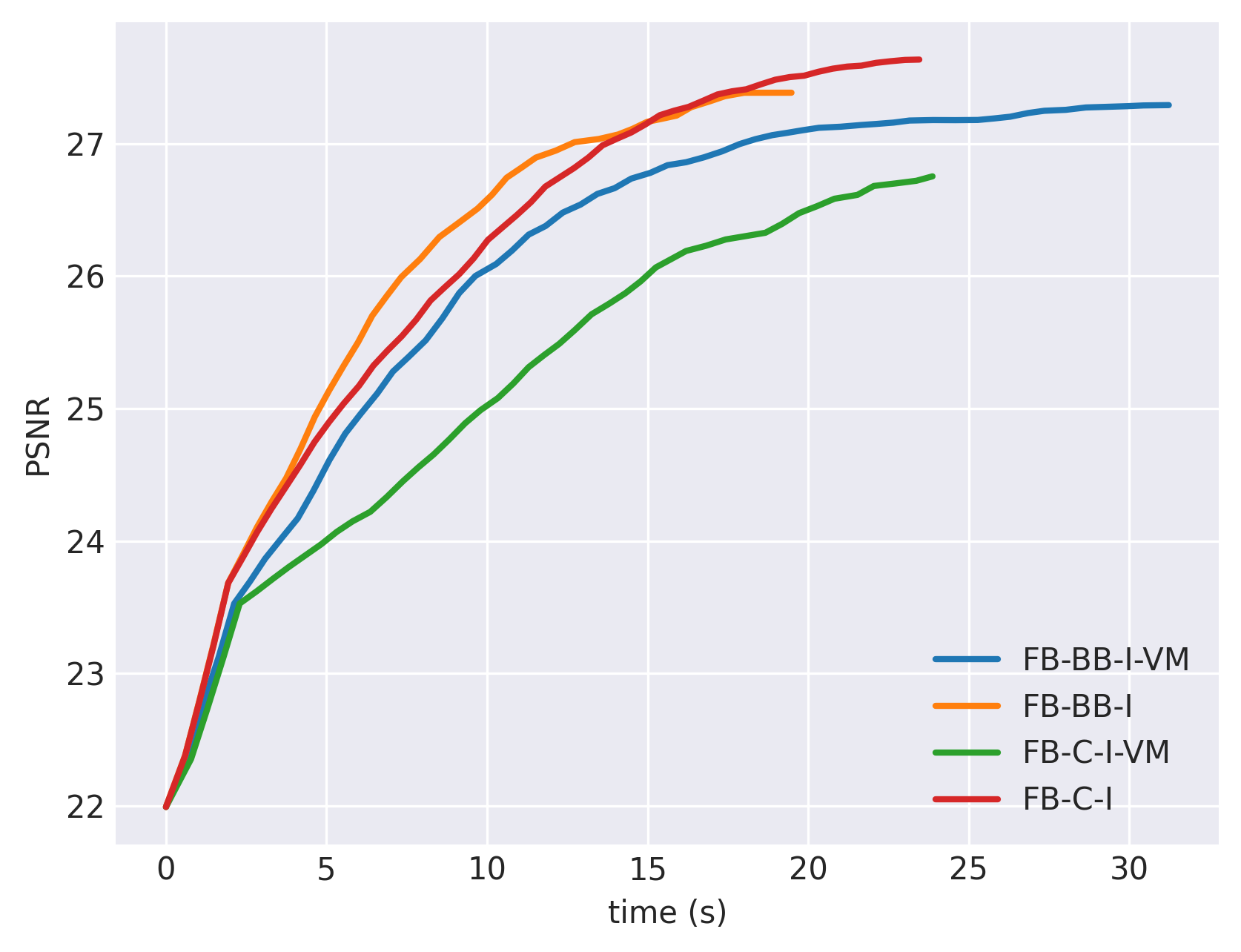}
        \caption{$N=4$}
    \end{subfigure}
    \caption{Performance comparison for the image deblurring problem on the \textit{leaves} image, varying the number of blocks $N$. The top row compares the Block-PHILA-{FB} variants (without inertia) with and without the variable metric scaling. The bottom row shows the same comparison for the Block-PHILA-{FB-I} versions (with inertia).
    }
    \label{fig:vm_graph}
\end{figure}

\subsection{Image super-resolution}\label{sec:image_SR}
In this section, assuming additive Gaussian noise with standard deviation $\nu\in\R^+$, we consider problem \eqref{eq:GS_PNP_problem} with data fidelity term $\mathcal{D}(x) = \frac{1}{2}\|Ax-b\|^2$. Here, the forward operator is defined as $A = SH$, where $H\in\mathbb{R}^{n\times n}$ is a convolution operator with an anti-aliasing kernel, and $S\in\mathbb{R}^{m\times n}$ is a standard $s$-fold downsampling matrix such that $n = s^2 m$. Specifically, the blur operator $H$ is the same as that used in the image deblurring problems discussed in Section \ref{sec:image_deblurring}, while the downsampled images correspond to a scale factor $s = 2$. The data $b\in\mathbb{R}^n$  represents a low-resolution image obtained by the high-resolution image  $x\in\mathbb{R}^n$ via $y = SHx + \eta$, where both the reference images and the noise $\eta$ are set as in Section \ref{sec:image_deblurring}. The values for $\sigma$ and $\textcolor{black}{\varrho}$ in \eqref{eq:GS_PNP_problem} are fixed to $2\nu$ and $0.065$, respectively, following \cite{hurault2022gradient}, where the same test problem was considered. 
All the compared algorithms were initialized using as $x_0$ a bicubic interpolation of $b$ (with a proper shift correction as suggested in \cite{Zhang-et-al-2021}). We refer again the reader to Appendix \ref{app:prox_computation_exact} and Appendix \ref{app:prox_computation} for the computation of the proximal point needed by Block-PHILA FB variants and GS-PnP.

Table \ref{tab:sr} reports the objective function values, PSNR, number of iterations, and computational times achieved by the compared methods when condition \eqref{eq:stop_Crit} is satisfied with $\varepsilon = 10^{-5}$. For each image and block configuration, the highest PSNR values are in bold, while the lowest computational times are underlined.
\begin{table}
\renewcommand{\arraystretch}{1.2}
\setlength{\tabcolsep}{5pt}
\adjustbox{max width=\textwidth}{%
\centering
\begin{tabular}{p{1cm} l *{3}{c c c c}}
\toprule
 &  & \multicolumn{4}{c}{\textit{Butterfly}} & \multicolumn{4}{c}{\textit{Leaves}} & \multicolumn{4}{c}{\textit{Starfish}} \\
\cmidrule(lr){3-6}\cmidrule(lr){7-10}\cmidrule(lr){11-14}
  &  & $F$ & PSNR & itr & time & $F$ & PSNR & itr & time & $F$ & PSNR & itr & time \\
\midrule
\multirow{11}{*}{\parbox{1cm}{\centering \rotatebox{90}{$N=1$}}} & GS-PnP & 33.87 & 26.77 & 43 & 26.30 & 34.04 & 26.79 & 63 & 38.34 & 34.31 & 27.41 & 24 & 14.63 \\
 & DPIR & - & 25.84 & 8 & \underline{3.67} & - & 25.51 & 8 & \underline{3.46} & - & 27.14 & 8 & \underline{3.43} \\
 \cmidrule(l){2-14}  & BCRED & - & 25.23 & 200 & 64.28 & - & 23.83 & 200 & 65.66 & - & 26.18 & 200 & 65.87 \\
 & {FB-BB-I} & 33.87 & 26.93 & 11 & 7.36 & 34.12 & 26.49 & 11 & 8.22 & 34.31 & \textbf{27.49} & 10 & 6.86 \\
 & {FB-BB} & 33.86 & 26.82 & 22 & 14.34 & 34.04 & 26.83 & 31 & 19.94 & 34.31 & 27.42 & 13 & 8.17 \\
 & {FB-C-I} & 33.86 & \textbf{27.00} & 22 & 13.76 & 34.03 & \textbf{26.92} & 33 & 20.34 & 34.31 & 27.48 & 16 & 9.99 \\
 & {FB-C} & 33.87 & 26.77 & 43 & 26.69 & 34.04 & 26.79 & 63 & 38.90 & 34.31 & 27.41 & 24 & 15.41 \\
\cmidrule(l){2-14} & {GD-BB-I} & 33.87 & 26.71 & 63 & 66.36 & 34.04 & 26.75 & 99 & 105.73 & 34.31 & 27.39 & 34 & 36.70 \\
 & {GD-BB} & 33.87 & 26.71 & 77 & 81.34 & 34.05 & 26.74 & 118 & 131.05 & 34.31 & 27.40 & 44 & 49.97 \\
 & {GD-C-I} & 33.87 & 26.82 & 23 & 15.75 & 34.04 & 26.83 & 34 & 22.88 & 34.31 & 27.45 & 16 & 11.03 \\
 & {GD-C} & 33.88 & 26.53 & 48 & 36.24 & 34.09 & 26.37 & 53 & 40.74 & 34.31 & 27.35 & 28 & 22.25 \\
\midrule
\multirow{9}{*}{\parbox{1cm}{\centering \rotatebox{90}{$N=2$}}} & BCRED & - & 24.78 & 200 & 50.61 & - & 23.43 & 200 & 51.60 & - & 26.03 & 200 & 52.93 \\
 & {FB-BB-I} & 33.88 & \textbf{26.94} & 19 & \underline{10.60} & 34.13 & 26.44 & 17 & \underline{9.99} & 34.34 & 27.41 & 13 & \underline{7.26} \\
 & {FB-BB} & 33.87 & 26.74 & 38 & 20.55 & 34.05 & 26.77 & 51 & 27.61 & 34.31 & 27.41 & 26 & 13.96 \\
 & {FB-C-I} & 33.87 & 26.98 & 30 & 16.06 & 34.04 & \textbf{26.90} & 59 & 31.40 & 34.31 & \textbf{27.50} & 26 & 14.10 \\
 & {FB-C} & 33.87 & 26.72 & 72 & 38.44 & 34.05 & 26.69 & 96 & 51.54 & 34.31 & 27.37 & 38 & 20.56 \\
\cmidrule(l){2-14} & {GD-BB-I} & 33.88 & 26.60 & 98 & 92.80 & 34.06 & 26.63 & 148 & 142.02 & 34.31 & 27.35 & 54 & 52.26 \\
 & {GD-BB} & 33.88 & 26.61 & 120 & 115.64 & 34.07 & 26.55 & 156 & 158.36 & 34.40 & 26.99 & 20 & 20.65 \\
 & {GD-C-I} & 33.91 & 26.44 & 28 & 16.34 & 34.05 & 26.74 & 54 & 30.98 & 34.32 & 27.39 & 23 & 13.53 \\
 & {GD-C} & 33.95 & 26.01 & 51 & 33.15 & 34.14 & 26.10 & 81 & 53.37 & 34.35 & 27.12 & 30 & 20.45 \\
\midrule
\multirow{9}{*}{\parbox{1cm}{\centering \rotatebox{90}{$N=4$}}} & BCRED & - & 24.08 & 200 & 43.76 & - & 22.74 & 200 & 45.45 & - & 25.79 & 200 & 45.59 \\
 & {FB-BB-I} & 33.88 & \textbf{26.81} & 41 & \underline{20.19} & 34.10 & 26.55 & 41 & \underline{20.52} & 34.31 & 27.46 & 40 & 19.51 \\
 & {FB-BB} & 33.87 & 26.69 & 60 & 28.37 & 34.05 & 26.72 & 85 & 39.60 & 34.31 & 27.38 & 38 & 17.73 \\
 & {FB-C-I} & 33.88 & 26.95 & 56 & 25.28 & 34.07 & \textbf{26.73} & 60 & 27.17 & 34.32 & \textbf{27.48} & 36 & \underline{16.53} \\
 & {FB-C} & 33.88 & 26.52 & 103 & 46.93 & 34.08 & 26.48 & 136 & 62.24 & 34.32 & 27.31 & 60 & 28.07 \\
\cmidrule(l){2-14} & {GD-BB-I} & 33.90 & 26.37 & 139 & 121.05 & 34.09 & 26.37 & 200 & 176.51 & 34.41 & 26.99 & 36 & 31.48 \\
 & {GD-BB} & 33.90 & 26.36 & 163 & 147.76 & 34.13 & 26.20 & 200 & 185.48 & 34.42 & 26.92 & 34 & 36.08 \\
 & {GD-C-I} & 33.90 & 26.44 & 56 & 28.08 & 34.07 & 26.58 & 84 & 41.83 & 34.31 & 27.36 & 44 & 22.88 \\
 & {GD-C} & 34.18 & 25.33 & 51 & 29.03 & 34.32 & 25.33 & 88 & 51.73 & 34.49 & 26.78 & 28 & 16.68 \\
\bottomrule
\end{tabular}}
\caption{Results achieved by the Block-PHILA variants and BCRED in solving the super-resolution problems by varying the number $N$ of blocks. For $N=1$, the results obtained by GS-PnP and DPIR are also included. For each image and value of $N$, the best PSNR is shown in bold and the lowest time is underlined.}
\label{tab:sr}
\end{table}

Figure \ref{fig:sr_graph} reports the PSNR  and objective function values achieved by Block-PHILA-FB-BB-I and Block-PHILA-GD-C-I on the \textit{Butterfly} image with respect to the computational time, for different numbers of blocks. In the case $N=1$, the results obtained by GS-PnP are also included.  
\begin{figure}
    \centering
    \begin{subfigure}[b]{0.3\textwidth}
        \centering
        \includegraphics[width=\textwidth]{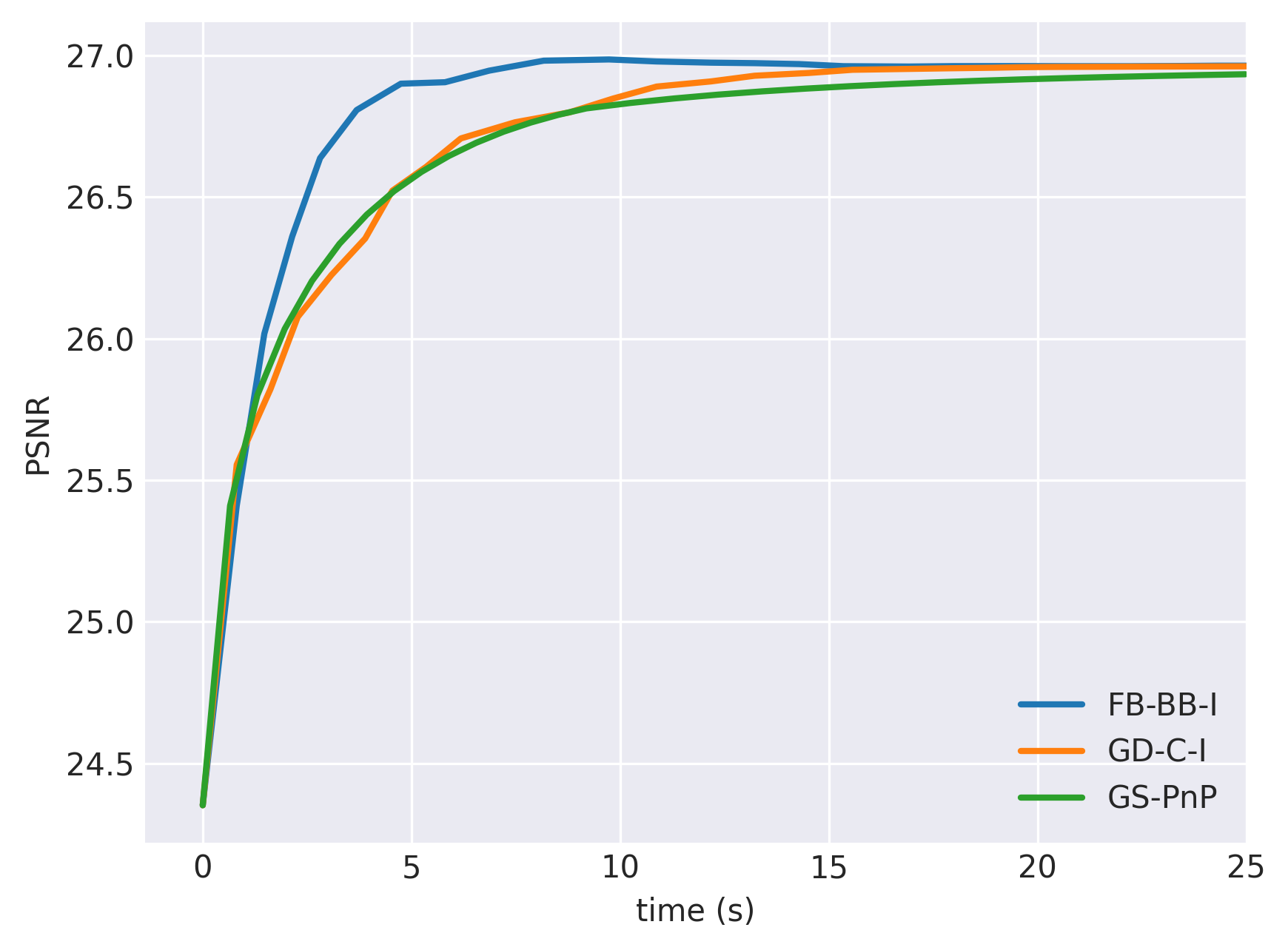}
        \caption{$N=1$}
    \end{subfigure}\hfill
    \begin{subfigure}[b]{0.3\textwidth}
        \centering
        \includegraphics[width=\textwidth]{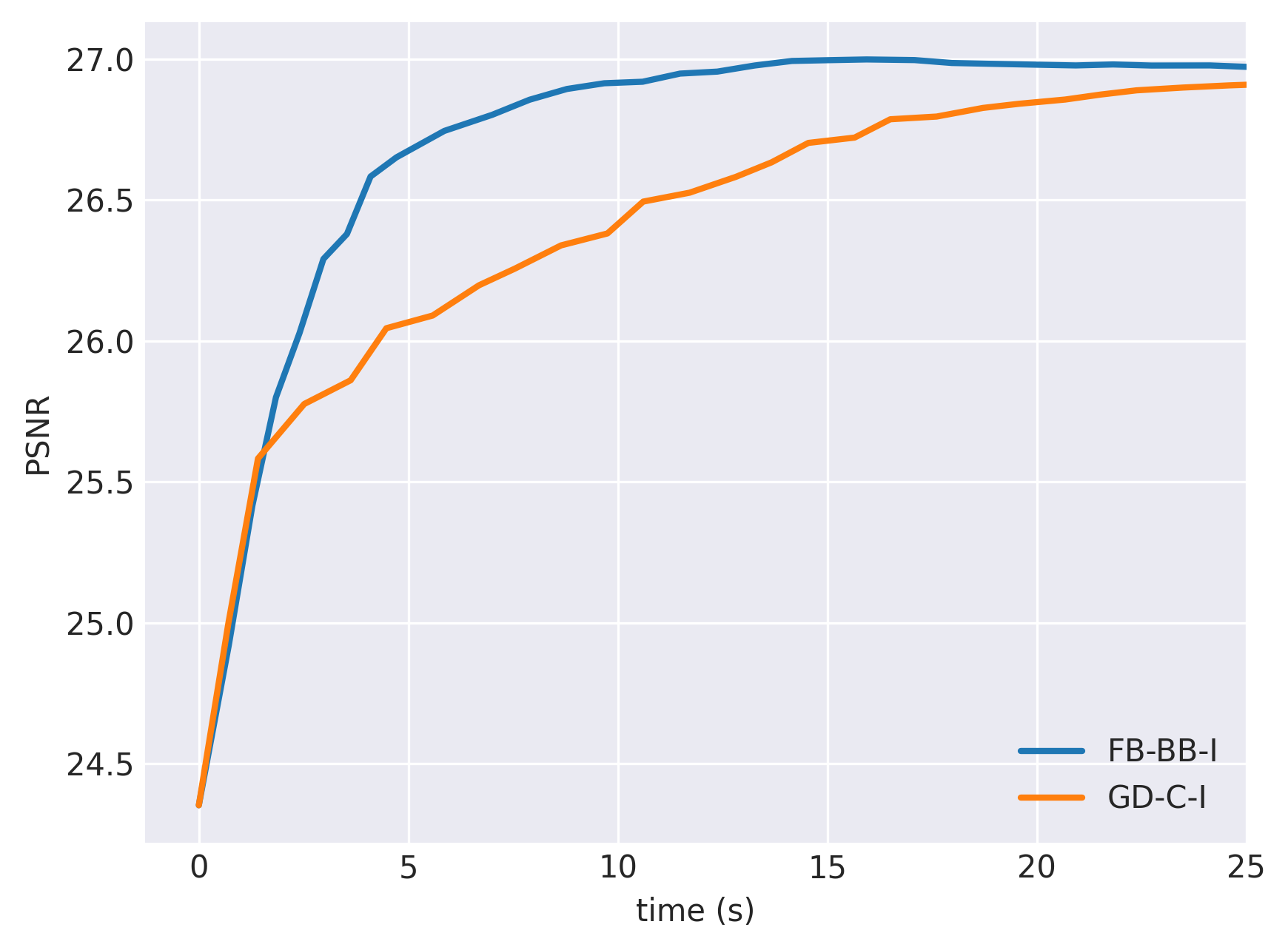}
        \caption{$N=2$}
    \end{subfigure}\hfill
    \begin{subfigure}[b]{0.3\textwidth}
        \centering
        \includegraphics[width=\textwidth]{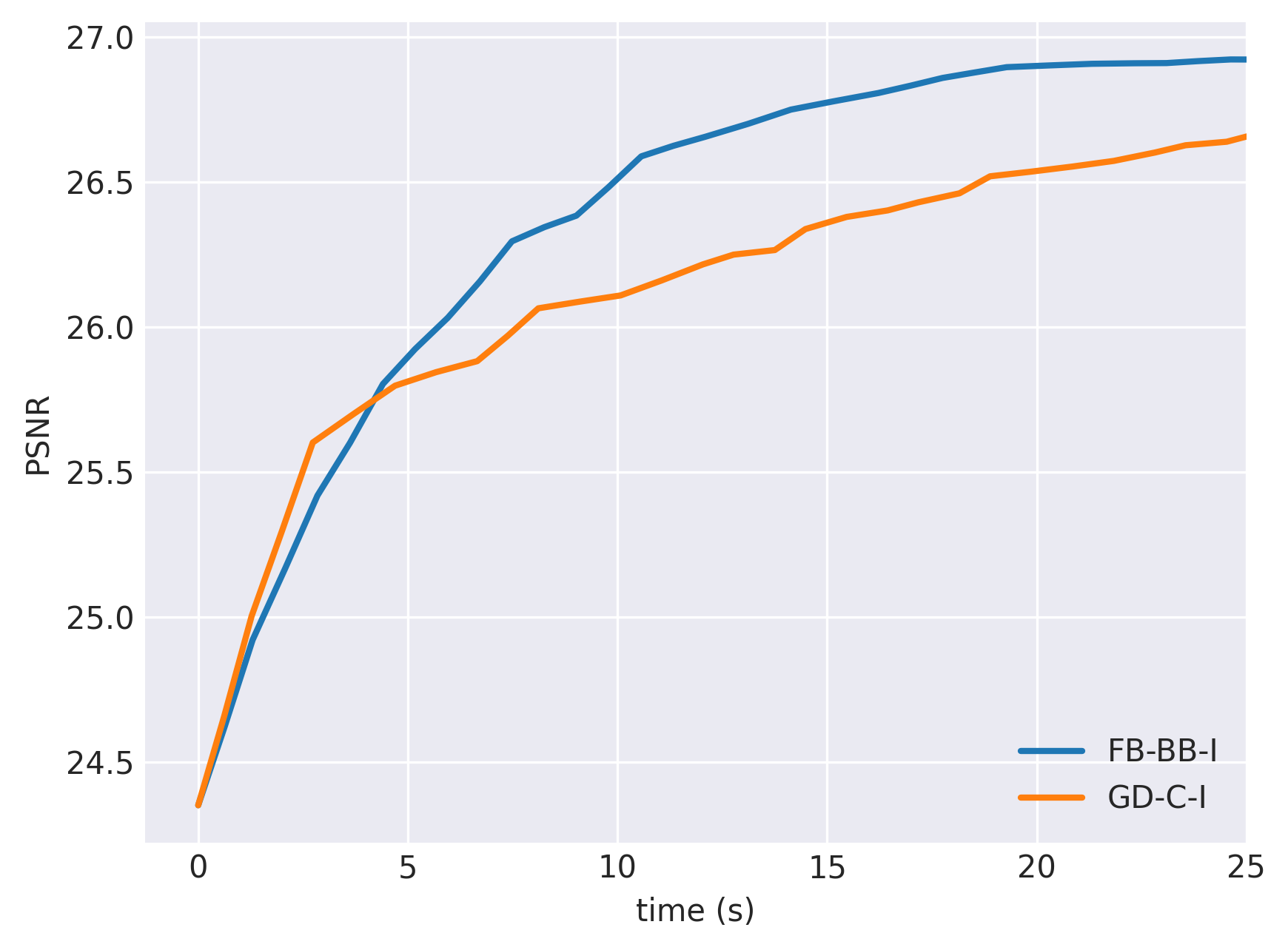}
        \caption{$N=4$}
    \end{subfigure}\\
    \begin{subfigure}[b]{0.3\textwidth}
        \centering
        \includegraphics[width=\textwidth]{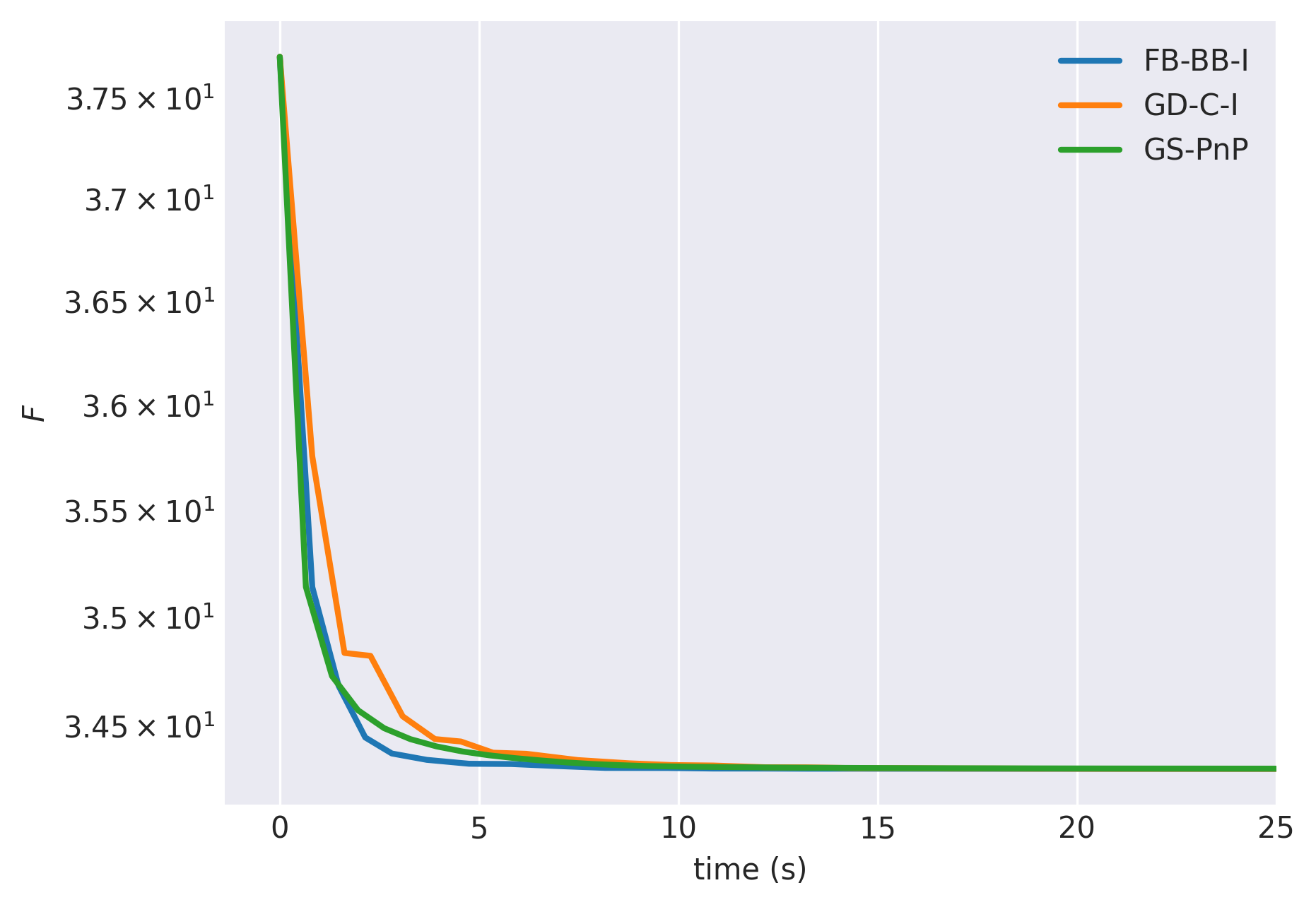}
        \caption{$N=1$}
    \end{subfigure}\hfill
    \begin{subfigure}[b]{0.3\textwidth}
        \centering
        \includegraphics[width=\textwidth]{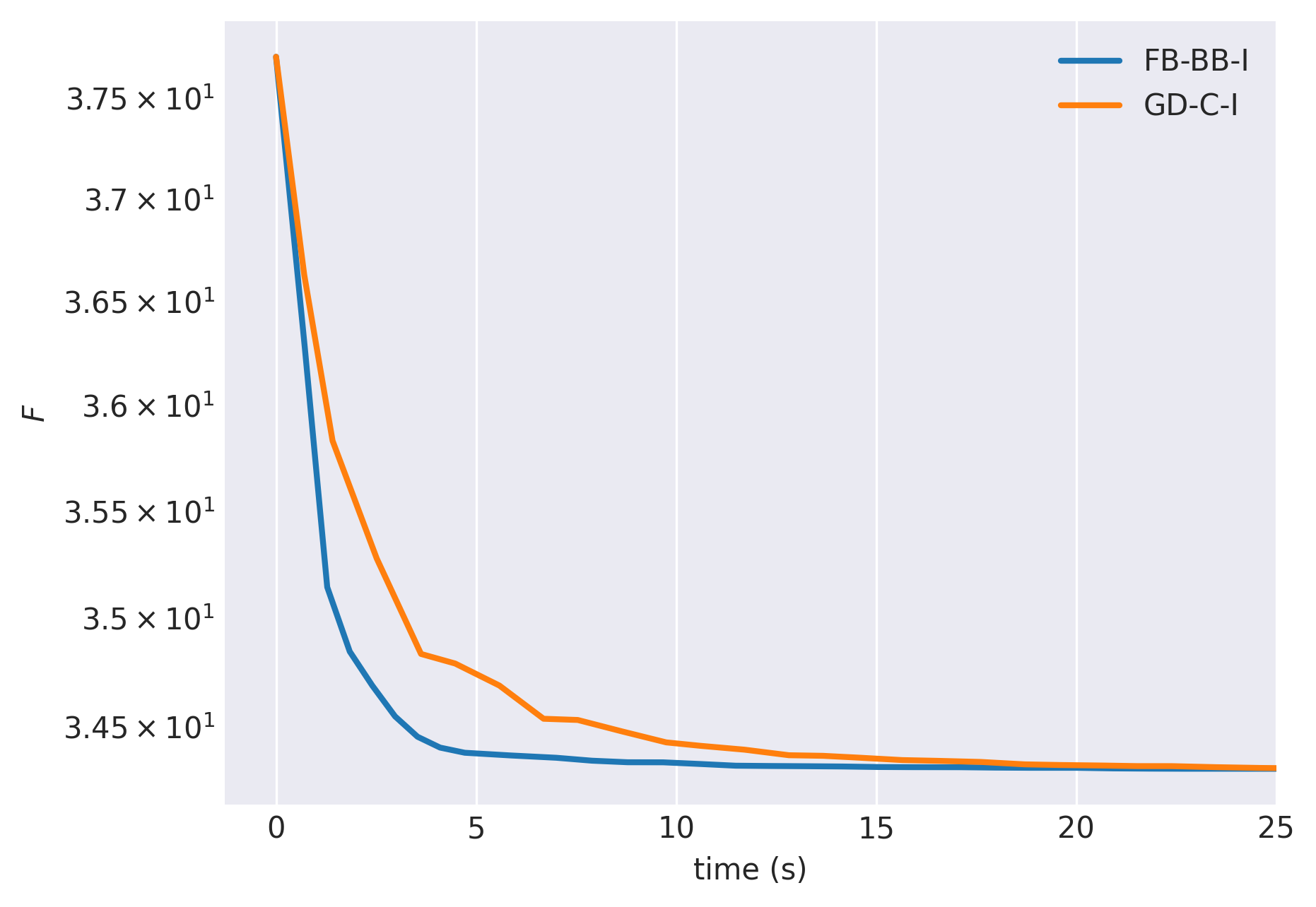}
        \caption{$N=2$}
    \end{subfigure}\hfill
    \begin{subfigure}[b]{0.3\textwidth}
        \centering
        \includegraphics[width=\textwidth]{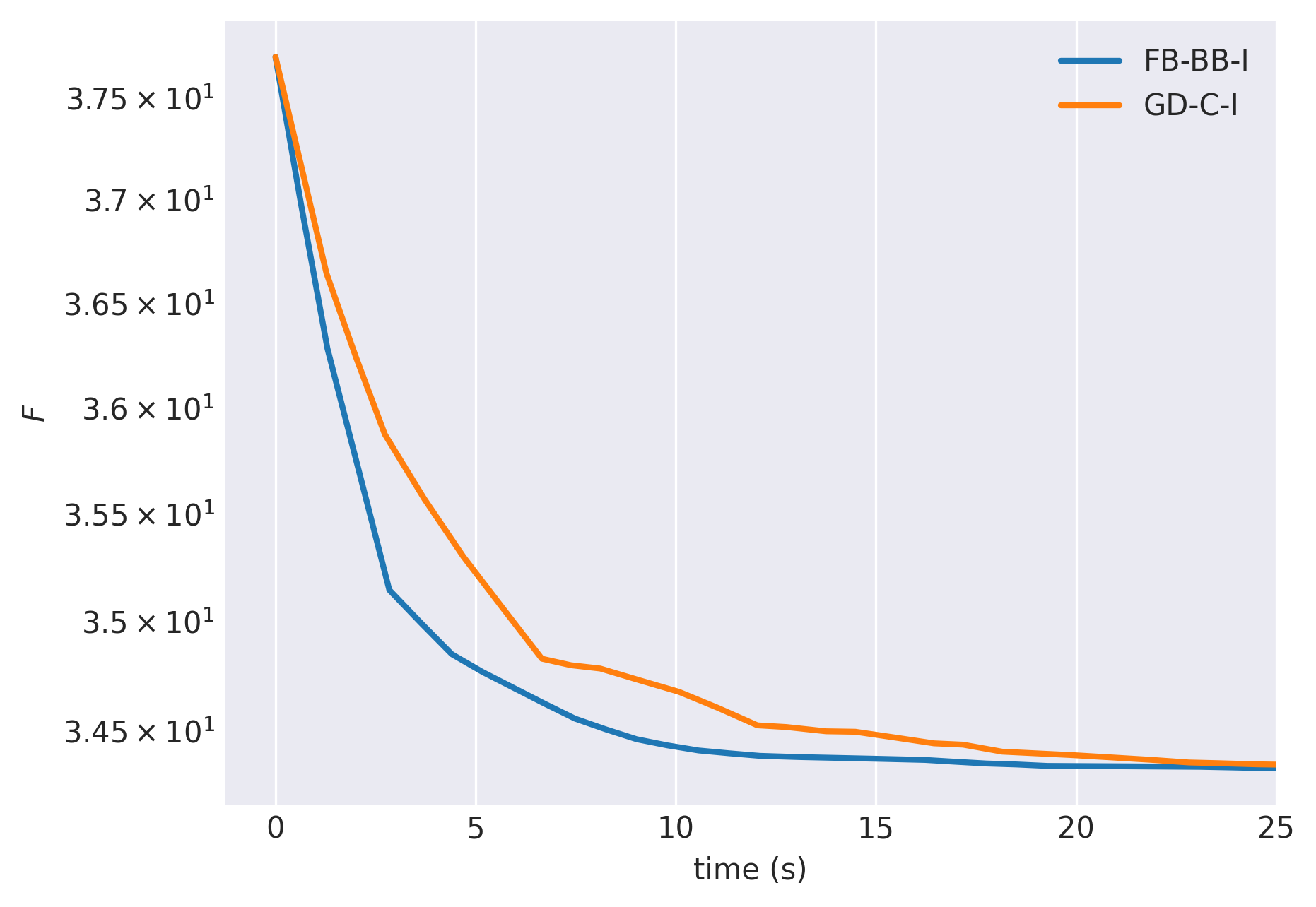}
        \caption{$N=4$}
    \end{subfigure}
    \caption{Results achieved solving the image super resolution problem on \textit{starfish} by varying the number $N$ of blocks. For $N=1$ the results obtained by GS-PnP are also included.}
    \label{fig:sr_graph}
\end{figure}
The results in Table \ref{tab:sr} and Figure \ref{fig:sr_graph} are consistent with the conclusions drawn in Section \ref{sec:image_deblurring} from the image deblurring experiments. In particular, across all tested images and block configurations, the Block-PHILA variants FB-C-I, FB-BB and FB-BB-I consistently satisfy the stopping criterion faster than the competing methods, while providing comparable or higher PSNR values. Moreover, the Block-PHILA FB versions based on the splitting~\eqref{eq:splitting_PnP} generally outperform the GD variants exploiting the differentiability of both the fidelity and regularization terms. This is further confirmed by Figure \ref{fig:sr_results} which reports the reconstructions achieved by GS-PnP, Block-PHILA FB-BB-I and Block-PHILA GD-C-I after a time budget of 2 seconds and an additional denoising step. For any number of blocks, Block-PHILA FB-BB-I produces sharp and accurate reconstructions without the artifacts observed in Block-PHILA GD-C-I for $N=2$ or $N=4$. We can conclude that the block-wise formulation remains advantageous for super-resolution, enabling efficient processing of possibly large-scale images with reduced memory usage and computational effort, while maintaining high reconstruction accuracy.

\begin{figure}
	\centering
	\subfloat[Ground Truth]{
	   \scalebox{1}{
    	\begin{tikzpicture}
    	\begin{scope}[spy using outlines={rectangle,blue,magnification=2,size=1.5cm}]
    	\node [name=c]{	\includegraphics[height=3cm]{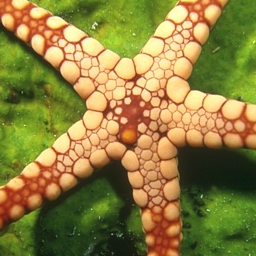}};
    	\spy on (0.8,0.1) in node [name=c1]  at (0.75,-2.25);
    	\spy on (-0.8,-0.1) in node [name=c1]  at (-0.75,-2.25);
    	\end{scope}
    	\end{tikzpicture}}
        }
	\subfloat[Corrupted data]{
	\scalebox{1}{
	\begin{tikzpicture}
	\begin{scope}[spy using outlines={rectangle,blue,magnification=2,size=1.5cm}]
	\node [name=c]{	\includegraphics[height=3cm]{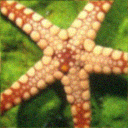}};
	\spy on (0.8,0.1) in node [name=c1]  at (0.75,-2.25);
	\spy on (-0.8,-0.1) in node [name=c1]  at (-0.75,-2.25);
	\end{scope}
	\end{tikzpicture}}
    }
	\subfloat[GS-PnP]{
	\scalebox{1}{
	\begin{tikzpicture}
	\begin{scope}[spy using outlines={rectangle,blue,magnification=2,size=1.5cm}]
	\node [name=c]{	\includegraphics[height=3cm]{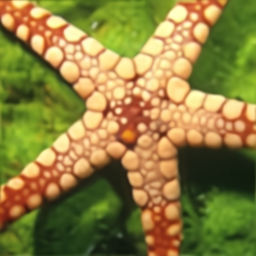}};
	\spy on (0.8,0.1) in node [name=c1]  at (0.75,-2.25);
	\spy on (-0.8,-0.1) in node [name=c1]  at (-0.75,-2.25);
	\end{scope}
	\end{tikzpicture}}
    }\\
	
	\subfloat[$N=1$]{
	\scalebox{1}{
	\begin{tikzpicture}
	\begin{scope}[spy using outlines={rectangle,blue,magnification=2,size=1.5cm}]
	\node [name=c]{	\includegraphics[height=3cm]{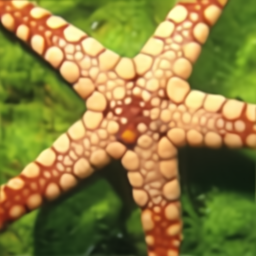}};
	\spy on (0.8,0.1) in node [name=c1]  at (0.75,-2.25);
	\spy on (-0.8,-0.1) in node [name=c1]  at (-0.75,-2.25);
	\end{scope}
	\end{tikzpicture}}
    }
	\subfloat[$N=2$]{
	\scalebox{1}{
	\begin{tikzpicture}
	\begin{scope}[spy using outlines={rectangle,blue,magnification=2,size=1.5cm}]
	\node [name=c]{	\includegraphics[height=3cm]{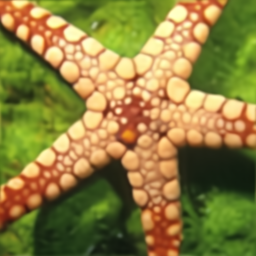}};
	\spy on (0.8,0.1) in node [name=c1]  at (0.75,-2.25);
	\spy on (-0.8,-0.1) in node [name=c1]  at (-0.75,-2.25);
	\end{scope}
	\end{tikzpicture}}
    }
	\subfloat[$N=4$]{
	\scalebox{1}{
	\begin{tikzpicture}
	\begin{scope}[spy using outlines={rectangle,blue,magnification=2,size=1.5cm}]
	\node [name=c]{	\includegraphics[height=3cm]{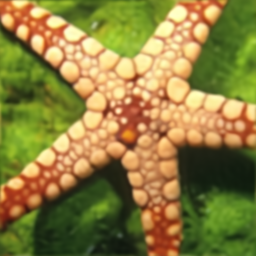}};
	\spy on (0.8,0.1) in node [name=c1]  at (0.75,-2.25);
	\spy on (-0.8,-0.1) in node [name=c1]  at (-0.75,-2.25);
	\end{scope}
	\end{tikzpicture}}
    }\\

	\subfloat[$N=1$]{
	\scalebox{1}{
	\begin{tikzpicture}
	\begin{scope}[spy using outlines={rectangle,blue,magnification=2,size=1.5cm}]
	\node [name=c]{	\includegraphics[height=3cm]{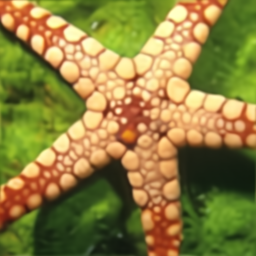}};
	\spy on (0.8,0.1) in node [name=c1]  at (0.75,-2.25);
	\spy on (-0.8,-0.1) in node [name=c1]  at (-0.75,-2.25);
	\end{scope}
	\end{tikzpicture}}
    }
	\subfloat[$N=2$]{
	\scalebox{1}{
	\begin{tikzpicture}
	\begin{scope}[spy using outlines={rectangle,blue,magnification=2,size=1.5cm}]
	\node [name=c]{	\includegraphics[height=3cm]{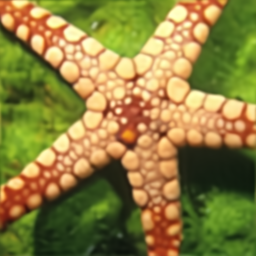}};
	\spy on (0.8,0.1) in node [name=c1]  at (0.75,-2.25);
	\spy on (-0.8,-0.1) in node [name=c1]  at (-0.75,-2.25);
	\end{scope}
	\end{tikzpicture}}
    }
	\subfloat[$N=4$]{
	\scalebox{1}{
	\begin{tikzpicture}
	\begin{scope}[spy using outlines={rectangle,blue,magnification=2,size=1.5cm}]
	\node [name=c]{	\includegraphics[height=3cm]{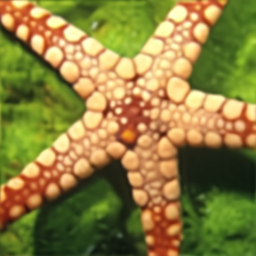}};
	\spy on (0.8,0.1) in node [name=c1]  at (0.75,-2.25);
	\spy on (-0.8,-0.1) in node [name=c1]  at (-0.75,-2.25);
	\end{scope}
	\end{tikzpicture}}
    }
    \caption{Reconstructions provided by GS-PnP, Block-PHILA FB-BB-I (second row) and Block-PHILA GD-C-I (third row) for the super-resolution problem. 
    }
    \label{fig:sr_results}
\end{figure}

\subsection{{Large scale scenario}}\label{sec:largescale}
\textcolor{black}{The experiments in the previous sections involved a least-squares data fidelity term and images of moderate size. In this section, we address a test problem that departs from both features, namely the restoration of high-resolution images corrupted by blur and Cauchy noise. This setting is of interest for two distinct reasons: it shows that the proposed framework is not confined to the Gaussian case, and it places the algorithm in a setting where the block-coordinate formulation is not merely convenient but necessary, as previously highlighted in Section \ref{subsec:memory_consumption}.}

{We recall that Cauchy noise is a heavy-tailed, impulsive noise which arises in applications such as synthetic aperture radar, sonar and low-frequency atmospheric imaging
\cite{Sciacchitano-et-al-2015}. This type of noise is characterized by a scale parameter $\gamma$, which determines its spread: higher values of $\gamma$ correspond to more extreme outliers due to the heavier tails. In this setting, it a standard practice to clip the measured data in $[0,1]$.}

{We solve problem \eqref{eq:GS_PNP_problem} where the data fidelity term for the Cauchy noise is defined as in \eqref{eq:cauchy_fidelity}. As already observed in \ref{remark_assumption1}, this term is continuously differentiable but \emph{not} convex along the coordinates. For this reason, we only compare the GD versions of Block-PHILA with an adapted version of the BCRED \cite{Sun-etal-2020}. In these experiments, we select a crop of size $1024\times 1024$ from an image of stars\footnote{\url{https://www.britannica.com/science/star-astronomy}}. The simulated measurement is generated considering the convolution with a Gaussian blur kernel
of size $25 \times 25$ with standard deviation $1.6$ (as in Section \ref{sec:image_deblurring} and further corruption with Cauchy noise having $\gamma=0.01$. The hyperparameters of the methods are selected as in Section \ref{sec:image_deblurring}, except for $\varrho=100$, $\sigma=0.09$ and the initial $\alpha=1e-5$. The number of blocks is set to $N=16$; we emphasize that the large size of the image prevents execution in a single-block configuration.  Full-image baselines such as DPIR and GS-PnP cannot be employed in this large-scale scenario, as their memory requirements exceed the limits of our resource-constrained architecture.}

{In Figure \ref{fig:cauchy_graph}, we report PSNR and the values of the objective function with respect to the number of iterations.  The GD-C is very slow compared to the other alternatives. It is evident that the use of the inertial term improves the performance of these approaches. It is evident that the inclusion of the inertial term significantly improves the performance of these approaches. Moreover, the BB adaptive stepsize selection enhances convergence by dynamically adapting the steplength to the local geometry of the objective function.} 

\begin{figure}
    \centering
    \begin{subfigure}[b]{0.45\textwidth}
        \centering
        \includegraphics[width=\textwidth]{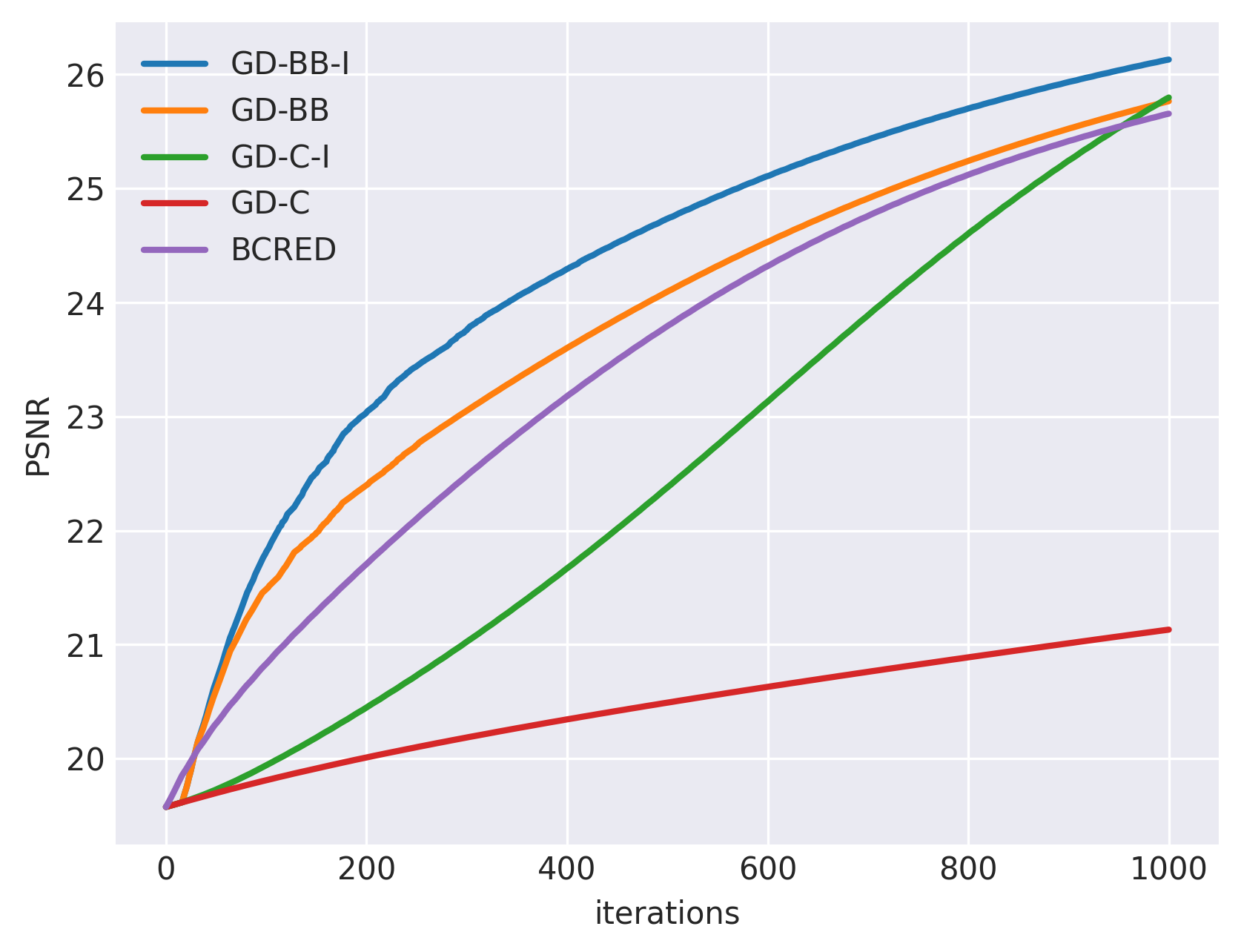}
    \end{subfigure}\hfill
    \begin{subfigure}[b]{0.45\textwidth}
        \centering
        \includegraphics[width=\textwidth]{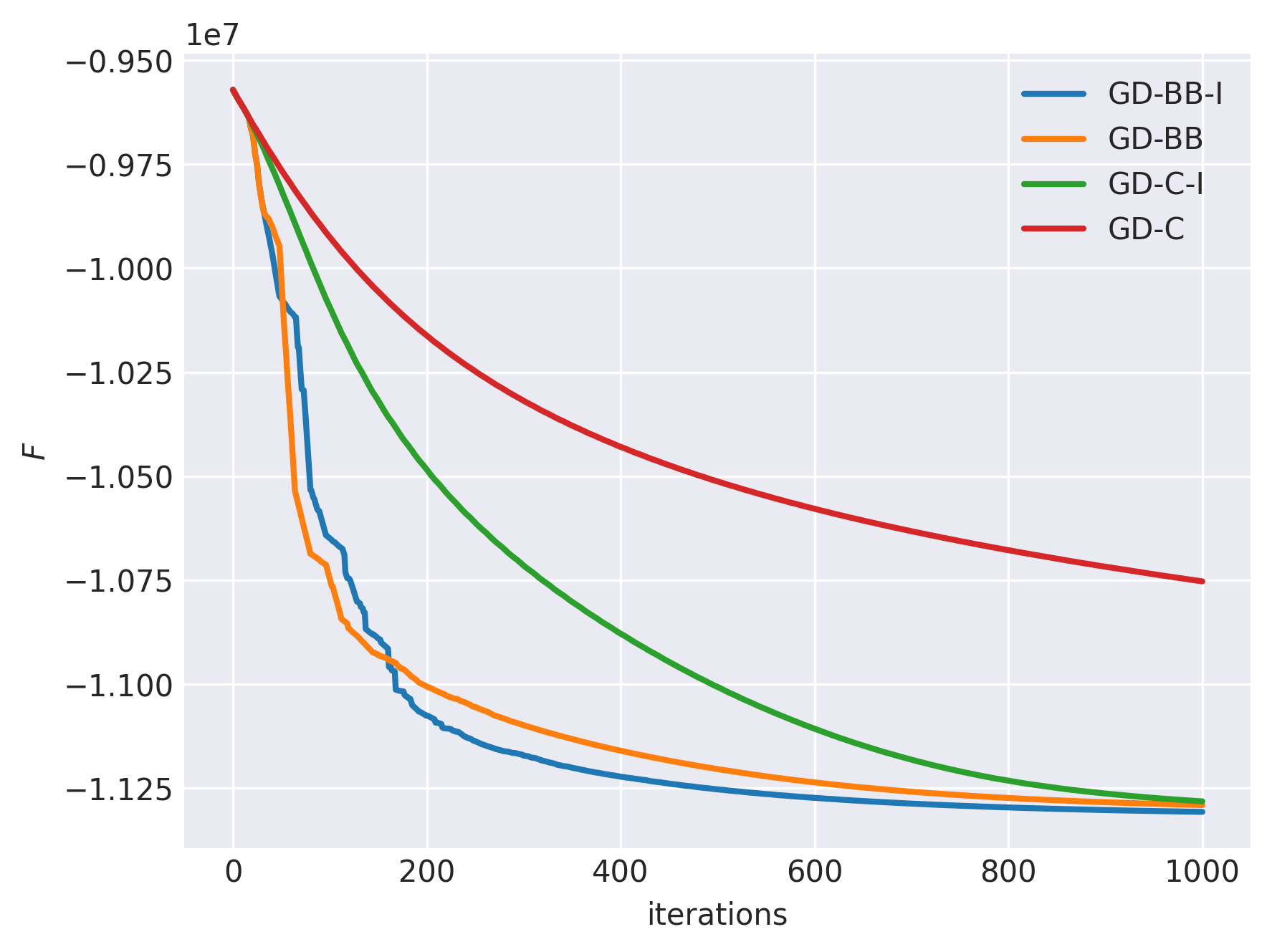}
    \end{subfigure}\hfill
    \caption{Results achieved solving the image deblurring problem on \textit{star} with Cauchy noise by the GD variants of Block-PHILA considering $N=16$ as number of blocks. For the PSNR the results of BCRED (adapted for the Cauchy setting) are also included.}
    \label{fig:cauchy_graph}
\end{figure}

{In Figure \ref{fig:cauchy_results}, we show the images reconstructed by BCRED and GD-BB-I. As can be observed, both methods achieve comparable visual quality and are equally effective in suppressing Cauchy noise artifacts. However, GD-BB-I  achieves a higher final PSNR, as shown in Figure \ref{fig:cauchy_graph}.}

\begin{figure}
	\centering
	\subfloat[Ground Truth]{
	   \scalebox{1}{
    	\begin{tikzpicture}
    	\begin{scope}[spy using outlines={rectangle,blue,magnification=5,size=3cm}]
    	\node [name=c]{	\includegraphics[height=3cm]{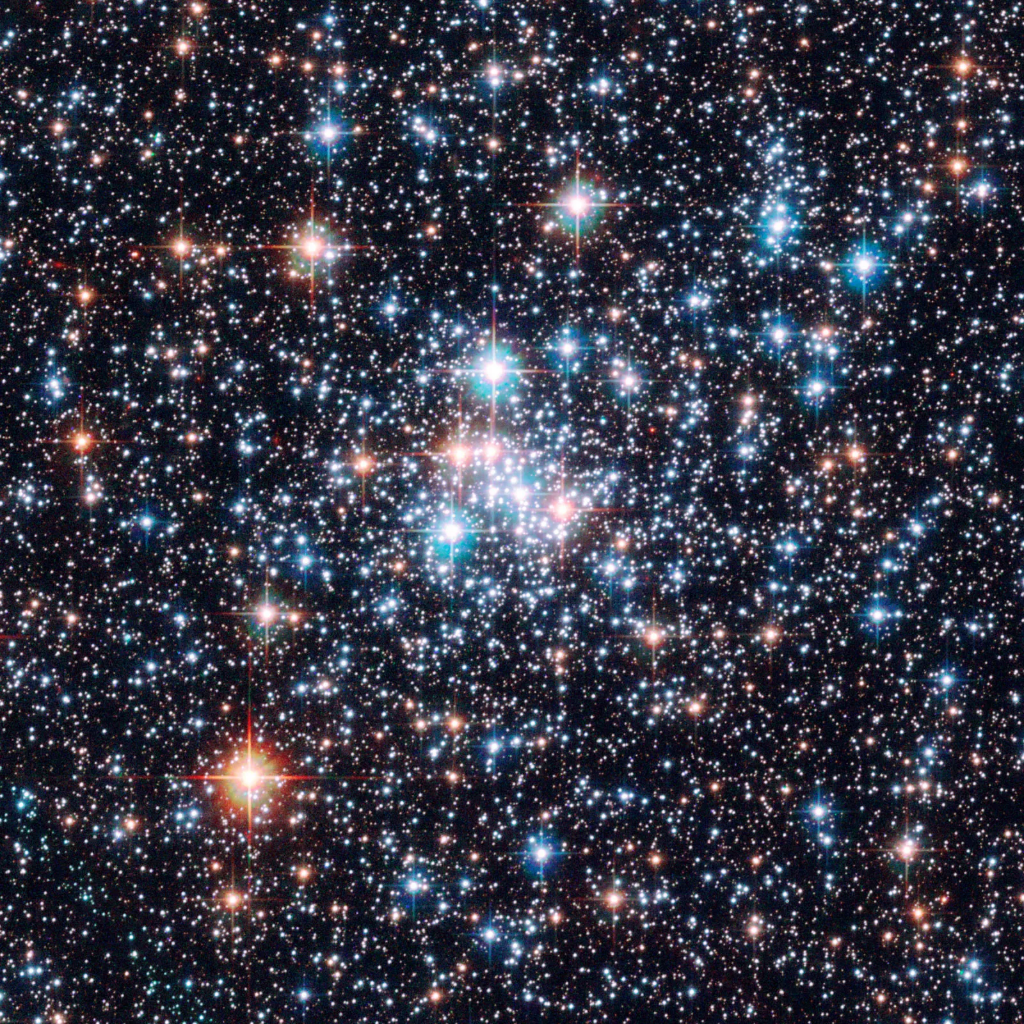}};
    	\spy on (0.7,0.1) in node [name=c1]  at (0.,-3.25);
    	\end{scope}
    	\end{tikzpicture}}
        }
	\subfloat[Corrupted data]{
	\scalebox{1}{
	\begin{tikzpicture}
	\begin{scope}[spy using outlines={rectangle,blue,magnification=5,size=3cm}]
	\node [name=c]{	\includegraphics[height=3cm]{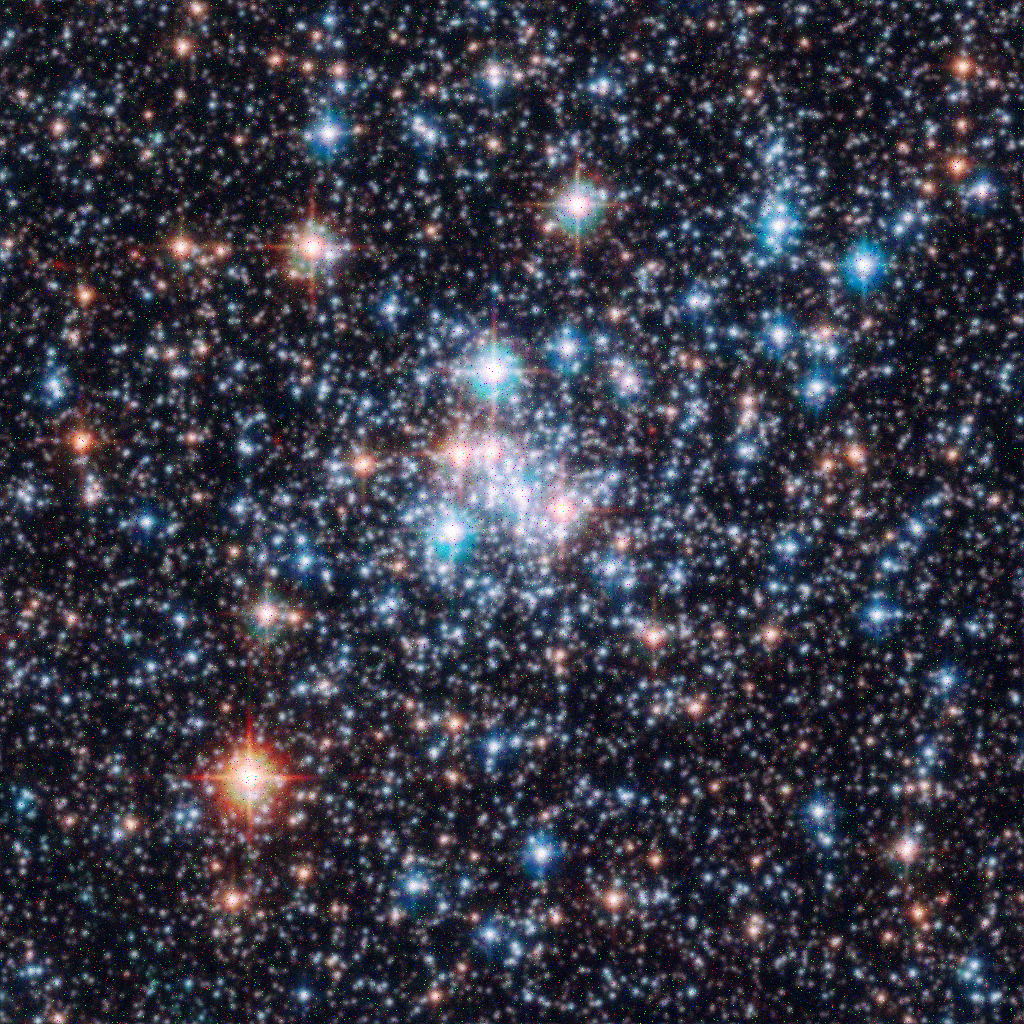}};
	\spy on (0.7,0.1) in node [name=c1]  at (0.,-3.25);
	\end{scope}
	\end{tikzpicture}}
    }
	\subfloat[BCRED]{
	\scalebox{1}{
	\begin{tikzpicture}
	\begin{scope}[spy using outlines={rectangle,blue,magnification=5,size=3cm}]
	\node [name=c]{	\includegraphics[height=3cm]{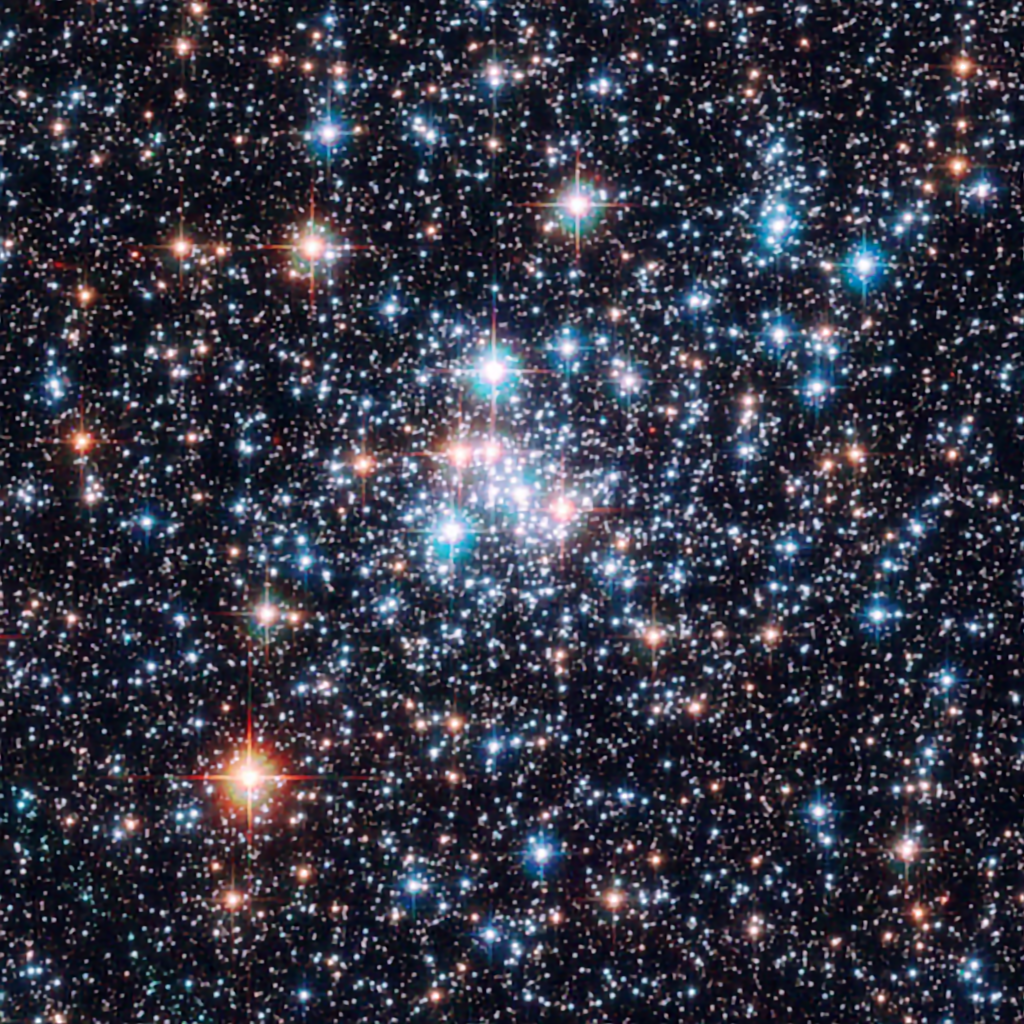}};
	\spy on (0.7,0.1) in node [name=c1]  at (0.,-3.25);
	\end{scope}
	\end{tikzpicture}}
    }
	\subfloat[GD-BB-I]{
	\scalebox{1}{
	\begin{tikzpicture}
	\begin{scope}[spy using outlines={rectangle,blue,magnification=5,size=3cm}]
	\node [name=c]{	\includegraphics[height=3cm]{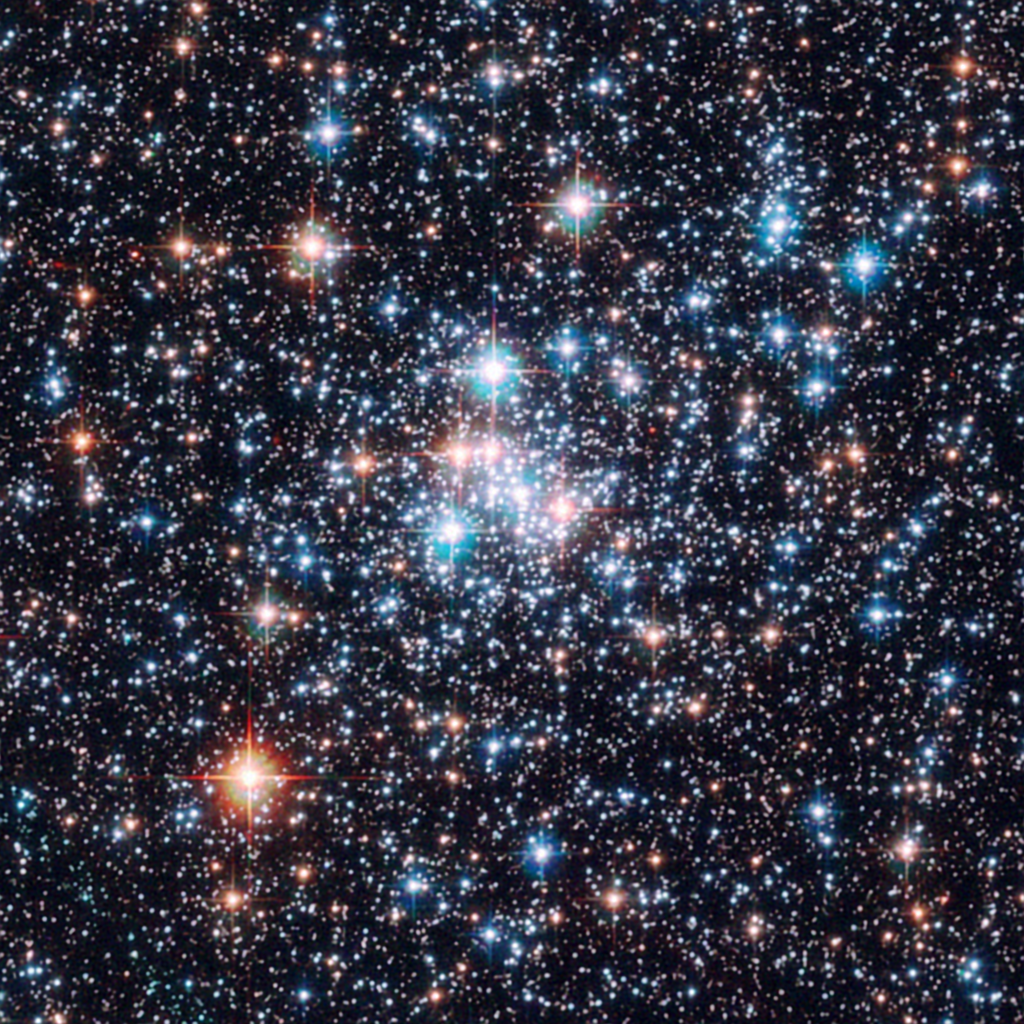}};
	\spy on (0.7,0.1) in node [name=c1]  at (0.,-3.25);
	\end{scope}
	\end{tikzpicture}}
    }
    \caption{Reconstructions provided by BCRED and Block-PHILA GD-BB-I for the deblurring problem problem with Cauchy noise.
    }
    \label{fig:cauchy_results}
\end{figure}

\section{Conclusions}\label{sec:conclusions}
In this paper we introduced a general block-coordinate forward-backward framework for the solution of non-convex and non-separable optimization problems. Convergence of the sequence of the iterates and related rates have been proved within the Kurdyka-Łojasiewicz framework. Building upon this theoretical framework, we developed a novel block-coordinate PnP approach based on Gradient Step denoisers. By decomposing the image into contiguous patches and exploiting the structural properties of convolutional neural networks, the suggested approach enables an efficient block-wise computation of the term related to the Gradient Step denoiser. Numerical experiments on imaging tasks, such as deblurring and super-resolution, show that the proposed block-coordinate PnP algorithm attains state-of-the-art reconstruction performance while significantly lowering GPU memory usage. Future work may concern the extension of our proposed block-coordinate forward-backward framework to nonsmooth objective functions, as well as the study of novel PnP line-search based approaches where the Gradient-Step denoiser replaces the proximal operator, rather than the gradient step, in the iterative procedure.

\section*{Acknowledgments}
All authors are members of the Gruppo Nazionale per il Calcolo Scientifico (GNCS) of the Italian Istituto Nazionale di Alta Matematica (INdAM), which is kindly acknowledged.\\
All authors are partially supported by  the Italian MUR through the PRIN 2022 PNRR Project “Advanced optimization METhods for automated central veIn Sign detection in multiple sclerosis from magneTic resonAnce imaging (AMETISTA)”, project code: P2022J9SNP (CUP  E53D23017980001), under the National Recovery and Resilience Plan (PNRR), Italy, Mission 04 Component 2 Investment 1.1   funded by the European Commission - NextGeneration EU programme. 

The work of S. Rebegoldi and A. Sebastiani is partially supported by (GNCS-INdAM) within the project “Metodi data-driven convergenti per l'imaging biomedicale” CUP E53C25002010001.

A. Sebastiani is supported by the funding received from the European Research Council (ERC) Starting project MALIN under the European Union's Horizon Europe programme (grant 101117133). 
This work represents only the view of the authors. The European Commission and the other organizations are not responsible for any use that may be made of the information it contains.

\bibliographystyle{siam}
\bibliography{biblio}
\appendix

\section{Proof of Lemma \ref{lem:attouch_modified}}\label{app:proof_of_lemma}
The following proof is obtained by slightly modifying the arguments of \cite[Theorem 2]{attouch20009proximal}.

{\bf Case $\theta\in (\frac{1}{2},1)$.} Let 
$R\in (1,+\infty)$. We distinguish between two cases.
\begin{itemize}
    \item Assume first that 
    $\Delta_{k+1}^{-2\theta}\leq R\Delta_{k-1}^{-2\theta}$.
    Then, from \eqref{eq:crucial_for_rate}, we get the following chain of inequalities:
\begin{align*}
    1&\leq \tilde{c}(\Delta_{k-1}-\Delta_{k+1})\Delta_{k+1}^{-2\theta}\\
    &\leq R\tilde{c}(\Delta_{k-1}-\Delta_{k+1})\Delta_{k-1}^{-2\theta}\\
    &\leq R \tilde{c} \int_{\Delta_{k+1}}^{\Delta_{k-1}}s^{-2\theta}ds\\
    &=\frac{R\tilde{c}}{1-2\theta}\left(\Delta_{k-1}^{1-2\theta}-\Delta_{k+1}^{1-2\theta}\right),
\end{align*}
where the third inequality follows from $\Delta_{k-1}^{-2\theta}\leq s^{-2\theta}$ 
for all $s\in[\Delta_{k+1},\Delta_{k-1}]$ (the function $s\rightarrow s^{-2\theta}$ is monotone decreasing) and the monotonicity of the integral. Setting
$$
\mu = \frac{2\theta - 1}{R\tilde{c}}>0, \quad \nu=1-2\theta<0,
$$ 
one obtains
\begin{equation}\label{eq:crucial_1}
    1\leq -\frac{1}{\mu}(\Delta_{k-1}^{\nu}-\Delta_{k+1}^\nu) \quad \Leftrightarrow \quad 0< \mu \leq \Delta_{k+1}^{\nu}-\Delta_{k-1}^\nu.
\end{equation}
\item Assume now that $\Delta_{k+1}^{-2\theta}>R\Delta_{k-1}^{-2\theta}$. 
Then
\begin{align*}
   \Delta_{k+1}^{-2\theta}> R \Delta_{k-1}^{-2\theta}\quad
    &\Leftrightarrow \quad \Delta_{k-1}^{2\theta}> R \Delta_{k+1}^{2\theta}\\
    &\Leftrightarrow \quad R^{-\frac{1}{2\theta}} \Delta_{k-1} > \Delta_{k+1}\\
    &\Leftrightarrow \quad R^{-\frac{1-2\theta}{2\theta}}\Delta_{k-1}^\nu < \Delta_{k+1}^\nu\\
    &\Leftrightarrow \quad (R^{\frac{2\theta-1}{2\theta}}-1) \Delta_{k-1}^\nu < \Delta_{k+1}^\nu-\Delta_{k-1}^\nu,
\end{align*}
where the direction of the third inequality changes due to $\nu < 0$. Since $R^{\frac{2\theta-1}{2\theta}}-1>0$ and $\lim\limits_{k\rightarrow \infty}\Delta_k = 0^+$, there exists $\tilde{\mu}>0$ such that $(R^{\frac{2\theta-1}{2\theta}}-1)\Delta_{k-1}^\nu>\tilde{\mu}$ for all $k\in \N$ for which $\Delta_{k+1}^{-2\theta}>R\Delta_{k-1}^{-2\theta}$. 
Therefore we obtain
\begin{equation}\label{eq:crucial_2}
    0<\tilde{\mu}\leq \Delta_{k+1}^{\nu}-\Delta_{k-1}^\nu.
\end{equation}
\end{itemize}
By setting $\hat{\mu} = \min \{\mu,\tilde{\mu}\}$ and combining \eqref{eq:crucial_1} and \eqref{eq:crucial_2}, one gets
\begin{equation*}
    \Delta_{k+1}^\nu-\Delta_{k-1}^\nu\geq \hat{\mu}>0, \quad \forall \ k\in\N.
\end{equation*}
By summing the above inequality for $k = 1,\ldots,K$, we obtain
\begin{align*}
    \sum_{k=1}^K(\Delta_{k+1}^\nu-\Delta_{k-1}^\nu)\geq \hat{\mu} K \quad &\Leftrightarrow \quad
    \sum_{k=1}^K(\Delta_{k+1}^\nu-\Delta_{k}^\nu)+\sum_{k=1}^K(\Delta_{k}^\nu-\Delta_{k-1}^\nu)\geq \hat{\mu} K\\
    &\Leftrightarrow \quad \Delta_{K+1}^{\nu}-\Delta_1^\nu+\Delta_K^{\nu}-\Delta_0^\nu\geq \hat{\mu}K\\
    &\Leftrightarrow \quad \Delta_{K+1}^{\nu}+\Delta_K^{\nu}\geq \Delta_0^\nu+\Delta_1^\nu+\hat{\mu}K,
\end{align*}
and since $\Delta_{K+1}\leq \Delta_K$ (being $\{\Delta_k\}_{k\in\N}$ monotone non-increasing), we finally obtain
\begin{equation*}
    \Delta_{K+1}^{\nu}\geq \frac{1}{2}(\Delta_0^\nu+\Delta_1^\nu+\hat{\mu}K),
\end{equation*}
from which we conclude that there exists a constant $C>0$ such that
\begin{equation*}
    \Delta_{K+1}\leq C K^{-\frac{1}{2\theta-1}}, \quad \forall \ K\in\N,
\end{equation*}
and item (i) follows.

{\bf Case $\theta\in (0, \frac{1}{2}$).} Since in this case $2\theta\leq 1$ and $\lim_{k\rightarrow \infty}\Delta_k = 0$, we have $\Delta_{k+1}\leq \Delta_{k+1}^{2\theta}$ for all sufficiently large $k\in\N$, and from \eqref{eq:crucial_for_rate} we get $\Delta_{k+1}\leq \tilde{c}(\Delta_{k-1}-\Delta_{k+1})$, or equivalently
\begin{equation*}
\Delta_{k+1}\leq \left(\frac{\tilde{c}}{1+\tilde{c}}\right)\Delta_{k-1}.
\end{equation*}
If $k$ is odd, the iterative application of the previous inequality entails
\begin{equation*}
\Delta_{k+1}\leq \left(\frac{\tilde{c}}{1+\tilde{c}}\right)^{\frac{k+1}{2}}\Delta_{0}.
\end{equation*}
If $k$ is even, the same reasoning leads to
\begin{equation*}
\Delta_{k+1}\leq \left(\frac{\tilde{c}}{1+\tilde{c}}\right)^{\frac{k}{2}}\Delta_{1}.    
\end{equation*}
Therefore, for all $k\in\N$, we have
\begin{equation*}
\Delta_{k+1}\leq C\left(\frac{\tilde{c}}{1+\tilde{c}}\right)^\frac{k}{2}, \quad C =\max\left\{\Delta_0\sqrt{\frac{\tilde{c}}{1+\tilde{c}}},\Delta_1\right\},
\end{equation*}
and the proof of item (ii) is complete.

\section{Computation of the proximal operator of a least squares functional}\label{app:prox_computation_exact}
Given the data fidelity term $\mathcal{D}(x) = \frac{1}{2}\|Ax - b\|^2$, where $A \in \mathbb{R}^{n \times n}$ and $x, y \in \mathbb{R}^n$, we recall the closed-form expression of $\prox_{\alpha \mathcal{D}}(z)$, with $z \in \mathbb{R}^n$, in the case where $A$ represents either a blurring or a super-resolution operator.

If $A=H$ is a convolution operator with circular boundary condition, it can be expressed as $H = F^{*}\Lambda F$,
where $F$ denotes the orthogonal discrete Fourier transform matrix, 
$F^{*}$ its inverse, and $\Lambda$ is diagonal. As a consequence, the computation of the corresponding proximal operator reduces to an element-wise inversion in the Fourier domain:
$$
\prox_{\alpha \mathcal{D}}(z) 
= F^{*}\,(I_n + \alpha \Lambda^{*}\Lambda)^{-1}F\,(\alpha H^{T}b + z).
$$

If $A = SH$ where $S$ is the standard $s$-fold downsampling matrix of size $m\times n$ and $n = s^2\times m$, and $H\in\mathbb{R}^{n\times n}$ is a convolution operator with circular boundary condition, then, in \cite{Zhao-2016}, the authors provide the closed-form computation of the proximal map for the data-fidelity $\mathcal{D}(x)$:
$$
\prox_{\alpha \mathcal{D}}(z) =  \hat{z}_{\alpha} - \frac{1}{s^2}F^*\underline{\Lambda}^*\left(I_m+\frac{\alpha}{s^2}\underline{\Lambda}\underline{\Lambda}^*\right)^{-1}\underline{\Lambda}F\hat{z}_{\alpha},
$$
where $\hat{z}_{\alpha} = \alpha H^TS^Tb+z$ and $\underline{\Lambda} = [\Lambda_1, \dots, \Lambda_{s^2}]\in\mathbb{R}^{m\times n}$, with $\Lambda = diag(\Lambda_1,\dots,\Lambda_{s^2})$ a blockdiagonal decomposition according to a $s\times s$ paving of the Fourier domain. We remark that $I_m+\frac{\alpha}{s^2}\underline{\Lambda}\underline{\Lambda}^*$ is a $m\times m$ diagonal matrix and its inverse is computed element-wise.
\section{Inexact computation of the proximal operator of a least squares functional restricted to a block of coordinates}\label{app:prox_computation}
To obtain a point $\tilde{y}_k$ as defined in \textsc{Step 3} of Algorithm \ref{alg:block-VMILA} when $\phi(x) = \frac{1}{2}\|Ax-b\|^2$ and $N\geq 1$, we first make the following remark. Hereafter, we drop the iteration index $k$ for simplicity. 
\begin{remark}
Given $x\in\mathbb{R}^n$, $i\in\{1,\dots,N\}$ and $U_j\in\R^{n\times n_j}$, $j=1,\dots,N$, note that, in general, 
\begin{equation*}
    \phi_i^x(z) = \phi\left(\begin{array}{c}
         U_1^Tx\\
         \vdots\\
         U_{i-1}^Tx\\
         z\\
         U_{i+1}^Tx\\
         \vdots\\
         U_{N}^Tx
    \end{array}\right).
\end{equation*}
Furthermore, if $\phi(x)=\frac{1}{2}\|Ax-b\|^2$, then
\begin{align*}
   \phi_i^x(z) & = \frac{1}{2}\left\|A\left(
         (U_1^Tx)^T, \cdots ,(U_{i-1}^Tx)^T,z,(U_{i+1}^Tx)^T, \cdots , 
         (U_{N}^Tx)^T\right)^T -b\right\|^2\\
         & = \frac{1}{2}\left\|AU_i z - b + A\left(
         (U_1^Tx)^T, \cdots ,(U_{i-1}^Tx)^T,0,(U_{i+1}^Ti)^T, \cdots , 
         (U_{N}^Tx)^T\right)^T\right\|^2\\
         &=\frac{1}{2}\|AU_i z - b_{i}^x\|^2,
\end{align*}
where $b_{i}^x = b - A\left(
         (U_1^Tx)^T, \cdots ,(U_{i-1}^Tx)^T,0,(U_{i+1}^Tx)^T, \cdots , 
         (U_{N}^Tx)^T\right)^T$.
\end{remark}

In view of this remark, we are interested in the inexact computation of the following proximal point
\begin{align}\label{eq:proximal_point2}
\prox^{{D}}_{\alpha \phi_i^x}(\bar{x}_i) &= \underset{y\in\R^{n_i}}{\operatorname{argmin}}\ \phi_i^x(y)+\frac{1}{2\alpha}\|y-\bar{x}_i\|_{{D}}^2
\end{align}
where {$D\in\R^{n_i}$ is a symmetric positive definite matrix}, $\bar{x}_i = U_i^T(x+\beta(x-w))-\alpha {D^{-1}}U_i^T\nabla f(x)$, being $x,w\in\R^n$, $\alpha,\beta \in\mathbb{R}^+$ and $\phi_i^{x}:\R^n\rightarrow \R$ is the least squares function given by
\begin{equation}\label{eq:fi_2}
    \phi_i^{x}(y) = \frac{1}{2}\|AU_i y-b_i^x\|^2.
\end{equation}
The following procedure is an application of the more general framework discussed in \cite[Appendix A]{Bonettini-Prato-Rebegoldi-2024} for computing inexact proximal points satisfying condition \eqref{eq:tilde_yki}. Let us rewrite function \eqref{eq:fi_2} as 
\begin{equation*}
    {\phi_i^{x}(y)} = \omega_x(M_iy), 
\end{equation*}
with $\omega_x:\R^n\rightarrow \R$, $\omega_x(\cdot) = \frac{1}{2}\|\cdot-b_i^x\|^2$, and $M_i = AU_i\in\R^{n\times n_i}$. Note that $\omega_x(t) = \rho(t-b_i^x)$, where $\rho(\cdot) = \frac{1}{2}\|\cdot\|^2$. By using a well-known calculus rule holding for conjugate functions \cite[Proposition 1.3.1(v)]{Hiriart-1993}, we conclude that
\begin{equation*}    \omega^*_x:\R^{n}\rightarrow \R, \quad \omega^*_x(s) = \rho^*(s)+s^Tb_i^x = \frac{1}{2}\|s\|^2+s^Tb_i^x.  
\end{equation*}
The primal problem associated to the computation of \eqref{eq:proximal_point2} writes as
\begin{align}
    &\underset{y\in\R^{n_i}}{\min} \ h(y) \equiv {\phi_i^{x}(y)}-{\phi_i^{x}(U_i^Tx)} +\frac{1}{2\alpha}\|y-U_i^Tx\|^2_{D}+\langle U_i^T\nabla f(x)-\frac{\beta}{\alpha}{D}U_i^T(x-w), y- U_i^Tx\rangle\nonumber\\
    & = \omega_x(M_iy)+\frac{1}{2\alpha}\|y-\bar{x_i}\|^2_D-\phi_i^{x}(U_i^Tx)-\frac{\alpha}{2}\left\|U_i^T\nabla f(x)-\frac{\beta}{\alpha}D U_i^T(x-w)\right\|^2_{D^{-1}}. \label{eq:primal_problem_2}
\end{align}
Note that the primal function $h$ has the same form as the function $h_k$ defined at each iteration of Algorithm \ref{alg:block-VMILA}. On the other hand, by means of \eqref{eq:primal_problem_2}, we can write the dual problem as follows
\begin{align}
  &\underset{v\in\R^{n}}{\max} \ \psi(v)\equiv -\omega^*_x(v) - \frac{1}{2\alpha}\|\bar{x}_i-\alpha D^{-1}M_i^* v \|^2_D+\frac{1}{2\alpha}\|\bar{x}_i\|^2_D-\phi_i^x(U_i^Tx)-\frac{\alpha}{2}\|U_i^T\nabla f(x)-\frac{\beta}{\alpha}D U_i^T(x-w)\|^2_{D^{-1}}\nonumber\\
  &=-\frac{1}{2}\|v\|^2-v^Tb_i^x - \frac{1}{2\alpha}\|\bar{x}_i-\alpha D^{-1}U_i^TA^T v \|^2_D+\underbrace{\frac{1}{2\alpha}\|\bar{x}_i\|^2_D-\phi_i^x(U_i^Tx)-\frac{\alpha}{2}\|U_i^T\nabla f(x)-\frac{\beta}{\alpha} D U_i^T(x-w)\|^2_{D^{-1}}}_{=constant}.\label{eq:dual_problem}
\end{align}
{Note that the gradient of $\psi$ at $v$ is given by
\begin{align*}
\nabla \psi(v) &= -v-b_i^x-AU_iD^{-\frac{1}{2}}(\alpha D^{-\frac{1}{2}}U_i^TA^Tv-D^{\frac{1}{2}}\bar{x}_i) \\
&= -v-b_i^x-\alpha A U_i\left(D^{-1}U_i^TA^Tv-\frac{1}{\alpha}\bar{x}_i\right)\\
&=-\left(I+\alpha A U_iD^{-1}U_i^TA^T\right)v+AU_i\bar{x}_i-b_i^x.
\end{align*}
}

By definition of the Fenchel dual, we have
$$
h(y)\geq \psi(v), \quad \forall y\in\mathbb{R}^{n_i}, \forall v \in\mathbb{R}^n. 
$$
In particular, the previous inequality holds for $y = \hat{y} = \prox_{\alpha\phi_i^x}(\bar{x}_i)$ and for any $v\in\mathbb{R}^n$. Therefore if $(\tilde{y},v)$ is a primal dual pair satisfying
$$
h(\tilde{y})\leq \frac{2}{2+\tau} \psi(v),
$$
then $\tilde{y}$ is an inexact proximal gradient point in the sense of \eqref{eq:tilde_yki}. According to \cite[Proposition 19]{Bonettini-Prato-Rebegoldi-2024}, the existence of such a pair is guaranteed, and it can be computed in practice. For completeness, we restate the proposition in the context of our framework.
\begin{prop}\label{prop:appendix}
Given $\phi_i^x$ defined as is \eqref{eq:fi_2}, let $\{v_\ell\}_{\ell\in\N}\subset \mathbb{R}^n$ be a sequence such that 
\begin{equation*}
        \lim_{\ell\rightarrow +\infty}v_\ell = \underset{v\in \R^n}{\operatorname{argmax}} \ \psi(v).
    \end{equation*}
    By defining the corresponding primal sequence $\{\tilde{y}_{\ell}\}_{\ell\in\N}\subset \mathbb{R}^{n_i}$ as
    \begin{equation*}
        \tilde{y}_{\ell} = \bar{x}_i-\alpha D^{-1} M_i^* v_\ell = \bar{x}_i-\alpha D^{-1}U_i^TA^Tv_\ell, \quad \forall \ \ell\in\N,
    \end{equation*}
    it holds that
    $$
    \lim_{\ell\rightarrow +\infty} \psi(v_\ell) = \lim_{\ell\rightarrow +\infty} h(\tilde{y}_\ell) = h(\prox_{\alpha \phi_i^x}(\bar{x}_i)).
    $$
    Therefore the inequality 
    $$
    h(\tilde{y}_\ell)\leq \frac{2}{2+\tau} \psi(v_\ell)
    $$
    holds for all sufficiently large $\ell$, for any given $\tau>0$.
\end{prop}

The dual sequence $\{v_\ell\}_{\ell\in\mathbb{N}}$ can be obtained by applying an iterative gradient-based optimization scheme to the dual problem, which is stopped at the first iteration $\ell^*$ satisfying
    \begin{equation*}
        h(\tilde{y}_{\ell^*})\leq \frac{2}{2+\tau}\psi(v_{\ell^*}).
    \end{equation*}   
 According to Proposition \ref{prop:appendix},   $\tilde{y}$ can be fixed as $\tilde{y}_{\ell^*}$.

The solution of the dual problem \eqref{eq:dual_problem} can be derived by means of the first optimality condition $\nabla\psi(v)=0$, namely

\begin{equation}\label{eq:opt_condition_dual}
    -\left(I+\alpha A U_iD^{-1}U_i^TA^T\right)v+AU_i\bar{x}_i-b_i^x = 0.
\end{equation}

Which reduces to finding a solution of a linear system. We distinguish two cases: the fist one, when the scaling matrix is the multiple of the identity, and the second one, when the scaling matrix is non-trivial.

\paragraph{Identity case} The linear system, \eqref{eq:opt_condition_dual} corresponds to 
\begin{align}
     \left(I + \alpha AU_iU_i^TA^T \right)v =  A U_i\bar{x}_i -b_i^x, 
\end{align}
whose solution can be computed iteratively by using a variant of the Gauss–Seidel method, observing that the matrix on the left can also be seen as

\begin{align}
    \left(I + \alpha AU_iU_i^TA^T \right)=\left(I + \alpha AA^T - \alpha AU_{i^C}U_{i^C}^TA^T \right).
\end{align}

So the iteration of the subroutine for the inexact computation of the restricted proximal operators reads as

\begin{equation}
    v_{\ell+1} = (I + \alpha AA^T)^{-1}\left(A^TU_i\bar{x}_i -b_i^x  + \alpha AU_{i^C}U_{i^C}^TA^Tv_{\ell}\right),
\end{equation}
where the inverse is computed in closed form, similarly to \ref{app:prox_computation_exact} and $U_{i^C}U_{i^C}^T=I-\Tilde{U}_i$.

\paragraph{Non-uniform} The linear system, \eqref{eq:opt_condition_dual} can be approximately solved by means of a conjugate gradient iteration until the inexactness condition is satisfied.
\end{document}